\documentclass[a4paper,11pt,leqno]{article}
\usepackage[nofancy]{svninfo}
\svnInfo $Id: dual.tex 1213 2019-10-07 14:23:37Z zheng $
\usepackage[hmargin=2cm,vmargin=2.5cm]{geometry}
\usepackage[all]{xy}
\usepackage{colonequals}
\usepackage{stackrel}
\usepackage{amssymb}
\usepackage{amsxtra}
\usepackage{amsthm}
\usepackage[T1]{fontenc}
\usepackage{lmodern}
\usepackage[pagebackref,breaklinks]{hyperref}
\usepackage[abbrev,shortalphabetic]{amsrefs}
\usepackage[shortlabels]{enumitem}
\usepackage{tocloft}
\setlist[enumerate,1]{(1)}
\setlist{nosep}
\theoremstyle{definition}
\newtheorem{definition}{Definition}[section]
\newtheorem{example}[definition]{Example}
\newtheorem{construction}[definition]{Construction}
\theoremstyle{plain}
\newtheorem{prop}[definition]{Proposition}
\newtheorem{lemma}[definition]{Lemma}
\newtheorem{theorem}[definition]{Theorem}
\newtheorem{cor}[definition]{Corollary}
\theoremstyle{remark}
\newtheorem{remark}[definition]{Remark}

\newcommand{\cHom}{\mathcal{H}\mathnormal{om}}

\newcommand{\fm}{\mathfrak{m}}

\newcommand{\Z}{\mathbb{Z}}
\newcommand{\A}{\mathbb{A}}
\newcommand{\PP}{\mathbb{P}}

\newcommand{\Hom}{\mathrm{Hom}}
\newcommand{\End}{\mathrm{End}}
\newcommand{\Tor}{\mathrm{Tor}}
\newcommand{\Ker}{\mathrm{Ker}}
\newcommand{\Coker}{\mathrm{Coker}}

\newcommand{\Spec}{\mathrm{Spec}}
\newcommand{\Shv}{\mathrm{Shv}}
\newcommand{\rSS}{\mathrm{SS}}
\newcommand{\op}{\mathrm{op}}

\newcommand{\red}{\mathrm{red}}
\newcommand{\cons}{\mathrm{cons}}
\newcommand{\BC}{\mathrm{BC}}
\newcommand{\Tr}{\mathrm{Tr}}

\newcommand{\can}{\mathrm{can}}
\newcommand{\var}{\mathrm{var}}
\newcommand{\Cone}{\mathrm{Cone}}
\newcommand{\cont}{\mathrm{cont}}
\newcommand{\Gal}{\mathrm{Gal}}
\newcommand{\id}{\mathrm{id}}
\newcommand{\ft}{\mathrm{ft}}
\newcommand{\cft}{{c\mathrm{ft}}}
\newcommand{\tr}{\mathrm{tr}}
\newcommand{\pt}{\mathrm{pt}}
\newcommand{\et}{\mathrm{et}}
\newcommand{\Et}{\mathrm{Et}}
\newcommand{\ev}{\mathrm{ev}}
\newcommand{\Co}{\mathrm{Co}}

\newcommand{\cC}{\mathcal{C}}
\newcommand{\cD}{\mathcal{D}}
\newcommand{\cF}{\mathcal{F}}
\newcommand{\cG}{\mathcal{G}}

\newcommand{\cO}{\mathcal{O}}
\newcommand{\cP}{\mathcal{P}}

\newcommand{\atimes}{\stackbin{\leftarrow}{\times}}
\newcommand{\ttimes}{\mathbin{\bar{\times}}}
\newcommand{\simto}{\xrightarrow{\sim}}
\DeclareMathOperator*{\colim}{colim}

\numberwithin{equation}{section}
\begin{document}
\title{Duality and nearby cycles over general bases}
\author{Qing Lu\thanks{School of Mathematical Sciences, Beijing Normal University, Beijing
100875, China; School of Mathematical Sciences, University of the Chinese
Academy of Sciences, Beijing 100049, China; email: \texttt{qlu@bnu.edu.cn}.
Partially supported by National Natural Science Foundation of China Grants
11371043, 11501541; China Scholarship Council; Fundamental Research Funds
for Central Universities of China.}\and Weizhe Zheng\thanks{Morningside
Center of Mathematics and Hua Loo-Keng Key Laboratory of Mathematics,
Academy of Mathematics and Systems Science, Chinese Academy of Sciences,
Beijing 100190, China; University of the Chinese Academy of Sciences,
Beijing 100049, China; email: \texttt{wzheng@math.ac.cn}. Partially
supported by National Natural Science Foundation of China Grants 11621061,
11688101, 11822110; National Center for Mathematics and Interdisciplinary
Sciences, Chinese Academy of Sciences.}
\thanks{Mathematics Subject Classification 2010: 14F20 (Primary); 18F10, 32S30
(Secondary).}}
\maketitle

\begin{abstract}
This paper studies the sliced nearby cycle functor and its commutation
with duality. Over a Henselian discrete valuation ring, we show that this
commutation holds, confirming a prediction of Deligne. As an application
we give a new proof of Beilinson's theorem that the vanishing cycle
functor commutes with duality up to twist. Over an excellent base scheme,
we show that the sliced nearby cycle functor commutes with duality up to
modification of the base. We deduce that duality preserves universal local
acyclicity over an excellent regular base. We also present Gabber's
theorem that local acyclicity implies universal local acyclicity over a
Noetherian base.
\end{abstract}

\tableofcontents

\section*{Introduction}
\subsection{Over a Henselian discrete valuation ring}
Let $S$ be the spectrum of a Henselian discrete valuation ring, of closed
point $s$ and generic point~$\eta$. Let $a\colon X\to S$ be a morphism of
schemes. Let $\Lambda$ be a Noetherian commutative ring such that
$m\Lambda=0$ for some integer $m$ invertible on $S$. We have the classical
nearby cycle functor \cite[XIII 2.1.1]{SGA7II}
\[R\Psi^s_\eta\colon D(X_\eta,\Lambda)\to D(X_s\ttimes_s \eta,\Lambda),\]
where $X_\eta\colonequals X\times_S \eta$ and $X_s\colonequals X\times_S s$.
The vanishing cycle functor $\Phi^s$ is a composition
\[D(X,\Lambda)\xrightarrow{R\Psi^s} D(X_s\ttimes_s S,\Lambda)\xrightarrow{L\Co} D(X_s\ttimes_s \eta,\Lambda),\]
where $L\Co (M)$ computes the cone of the specialization map $M_s\to M_\eta$
\eqref{e.treta1}. We call $R\Psi^s$ the sliced nearby cycle functor. Here
$\ttimes_s$ denotes fiber products of \'etale topoi. (In the case where
$\eta$ is a scheme over $s$, $X\ttimes_s \eta$ is typically not the \'etale
topos of $X\times_s \eta$.)  Unless otherwise indicated, we work in the
unbounded derived categories.

Assume $X$ separated and of finite type over $S$. Gabber proved that the
classical nearby cycle functor $R\Psi^s_\eta$, when restricted to $D_\cft$
(the full subcategory of $D^b_c$ spanned by complexes of finite
tor-amplitude, where $D^*_c$ denotes the full subcategory of $D^*$ spanned
by complexes with constructible cohomology sheaves), commutes with duality
\cite[Th\'eor\`eme 4.2]{IllusieAutour}. Our first result confirms Deligne's
prediction \cite{Deligne} that the same holds for the sliced nearby cycle
functor $R\Psi^s$. Let $D_X=R\cHom(-,Ra^!\Lambda_S)$, $D_{X_s\ttimes_s
S}=R\cHom(-,R(a_s\ttimes_s \id_S)^!\Lambda_S)$.

\begin{theorem}\label{t.S}\leavevmode
\begin{enumerate}
\item The canonical map $R\Psi^s D_XL\to D_{X_s\ttimes_s S} R\Psi^s L$
    \eqref{e.PsiD} is an isomorphism for $L\in D^-_c(X,\Lambda)$.
\item We have a natural isomorphism $D_{X_s\ttimes_s \eta}L\Co\simeq \tau
    L\Co D_{X_s\ttimes_s S}$ \eqref{e.nt}.
\end{enumerate}
\end{theorem}

Here $\tau$ (Definition \ref{d.Iwasawa}) is a functor that modifies only the
tame part by what we call, following Illusie, the Iwasawa twist. We
immediately deduce that the vanishing cycle functor commutes with duality up
to the Iwasawa twist, which is a theorem of Beilinson (at least when $s$ is
separably closed) \cite[2.3]{Beilinson}.

\begin{cor}[Beilinson]\label{c.Beilinson}
We have an isomorphism $D_{X_s\ttimes_s \eta}\Phi^s L\simeq \tau \Phi^s D_X
L$, functorial in $L\in D^-_c(X,\Lambda)$.
\end{cor}

Though the functors in Theorem \ref{t.S} (1) involve only the fiber product
$X_s\ttimes_s S$, the construction of the map and the proof uses the
\emph{oriented product} $X\atimes_S S$. Theorem \ref{t.S} (2) follows from
quasi-periodic adjunctions for the functor $L\Co$ \eqref{e.adjDj}.
Beilinson's proof of Corollary \ref{c.Beilinson} was rather different, using
his maximal extension functor $\Xi$.

\subsection{Over general bases}
For an arbitrary base scheme $S$ and any scheme $X$ over~$S$, Deligne
defined the \emph{vanishing topos} $X\atimes_S S$ and the nearby cycle
functor (\cite{Laumon}, \cite[XI Section 4]{ILO})
\[R\Psi\colon D(X,\Lambda)\to D(X\atimes_S S,\Lambda).\]
Recently Illusie proved a K\"unneth formula for $R\Psi$ \cite[Theorems 2.3,
A.3]{IZ}, generalizing Gabber's theorem over a Henselian discrete valuation
ring (\cite[Th\'eor\`eme 4.7]{IllusieAutour}, \cite[Lemma 5.1.1]{BB}). It
was used in Saito's work on characteristic cycles, notably in his proof of
the global index formula (\cite{Saito}, \cite{Saito2}).

Following a suggestion of Illusie, we study the commutation with duality of
nearby cycles over general bases. While there is typically no good duality
on the vanishing topos $X\atimes_S S$ as a whole (Remark \ref{r.nodual}),
there is a good duality on the \emph{slice} $X_s\atimes_S S\simeq
X_s\ttimes_s S_{(s)}$ for every point $s$ of $S$, under usual assumptions,
where $S_{(s)}$ denotes the Henselization of $S$ at $s$. We define the
sliced nearby cycle functor $R\Psi^s$ to be the composition
\[D(X,\Lambda)\xrightarrow{R\Psi} D(X\atimes_S S,\Lambda)\to D(X_s\ttimes_s S_{(s)},\Lambda)\]
of $R\Psi$ with the restriction functor. Our main result is that the sliced
nearby cycle functor commutes with duality up to modification of the base,
generalizing the excellent case of Theorem \ref{t.S} (1). More generally, we
have the following result, where $K_S$ is not assumed to be a dualizing
complex.

\begin{theorem}\label{t.Psisc}
Let $S$ be an excellent scheme and let $K_S\in D^b_c(S,\Lambda)$. Let
$a\colon X\to S$ be a separated morphism of schemes of finite type and let
$L\in D^b_c(X,\Lambda)$ such that $R\cHom(L,Ra^!K_S)\in D^b_c(X,\Lambda)$.
Then there exists a modification $S'\to S$ such that for every morphism
$T\to S'$ separated of finite type, and for every point $t\in T$, the
canonical map \eqref{e.A}
\[R\Psi^t D_{X_T} (L|_{X_T})\to D_{X_t\ttimes_t T_{(t)}}R\Psi^t
(L|_{X_T})
\]
is an isomorphism. Here $X_T\colonequals X\times_S T$, $X_t\colonequals
X\times_S t$, $D_{X_T}\colonequals R\cHom(-,K_{X_T})$, $D_{X_t\ttimes_t
T_{(t)}}\colonequals R\cHom(-,K_{X_t\ttimes_t T_{(t)}})$, $K_{X_T}$ and
$K_{X_t\ttimes_t T_{(t)}}$ are $!$-pullbacks of $K_S$.
\end{theorem}

Here a modification means a proper birational morphism. Excellent schemes
are assumed to be Noetherian.

Further restrictions of the nearby cycle functor to shreds and local
sections were previously studied by Orgogozo \cite[Section 6]{Org} and
Illusie \cite[Section 1]{IZ}, but these restrictions carry too little
information on the base and fibers, respectively, for an analogue of Theorem
\ref{t.Psisc} to hold (Remark \ref{r.fail}).

One ingredient of the proof of Theorem \ref{t.Psisc} is Orgogozo's theorem
that the nearby cycle functor commutes with base change after modification
of the base \cite[Th\'eor\`eme 2.1]{Org}. Since duality swaps pullback and
$!$-pullback, we are also lead to study the commutation of the sliced nearby
cycle functor with $!$-pullback. The proof of Theorem \ref{t.Psisc} relies
on both Orgogozo's theorem and an analogue (Theorem \ref{t.good}) thereof
for $!$-pullback.

As an application of Theorem \ref{t.Psisc}, we show that universal local
acyclicity over a regular excellent base is preserved by duality (Corollary
\ref{c.DLA}), answering a question of Illusie. Gabber gave a different proof
of the case of finite tor-amplitude of this corollary. He also showed that
over a Noetherian base, local acyclicity implies universal local acyclicity
(Corollary \ref{c.la}), answering a question of M.~Artin. We deduce that
weak singular support over a regular excellent base is compatible with
duality (Corollary \ref{c.ss}).

The results of this paper have applications to the perversity of nearby
cycles, which we hope to explore in a future article.

\paragraph*{Organization} In Section \ref{s.fp}, we define and study duality and other
operations on fiber products of topoi, on which the sliced nearby cycles
live. In Section~\ref{s.2}, after preliminaries on the Iwasawa twist, we
study adjunctions and duality for the functor $L\Co$ and prove Theorem
\ref{t.S} (2). In Section~\ref{s.2+}, we prove Theorem \ref{t.S} (1) on
duality and the sliced nearby cycle functor over Henselian discrete
valuation rings. In Section~\ref{s.3}, we study the sliced nearby cycle
functor over general bases and prove Theorem \ref{t.good}, the analogue of
Orgogozo's theorem for $!$-pullback, which is a key step toward the proof of
Theorem \ref{t.Psisc}. In Section~\ref{s.3+}, we prove Theorem \ref{t.Psisc}
and the applications to local acyclicity and singular support. In
Section~\ref{s.Gabber}, we present Gabber's results on local acyclicity.

Section \ref{s.3+} depends on Section \ref{s.3}. Sections \ref{s.2} through
\ref{s.3+} depend on Section \ref{s.fp}. For those interested only in the
case over a Henselian discrete valuation ring (Sections \ref{s.2} and
\ref{s.2+}), we recommend consulting Section \ref{s.fp} only when necessary.

\paragraph*{Acknowledgments} This paper grew out of suggestions from Luc Illusie, and
would not have existed without his insights. We thank him for hundreds of
comments on drafts of this paper. We thank Ofer Gabber for kindly allowing
us to present his results on local acyclicity in Section \ref{s.Gabber} and
for many crucial suggestions, including the removal of a finite
dimensionality assumption on the base (Remark \ref{r.Gabber}). We thank
Takeshi Saito and the referee for valuable comments. We thank Lei Fu and
Haoyu Hu for useful discussions. Part of this work was done during visits to
l'Institut des Hautes \'Etudes Scientifiques, the University of Hong Kong,
and Princeton University. We thank these institutions for hospitality and
support.

\section{Fiber products of topoi}\label{s.fp}
The sliced nearby cycles live on fiber products of topoi over the \'etale
topos of a point. The goal of this section is to define and study various
operations on fiber products of topoi, including the dualizing functor in
our main theorems.

\subsection{Over general bases}

In this subsection, we prove a proper base change theorem for fiber products
of topoi over general bases (Proposition \ref{p.Org}) and use it to
construct various operations on such fiber products. The construction
provides operations on slices of the vanishing topos used in our main
theorems. It also applies to the vanishing topos itself (Construction
\ref{c.o!}).

Let $\Lambda$ be a commutative ring. For any topos $S$, we write
$\Shv(S,\Lambda)$ for the category of sheaves of $\Lambda$-modules on $S$
and we write $D(S,\Lambda)$ for its derived category.

Let $X\to S$ and $Y\to S$ be morphisms of topoi. We refer to \cite[XI]{ILO}
for the constructions of the oriented product $X\atimes_S Y$ and the fiber
product $X\ttimes_S Y$. We adopt the notation $\ttimes$ for products and
fiber products of topoi to avoid confusion with fiber products and products
of schemes.

The following base change results for oriented products of topoi will be
used in Proposition \ref{p.prod} and Lemma \ref{l.Psi!}.

\begin{prop}\label{p.obc}
Let $X\xrightarrow{a} S\leftarrow Y\xleftarrow{g} Y'$ be coherent morphisms
\cite[VI D\'efinition 3.1]{SGA4} of coherent topoi \cite[VI D\'efinition
2.3]{SGA4}. Consider $\id_X\atimes_S g\colon X\atimes_S Y'\to X\atimes_S Y$.
For $L\in D^+(X\atimes_S Y',\Lambda)$, the stalk of $R(\id_X\atimes_S g)_*
L$ at any point $(x,y,\phi)$ of $X\atimes_S Y$ is isomorphic to
$R\Gamma(Y'_{(y)},c^*L)$, where $Y'_{(y)}=Y'\ttimes_Y Y_{(y)}$, $Y_{(y)}$ is
the localization of $Y$ at $y$, $c$ is the composite $Y'_{(y)}\to
Y'\xrightarrow{\sigma} X\atimes_S Y'$, and $\sigma$ is the canonical section
induced by $x$ \cite[XI 2.2]{ILO}.
\end{prop}

By \cite[XI Lemme 2.5]{ILO}, $\id_X\atimes_S g\colon X\atimes_S Y'\to
X\atimes_S Y$ is a coherent morphism of coherent topoi.

\begin{proof}
By limit arguments, the stalk in question can be identified with
$R\Gamma(T,L|_T)$, where $T=X_{(x)}\atimes_{S_{(a(x))}} Y'_{(y)}$. We have
$R\Gamma(T,L|_T)\simeq R\Gamma(Y'_{(y)}, Rp_{2*}(L|_T))$, where $p_2\colon
T\to Y'_{(y)}$ is the projection and $Rp_{2*}(L|_T)\simeq c^*L$ by \cite[XI
Proposition 2.3]{ILO}.
\end{proof}

\begin{cor}\label{c.obc}
Let $X'\xrightarrow{f} X\xrightarrow{a} S\xleftarrow{b} Y\xleftarrow{g} Y'$
be morphisms of coherent topoi with $a$, $af$, $b$, $g$ coherent. Then the
base change map
\[\alpha\colon (f\atimes_S \id_Y)^* R(\id_{X}\atimes_S g)_* \to R(\id_{X'}\atimes_S g)_* (f\atimes_S \id_{Y'})^*\]
associated to the square (Cartesian by \cite[XI Proposition 4.2]{ILO})
\[\xymatrix{X'\atimes_S Y'\ar[d]_{f\atimes_S \id_{Y'}}\ar[r]^{\id_{X'}\atimes_S g}\ar@{}[rd]|\Leftrightarrow & X'\atimes_S Y\ar[d]^{f\atimes_S \id_Y}\\
X\atimes_S Y'\ar[r]^{\id_X\atimes_S g} & X\atimes_S Y}
\]
is an isomorphism on $D^+(X\atimes_S Y',\Lambda)$.
\end{cor}

Here $\Leftrightarrow$ denotes an isomorphism of morphisms of topoi.

Recall that any locally coherent topos \cite[VI D\'efinition 2.3]{SGA4} has
enough points by Deligne's theorem \cite[VI Th\'eor\`eme 9.0]{SGA4}. The
stalk of $\alpha$ at every point $(x',y,\phi)$ of $X'\atimes_S Y$ is an
isomorphism by Proposition \ref{p.obc}.

Let $f\colon X\to S$ and $g\colon Y\to S$ be morphisms of topoi. For sheaves
$X'$, $Y'$, $S'$ on $X$, $Y$, $S$, equipped with morphisms $X'\to S'$ and
$Y'\to S'$ above $f$ and $g$, we write $X'\ttimes_{S'} Y'$ for the object
$p_1^* X'\times_{h^*S'} p_2^*Y'$ of $X\ttimes_S Y$, where $p_1\colon
X\ttimes_S Y\to X$ and $p_2\colon X\ttimes_S Y\to Y$ are the projections and
$h$ denotes $f p_1\simeq g p_2$.

The following is an analogue of \cite[XI Lemme 2.5]{ILO}, with essentially
the same proof.

\begin{lemma}\label{l.1}
Let $X\to S$ and $Y\to S$ be coherent morphisms of coherent topoi. Then
$X\ttimes_S Y$ is coherent and the projections $p_1$ and $p_2$ are coherent.
\end{lemma}

Lemma \ref{l.1} implies that objects of the form $X'\ttimes_{S'} Y'$ with
$X'$, $Y'$, $S'$ coherent form a generating family.

The following is an analogue of \cite[XI Corollaire 2.3.2]{ILO}.

\begin{lemma}\label{l.topos}
Let $f\colon X\to S$ and $g\colon Y\to S$ be local morphisms of local topoi.
Then the fiber product $X\ttimes_S Y$ is a local topos of center
$z=(x,y,\phi)$, where $x$ and $y$ are the centers of $X$ and $Y$,
respectively, and $\phi\colon f(x)\simeq g(y)$ is the unique isomorphism.
\end{lemma}

\begin{proof}
It suffices to show that for every sheaf $\cF$ on $X\ttimes_S Y$, any
element of the stalk $\cF_z$ lifts uniquely to a section. For this we may
assume that $\cF=X'\ttimes_{S'} Y'$ for sheaves $X'$, $Y'$, $S'$ on $X$,
$Y$, $S$, equipped with morphisms $X'\to S'$, $Y'\to S'$ above $f$ and $g$.
Any element of $\cF_z$ corresponds to compatible elements of $X'_x$, $Y'_y$,
$S'_s$, which correspond in turn to sections of $X'$, $Y'$, $S'$, providing
a section of $\cF$. Here $s$ denotes the center of $S$.
\end{proof}

\begin{remark}
As observed in \cite[XI Exemples 3.4 (2)]{ILO}, given morphisms of
    schemes $X\to S\leftarrow Y$, the morphism of topoi
\[(X\times_{S} Y)_\et \to X_\et\ttimes_{S_\et} Y_\et \]
is not an equivalence in general, even if $X$, $Y$, and $S$ are spectra of
fields.\footnote{The claim in \cite[XI Exemples 3.4 (2)]{ILO} that $G\mapsto
BG$ preserves fiber products is false (cf.\ \eqref{e.BH}).} Indeed, if
$S=\Spec(k)$ with $k$ separably closed, and $X=\Spec(k_1)$, $Y=\Spec(k_2)$
with $k_1/k$ and $k_2/k$ transcendental, then $X_\et\ttimes_{S_\et} Y_\et$
has only one isomorphism class of points, while $(X\times_S Y)_{\et}$ has
infinitely many.
\end{remark}

Let $\Lambda$ be a torsion commutative ring. The following is a
generalization of \cite[Lemme 10.1]{Org}.

\begin{prop}\label{p.Org}
Let $f\colon X\to S$ be a proper morphism of schemes. Let $g\colon Y\to
S_\et$ be a locally coherent morphism \cite[VI D\'efinition 3.7]{SGA4} of
locally coherent topoi. Then for any point $y$ of $Y$, any geometric point
$s$ of $S$, and any isomorphism $\phi\colon s_\et\simto g(y)$, the base
change map $\alpha\colon (Rp_{2*}L)_y\to R\Gamma(X_s,c^*L)$ associated to
the square
\begin{equation}\label{e.Org}
\xymatrix{(X_s)_{\et} \ar[r]^c\ar[d]\ar@{}[rd]|{\Leftrightarrow} & X_\et \ttimes_{S_\et} Y\ar[d]^{p_2}\\
\pt\ar[r]^y & Y}
\end{equation}
is an isomorphism for $L\in D(X_\et\ttimes_{S_\et} Y,\Lambda)$. Here
$X_s=X\times_S s$, and $c$ is induced by the diagram
\[\xymatrix{(X_s)_\et\ar[d] \ar[r]\ar@{}[rd]|{\Leftrightarrow} & \pt\ar[r]^y\ar[rd]_{s_\et}^{\Leftrightarrow} & Y\ar[d]^g\\
X_\et \ar[rr] && S_\et,}
\]
where the triangle is given by $\phi$.
\end{prop}

\begin{proof}
The proof is similar to that of \cite[Lemme 10.1]{Org}. We may assume that
$S$ is strictly local of center $s$ and $Y$ is local of center $y$. There
exists an integer $d$ such that the dimensions of the fibers of $f$ are $\le
d$. It suffices to show that $\alpha$ is an isomorphism for $L\in D^+$.
Indeed, this implies that $Rp_{2*}$ has cohomological dimension $\le 2d$,
and the general case follows.

Consider the diagram of topoi
\[\xymatrix{(X_s)_\et\ar[r]^c\ar[rd]_{i_\et}^\Leftrightarrow &X_\et \ttimes_{S_\et} Y\ar[d]^{p_1}\\
& X_\et}
\]
and the base change map $\beta\colon  i^* Rp_{1*}L\to c^* L$. Then $\alpha$
can be identified with the composition
\[R\Gamma(X, Rp_{1*}L)\simto R\Gamma(X_s,i^* Rp_{1*}L)\xrightarrow{R\Gamma(X_s,\beta)} R\Gamma(X_s, c^*L),\]
where the first arrow is proper base change. For any geometric point $x$ of
$X_s$, $\beta_x$ can be identified with the map
$R\Gamma((X_{(x)})_\et\ttimes_{S_\et} Y, L)\to L_{(x,y,\phi)}$, which is an
isomorphism by Lemma \ref{l.topos}. It follows that $\beta$ and $\alpha$ are
isomorphisms.
\end{proof}

\begin{remark}\leavevmode\label{r.Org0}
\begin{enumerate}
\item The square \eqref{e.Org} is not Cartesian in general. Indeed, in the
    case where $S=\Spec(k)$ with $k$ a separably closed field, $\dim(X)\ge
    1$, $Y=\pt$, and $s=\Spec(k')$ with $k'$ a transcendental (separably
    closed) extension of $k$, the morphism $(X_s)_{\et}\to X_\et$ is not
    an equivalence.

\item For locally coherent morphisms of locally coherent topoi $S_\et\to
    B\xleftarrow{h} B'$, we can take $g\colon S_\et\ttimes_{B} B'\to
    S_\et$ to be the first projection, which is a locally coherent
    morphism of locally coherent topoi by Lemma \ref{l.1}. In this case,
    $p_2$ can be identified with $f_\et\ttimes_B B'\colon X_\et \ttimes_B
    B'\to S_\et \ttimes_B B'$.

\item For locally coherent morphisms of locally coherent topoi $S_\et\to
    B\leftarrow B'$, we can take $g\colon Y=S_\et\atimes_{B} B'\to S_\et$
    to be the first projection, which is a locally coherent morphism of
    locally coherent topoi by \cite[XI Lemme 2.5]{ILO}. In this case,
    $p_2$ can be identified with $f\atimes_B B'\colon X_\et\atimes_B B'\to
    S_\et\atimes_B B'$ by \cite[XI Proposition 4.2]{ILO}, and we recover
    \cite[Lemme 10.1]{Org}. This can be identified with the special case
    of (2) with $h$ being the first projection $B\atimes_B B' \to B$.

\item For $f$ integral, Proposition \ref{p.Org} holds without the
    assumption that $\Lambda$ is torsion. Indeed, in the proof above, it
    suffices to replace proper base change by integral base change
    \cite[VIII Corollaire 5.6]{SGA4}.
\end{enumerate}
\end{remark}

Recall that for any morphism of topoi $f\colon X\to Y$, we have a projection
formula map
\[Rf_* L\otimes^L M\to Rf_*(L\otimes^L f^* M),\]
adjoint to the
composition
\[f^*(Rf_* L\otimes^L M)\simeq f^*Rf_*L\otimes^L f^* M\to
L\otimes^L f^* M.
\]

\begin{construction}\label{con.!}
Applying Nagata's compactification theorem \cite[Theorem 4.1]{Conrad} and
Deligne's gluing formalism \cite[XVII 3.3]{SGA4}, we define for $f\colon
X'\to X$ a separated morphism of finite type of coherent schemes equipped
with locally coherent morphisms of locally coherent topoi $X_\et\to S
\leftarrow Y$, a functor
\[R(f_\et\ttimes_S \id_Y)_!\colon D(X'_\et\ttimes_S
Y,\Lambda)\to D(X_\et\ttimes_S Y,\Lambda),
\]
isomorphic to $R(f_\et\ttimes_S\id_Y)_*$ for $f$ proper and left adjoint to
$(f_\et\ttimes_S\id_Y)^*$ for $f$ an open immersion, and compatible with
composition.

Given a Cartesian square
\[\xymatrix{U'\ar[r]^{j'}\ar[d]_{f_U} & X'\ar[d]^f\\
U\ar[r]^j & X}
\]
of coherent schemes with $f$ proper and $j$ an open immersion, as in
\cite[XVII Lemme 5.1.6]{SGA4} we need to check that the morphism
\begin{equation}\label{e.bc}
(j_\et\ttimes_S\id_Y)_! (f_{U\et}\ttimes_S \id_Y)_*\to (f_\et\ttimes_S
\id_Y)_*(j'_\et\ttimes_S\id_Y)_!
\end{equation}
induced by the inverse of the isomorphism
\[(j_\et\ttimes_S\id_Y)^*(f_{\et}\ttimes_S \id_Y)_*\simto (f_{U\et}\ttimes_S
\id_Y)_*(j'_\et\ttimes_S\id_Y)^*
\]
is an isomorphism. The restriction of \eqref{e.bc} to $U_\et\ttimes_S Y$ is
trivially an isomorphism. Let $X_1=X-U$ be a complement of $U$. The
restriction of the left-hand side of \eqref{e.bc} to $(X_1)_\et\ttimes_S Y$
is zero, and the restriction of the right-hand side of \eqref{e.bc} to
$(X_1)_\et\ttimes_S Y$ is zero by Proposition \ref{p.Org} and Remark
\ref{r.Org0} (2).

We have the following isomorphisms.
\begin{enumerate}
\item (base change on $X$) For any Cartesian square of coherent schemes
\[\xymatrix{X_1'\ar[r]^{h'}\ar[d]_{f'} & X'\ar[d]^f\\ X_1\ar[r]^h & X}\]
we have
\[(h_\et\ttimes_S \id_Y)^*R(f_\et\ttimes_S \id_Y)_!\simeq R(f'_\et\ttimes_S \id_Y)_!(h'_\et\ttimes_S \id_Y)^*.\]
\item (base change on $Y$) For any locally coherent morphism $g\colon
    Y'\to Y$ of locally coherent topoi, we have
\[
(\id_{X_\et}\ttimes_S g)^*R(f_\et\ttimes_S
    \id_Y)_!\simeq R(f_\et\ttimes_S \id_{Y'})_!(\id_{X'_\et}\ttimes_S g)^*.
\]
\item (projection formula) For $L\in D(X'_\et\ttimes_S Y,\Lambda)$, $M\in
    D(X_\et\ttimes_S Y,\Lambda)$, we have
\[R(f_\et\ttimes_S
    \id_Y)_!L \otimes^L M\simeq R(f_\et\ttimes_S
    \id_Y)_!(L \otimes^L(f_\et\ttimes_S
    \id_Y)^*M).\]
\end{enumerate}
For $f$ proper, these maps are given by adjunction. In this case the maps in
(1) and (2) are isomorphisms by Proposition \ref{p.Org} and Remark
\ref{r.Org0} (2). To show that the map in (3) is an isomorphism in this
case, we reduce by Proposition \ref{p.Org} and Remark \ref{r.Org0} (2) to
the classical case where $Y=S=\pt$. For $f$ an open immersion, the inverses
of these isomorphisms are standard \cite[Constructions 2.6, 2.7]{six}.
\end{construction}

\begin{construction}\label{c.dim}
The functor $R(f_\et\ttimes_S \id_Y)_!$ has cohomological dimension $\le
2d$, where $d$ is the maximum of the dimensions of the fibers of $f$. Thus
the functor admits a right adjoint
\begin{equation}\label{e.upshr}
R(f_\et\ttimes_S \id_Y)^!\colon D(X_\et\ttimes_S Y,\Lambda)\to
D(X'_\et\ttimes_S Y,\Lambda)
\end{equation}
by Lemma \ref{l.adjtriv} below applied to the proper case.

Assume $m\Lambda=0$ with $m$ invertible on $X$ and $f$ flat with fibers of
dimension $\le d$. Then the trace map $\Tr_{f_\et}(\Lambda)\colon Rf_{\et!}
\Lambda_{X'_\et} \to \Lambda_{X_\et}(-d)[-2d]$ for $f_\et$ induces a trace
map
\[\Tr_{f_\et\ttimes_S \id_Y}(\Lambda)\colon R(f_{\et}\ttimes_S \id_Y)_!
\Lambda \simeq p_1^*Rf_{\et!}\Lambda \to \Lambda(-d)[-2d]
\]
for $f_\et\ttimes_S \id_Y$, where $p_1\colon X_\et\ttimes_S Y\to X_\et$
denotes the projection. Here we used base change (2) on $Y$. By the
projection formula (3), this induces a natural transformation
\[
\Tr_{f_\et\ttimes_S\id_Y}\colon R(f_{\et}\ttimes_S \id_Y)_! (f_{\et}\ttimes_S \id_Y)^*(d)[2d]\to
\id,
\]
which induces, by adjunction, a natural transformation
\begin{equation}\label{e.tr}
\tr_{f_\et\ttimes_S \id_Y}\colon (f_{\et}\ttimes_S \id_Y)^*(d)[2d]\to R(f_{\et}\ttimes_S \id_Y)^!.
\end{equation}
\end{construction}

\begin{lemma}\label{l.adjtriv}
Let $h\colon Z'\to Z$ be a coherent morphism of locally coherent topoi such
that there exists a covering family of objects $Y_\alpha$ of $Z$ with
$h^*Y_\alpha$ algebraic \cite[VI D\'efinition 2.3]{SGA4}. Then for all $q$,
$R^q h_*$ commutes with filtered colimits. If, moreover, $Rh_*\colon
D(Z',\Lambda)\to D(Z,\Lambda)$ has finite cohomological dimension, then
$Rh_*$ admits a right adjoint.
\end{lemma}

\begin{proof}
For the first assertion we may assume $Z$ coherent and the assertion becomes
\cite[VI Th\'eor\`eme 5.1]{SGA4}. The second assertion follows by the Brown
representability theorem. More precisely, one applies \cite[Corollary
14.3.7]{KS} to the collection of $h_*$-acyclic sheaves, which is stable
under small direct sums by the first assertion.
\end{proof}

\begin{construction}\label{c.adj}
Consider functors
\[\xymatrix{\cD'\ar[d]_G\ar@{}[rd]|{\Leftarrow} & \ar[l]_{L'}\cC'\ar[d]^F &\ar@{}[rd]|{\Leftarrow} \cD'\ar[d]_G\ar[r]^{R'} & \cC'\ar[d]^F\\\cD & \cC\ar[l]_L & \cD\ar[r]^{R} & \cC}\]
and adjunctions $L\dashv R$ and $L'\dashv R'$. Then natural transformations
$\alpha\colon LF\to GL'$ correspond by adjunction to natural transformations
$\beta\colon FR'\to RG$. For $\alpha$ given, $\beta$ is the composition
\[FR'\to RLFR'\xrightarrow{\alpha} RGL'R'\to RG.\]
The same holds for adjunctions in $2$-categories \cite[Proposition
1.1.9]{Ayoub}.
\end{construction}

\begin{lemma}\label{l.trivshr}
Let $Y'$ be an object of $Y$ and let $g\colon Y'\to Y$. The map
\begin{equation}\label{e.trivshr}
(\id_{X'_\et}\ttimes_S g)^*R(f_\et \ttimes_S \id_{Y})^!\to R(f_\et\ttimes_S
\id_{Y'})^!(\id_{X_\et}\ttimes_S g)^*
\end{equation}
induced by Constructions \ref{con.!} (2) and \ref{c.adj} is a natural
isomorphism.
\end{lemma}

\begin{proof}
We may assume $f$ proper. It suffices to show that the map
$(\id_{X}\ttimes_S g)_!R(f_\et\ttimes_S \id_{Y'})_*\to R(f_\et \ttimes_S
\id_{Y})_*(\id_{X'}\ttimes_S g)_!$, left adjoint of \eqref{e.trivshr}, is an
isomorphism. For this, we may assume that $Y$ is a point. In this case, $Y'$
is a small set $I$ and $(\id \ttimes_S g)_!$ can be identified with
$\bigoplus_I$. It then suffices to note that $R(f_\et \ttimes_S \id_{Y})_*$
commutes with small direct sums.
\end{proof}

By adjunction, the projection formula induces the following isomorphisms
(see for example \cite[Proposition 2.24]{six}).

\begin{cor}\label{c.pj}
We have isomorphisms, natural in $L,M\in D(X_\et\ttimes_S Y,\Lambda)$ and
$L'\in D(X'_\et\ttimes_S Y,\Lambda)$,
\begin{gather}
\label{e.1.13.1}R(f_\et\ttimes_S \id_Y)^!R\cHom(L,M)\simeq R\cHom((f_\et\ttimes_S \id_Y)^*L,R(f_\et\ttimes_S \id_Y)^!M),\\
\label{e.1.13.2}R\cHom(R(f_\et\ttimes_S \id_Y)_!L',M)\simeq R(f_\et\ttimes_S \id_Y)_*R\cHom(L,R(f_\et\ttimes_S \id_Y)^!M),
\end{gather}
\end{cor}

For $L\in D(X,\Lambda)$, $M\in D(Y,\Lambda)$, we write $L\boxtimes^L M$ for
$p_1^*L\otimes^L p_2^*M$, where $p_1\colon X\ttimes_S Y \to X$ and
$p_2\colon X\ttimes_S Y \to Y$ are the projections.  The base change
isomorphisms and projection formula induce the following K\"unneth formula.

\begin{cor}\label{c.Kunn!}
Let $f\colon X'\to X$ and $g\colon Y'\to Y$ be morphisms separated of finite
type of coherent schemes, and let $X_\et\to S$ and $Y_\et\to S$ be locally
coherent morphisms of locally coherent topoi. Let $L\in D(X'_\et,\Lambda)$
and $M\in D(Y'_\et,\Lambda)$. Then we have an isomorphism
\[Rf_!L\boxtimes^L Rg_! M\simto R(f\ttimes_S g)_!(L\boxtimes^L M).\]
Here $R(f\ttimes_S g)_!$ denotes the functor $R(f\ttimes_S
\id_X)_!R(\id_{Y'}\ttimes_S g)_!\simeq R(\id_{X}\ttimes_S g)_!R(f\ttimes_S
\id_{Y'})_!$.
\end{cor}

Here we have omitted the subscript ``$\et$'' when no confusion arises.

\subsection{Over a classifying topos}
In this subsection, we study operations on fiber products of topoi over the
classifying topos of a profinite group, which include slices of vanishing
topoi. Among other things we prove K\"unneth formulas and biduality, which
will be used in Subsection \ref{ss.DLCo} and in many places in Sections
\ref{s.3} and \ref{s.3+}.

Let $s=BG$ be the classifying topos of a profinite group $G$, consisting of
discrete sets equipped with a continuous action of $G$. In our applications
$s$ will be the \'etale topos of a spectrum of a field. Let $X$ and $Y$ be
topoi over $s$. For $G=\{1\}$, $X\ttimes_s Y$ is the product topos $X\ttimes
Y$.

\begin{remark}\label{r.o}
For any topos $Z$, the category of morphisms of topoi $Z\to s$ is a
groupoid. In other words, if $p_1$ and $p_2$ are morphisms $Z\to s$, then
any natural transformation $\alpha\colon p_1^*\to p_2^*$ is an isomorphism.
Indeed, since $p_i^*$ carries the final object to the final object and
preserves coproducts, $\alpha$ is an isomorphism on constant sheaves. It
follows that $\alpha$ is an isomorphism on locally constant sheaves. It then
suffices to note that every sheaf on $s$ is a coproduct of locally constant
sheaves.

It follows that $X\ttimes_s Y$ is also the oriented product $X\atimes_s Y$.
\end{remark}

\begin{remark}\label{r.gen}
Given a generating family $\{U\}$ of $X$ and a generating family $\{V\}$ of
$Y$, $\{\Lambda_{U\ttimes_s V,X\ttimes_s Y}\}$ is a generating family of
$\Shv(X\ttimes_s Y,\Lambda)$. Here, $\Lambda_{U\ttimes_s V,X\ttimes_s Y}$ is
the free sheaf of $\Lambda$-modules on $X\ttimes_s Y$ generated by
$U\ttimes_s V$ \cite[IV Proposition 11.3.3]{SGA4}. Our notation $U\ttimes_s
V$ is consistent with \cite[XI 1.10 (b)]{ILO} (cf.\ \cite[Proposition
VI.3.15]{AG}), using the same notation for objects of a topos and the
corresponding localized topoi.

Indeed, by the construction of $X\ttimes_s Y$, objects of the form
$U\times_{a,BH,b} V$, $H$ running through open normal subgroups of $G$, form
a generating family of $X\ttimes_s Y$. Here $a\colon U\to BH$ and $b\colon
V\to BH$ are morphisms of topoi. We have
\begin{equation}\label{e.BH}
BH\ttimes_s BH \simeq \coprod_{g\in G/H} BH.
\end{equation}
The two projections on the component indexed by $gH$, $g\in G$ are given
respectively by the identity and by the morphism $c_g\colon BH\to BH$
induced by conjugation by $g$. The isomorphism class of $c_g$ depends only
on $gH$. It follows that $U\ttimes_s V\simeq \coprod_{g\in G/H}
U\ttimes_{a,BH,c_gb} V$.
\end{remark}

We restrict our attention to coherent topoi equipped with morphisms to $s$,
which we call coherent topoi over $s$. Note that $s$ is a coherent topos and
any morphism $X\to s$ with $X$ coherent is coherent. Indeed, $BH$, $H$
running through open normal subgroups of $G$, form a generating family of
coherent objects of $s$, and by \cite[VI Proposition 3.2]{SGA4}, it suffices
to show that $BH\ttimes_s X$ is coherent. By \eqref{e.BH} the pullback of
the projection $p\colon BH\ttimes_s X\to X$ by itself is coherent, so that
$p$ is coherent by \cite[VI Proposition 1.10 (ii)]{SGA4}.

Let $\pt=\bar s\to s$ be the point of $s$ with inverse image given by the
functor forgetting the action of $G$. By Remark \ref{r.o}, all points of $s$
are isomorphic.

Corollary \ref{c.obc} takes the following form.

\begin{prop}\label{p.prod}
Let $f\colon X'\to X$ and $g\colon Y'\to Y$ be morphisms of coherent topoi
over $s$. Assume that $f$ is coherent. Then the base change map
\[\BC_g\colon (\id_{X}\ttimes_s g)^*R(f\ttimes_s \id_{Y})_*\to R(f\ttimes_s
\id_{Y'})_*(\id_{X'}\ttimes_s g)^*
\]
associated to the (Cartesian) square of topoi
\[\xymatrix{X'\ttimes_s Y'\ar[r]^{\id_{X'}\ttimes_s g}\ar[d]_{f\ttimes_s
\id_{Y'}}\ar@{}[rd]|{\Leftrightarrow}
 & X'\ttimes_s Y\ar[d]^{f\ttimes_s \id_Y}\\
X\ttimes_s Y'\ar[r]^{\id_X\ttimes_s g} & X\ttimes_s Y}
\]
is a natural isomorphism on $D^+(X'\ttimes_s Y,\Lambda)$.
\end{prop}

\begin{remark}\label{r.cd}
Note that $Rf_{*}$ and $Rf_{\bar s*}$ have the same cohomological dimension.
Moreover, $R(f\ttimes_s \id)_*$ has the same cohomological dimension by the
proposition applied to $Y'=\pt$. If these cohomological dimensions are
finite, then the proposition holds more generally on the unbounded derived
category $D(X'\ttimes_s Y,\Lambda)$.
\end{remark}

\begin{remark}
The analogue of Proposition \ref{p.prod} does not hold for fiber products
    of schemes. For example, let $s=\Spec(k)$, where $k$ is a field,
    $X=Y=\PP^1_k$, and let $f\colon X'=\A^1_k\to X$ and $g\colon
    Y'=\{\infty\}\to Y$ be the immersions. Let $\Gamma\colon X'\to
    X'\times_s Y$ be the graph of $f$. Then $(\id_{X'}\times_s
    g)^*(\Gamma_*\Lambda)=0$ while $(\id_X\times_s g)^*(f\times_s
    \id_Y)_*(\Gamma_*\Lambda)$ can be identified with $g_*\Lambda$.
\end{remark}

\begin{cor}\label{c.pf}
Let $f$ and $Y$ be as above, $L\in D(X'\ttimes_s Y,\Lambda)$, $M\in
D(Y,\Lambda)$. Assume that one of the following holds:
\begin{enumerate}[(a)]
\item $Rf_{*}$ has finite cohomological dimension; or
\item $L,M\in D^+$, and $\Lambda$ has finite weak dimension; or
\item $L\in D^+$ and $M$ has finite tor-amplitude; or
\item $f=(f_0)_\et$, where $f_0\colon X_0'\to X_0$ is a morphism of finite
    type of Noetherian schemes, $m\Lambda=0$ for an integer $m$ invertible
    on $X_0$, and $L\in D^+$, $L\otimes^L p'^*_2M\in D^+$.
\end{enumerate}
Then the projection formula map
\[R(f\ttimes_s
\id_{Y})_*L \otimes^L p_2^*M \to R(f\ttimes_s \id_{Y})_*(L\otimes^L p'^*_2M)
\]
is an isomorphism. Here $p_2\colon X\ttimes_s Y\to X$ and $p'_2\colon
X'\ttimes_s Y\to X'$ denote the projections.
\end{cor}

Recall that the weak dimension of $\Lambda$ is the supremum of integers $n$
such that $\Tor_n^\Lambda(A,B)\neq 0$ for some $\Lambda$-modules $A$ and
$B$.

\begin{proof}
We may assume that $s$ is a point. By Proposition \ref{p.prod}, we may
further assume that $Y=s$. This case is standard. For (a), see \cite[Lemma
A.8]{IZ}. For (b), we reduce to the case $M\in D^b$, which is a special case
of (c). For (c), we reduce to the case where $M$ is a flat $\Lambda$-module,
and then to the trivial case where $M$ is a finite free $\Lambda$-module.
For (d), we reduce by localization at a point to the case where $X_0$ is
finite-dimensional, which is a special case of (a) by Gabber's theorem on
the finiteness of cohomological dimension \cite[XVIII$_{\text{A}}$ Corollary
1.4]{ILO}.
\end{proof}

\begin{remark}\label{r.tor}
In case (d), the projection formula implies that $R(f\ttimes_s \id_{Y})_*$
preserves complexes of tor-amplitude $\ge n$.
\end{remark}

Combining the corollary with the proposition, we obtain the following
K\"unneth formula for $R(f\ttimes_s g)_*$. The analogue for fiber products
of schemes over a field is \cite[III (1.6.4)]{SGA5}.

\begin{cor}\label{c.Kunn*}
Let $f\colon X'\to X$ and $g\colon Y'\to Y$ be coherent morphisms of
coherent topoi over~$s$. Let $L\in D(X',\Lambda)$, $M\in D(Y',\Lambda)$.
Assume that either one of the following holds:
\begin{enumerate}[(a)]
\item $Rf_{*}$ and $Rg_*$ have finite cohomological dimensions;
\item $L,M\in D^+$, and $\Lambda$ has finite weak dimension.
\end{enumerate}
Then the
K\"unneth formula map
\[Rf_*L\boxtimes^L Rg_* M\to R(f\ttimes_s
g)_*(L\boxtimes^L M)
\]
is an isomorphism.
\end{cor}

By a scheme over $s$ we mean a scheme $Z$ equipped with a morphism of topoi
$Z_\et\to s$.  Let $X$ be a coherent scheme over $s$. For any open subgroup
$H<G$, we have $X_\et\ttimes_s t \simeq (X_t)_\et$, where $t=BH$ and $X_t\to
X$ is a finite \'etale cover. Taking limit, we get $X_\et\ttimes_s \bar
s\simeq (X_{\bar s})_\et$, where $X_{\bar s}=\lim X_t$.

Let $Y$ be a coherent topos over $s$. Assume $\Lambda$ of torsion. By
Constructions \ref{con.!} and \ref{c.dim}, for $f\colon X'\to X$ a morphism
of schemes separated of finite type, the functors $(f_\et\ttimes_s \id_Y)_!$
and $(f_\et\ttimes_s \id_Y)^!$ are defined. In fact, in this case, the
proper base change used in the construction of $(f_\et\ttimes_s \id_Y)_!$
does not require Proposition \ref{p.Org} and can be reduced by Proposition
\ref{p.prod} to classical proper base change.

Assume $m\Lambda=0$ for some $m$ invertible on $X$.

\begin{prop}\label{l.trace}
For $f$ smooth of dimension $d$, the trace map \eqref{e.tr}
\[\tr_{{f_\et\ttimes_s \id_Y}}\colon
(f_\et\ttimes_s\id_Y)^*(d)[2d]\to R(f_\et\ttimes_s\id_Y)^!
\]
is an isomorphism.
\end{prop}

\begin{proof}
Let us first show that $\tr(L)$ is an isomorphism for $L\in
D^+(X_\et\ttimes_s Y,\Lambda)$. For this, it suffices to check that
$R\Gamma(U'_\et\ttimes_s V,\tr(L)|_{U'_\et\ttimes_s V})$ is an isomorphism
for $u\colon U'\to X'$ an \'etale morphism of coherent schemes and $V$ a
coherent object of $Y$. By Lemma \ref{l.trivshr}, $\tr(L)|_{X'_\et\ttimes_s
V}\simeq \tr(L|_{X_\et\ttimes_s V})$. Changing notation, it suffices to
check that
\[R\Gamma(U'_\et\ttimes_s Y,(u\ttimes_s
\id_Y)^*\tr(L))\simeq R\Gamma(U'_\et,Rp''_{1*}(u\ttimes_s
\id_Y)^*\tr(L))\simeq R\Gamma(U'_\et,u^*Rp'_{1*}\tr(L))
\]
is an isomorphism, where $p'_1\colon X'_\et\ttimes_s Y\to X'_\et$ and
$p''_1\colon U'_\et\ttimes_s Y\to U'_\et$ are the projections. We have a
commutative diagram
\[\xymatrix{f_{\et}^* Rp_{1*}L(d)[2d]\ar[d]_{\simeq}\ar[r]^{\tr(Rp_{1*}L)}_\sim  &
Rf_{\et}^!\ar[d]^{\simeq}Rp_{1*}L\\
Rp'_{1*}(f_\et\ttimes_s\id_Y)^*L(d)[2d]\ar[r]^-{Rp'_{1*}\tr(L)}& Rp'_{1*}R(f_\et\ttimes_s\id_Y)^!L.}
\]
The left vertical isomorphism is given by Proposition \ref{p.prod} applied
to the pair $(g,f_\et)$, where $g\colon Y\to s$. The right vertical
isomorphism is given by adjunction from Construction \ref{con.!} (2) applied
to the pair $(f,g)$. The upper horizontal isomorphism is \cite[XVIII
Th\'eor\`eme 3.2.5]{SGA4}. It follows that $Rp'_{1*}\tr(L)$ is an
isomorphism.

As in the proof of \cite[XVIII Th\'eor\`eme 3.2.5]{SGA4}, the above result
extends to unbounded $L$ by a more explicit construction of
$R(f_\et\ttimes_s \id_Y)^!$. For any coherent topos $E$, $\Shv(E,\Lambda)$
is equivalent to the Ind-category of the category of finitely presented
sheaves of $\Lambda$-modules on $E$ by \cite[VI Corollaire 1.24.2, I
Corollaire 8.7.7]{SGA4}. In fact, for each finitely presented object $U$ of
$E$, $u_!\Lambda_U$ is a finitely presented object of $\Shv(E,\Lambda)$,
where $u\colon U\to E$. Imitating \cite[XVIII Section 3.1]{SGA4}, we define
the modified Godement resolution for a sheaf $\cF=\colim_i \cF_i$ of
$\Lambda$-modules on $E$ by $M^\bullet(\cF)=\colim_i G^\bullet(\cF_i)$,
where $G^\bullet$ denotes the Godement resolution associated to a
conservative family of points of $E$ (which exists by Deligne's theorem
\cite[VI Th\'eor\`eme 9.0]{SGA4}), and $\cF_i$ are finitely presented. For a
compactification $f=\bar fj$, we define $(f_\et\ttimes_s \id)_!^\bullet
\cF\colonequals (\bar f_\et\ttimes_s\id)_*\tau^{\le 2d} M^\bullet
(j_\et\ttimes_s \id)_!\cF$, where $d$ is the maximum of the dimensions of
the fibers of $f$. Then $R(f_\et\ttimes_s \id)^!$ is the derived functor of
the complex of functors $(f_\et\ttimes_s \id)^!_\bullet$, right adjoint to
$(f_\et\ttimes_s \id)_!^\bullet$. By \cite[XVII Proposition 1.2.10]{SGA4}
(variant for a bounded complex of functors), $R(f_\et\ttimes_s \id)^!$ has
finite cohomological amplitude on the unbounded derived category.
\end{proof}

\begin{prop}\label{p.bcupper!}
Assume $f$ separated of finite presentation. The map
\[(\id_{X'_\et}\ttimes_s g)^*R(f_\et \ttimes_s \id_{Y})^!\to
R(f_\et\ttimes_s \id_{Y'})^!(\id_{X_\et}\ttimes_s g)^*
\]
induced by
Constructions \ref{con.!} (2) and \ref{c.adj} is a natural isomorphism on
$D^+$.
\end{prop}

\begin{proof}
We reduce to the following cases: (a) $f$ is smooth; (b) $f$ is a closed
immersion. For (a), we apply the preceding proposition. For (b), we apply
Proposition \ref{p.prod} to $Rj_*$, where $j$ is the complementary open
immersion.
\end{proof}

Similarly, we have the following projection formula.

\begin{prop}\label{p.proj}
Assume $X$ Noetherian finite-dimensional. Then the map
\[R(f_\et\ttimes_s \id_Y)^!L \otimes^L p_2'^*M\to R(f_\et\ttimes_s \id_Y)^! (L\otimes^L p_2^*M),\]
adjoint to the composition
\[R(f_\et\ttimes_s \id_Y)_!(R(f_\et\ttimes_s \id_Y)^!L \otimes^L p_2'^*M)\simeq R(f_\et\ttimes_s \id_Y)_!R(f_\et\ttimes_s \id_Y)^!L\otimes^L p_2^* M\to L\otimes^L p_2^*M\]
of Construction \ref{con.!} (3) and the adjunction map, is an isomorphism
for $L\in D(X_\et\ttimes_s Y,\Lambda)$ and $M\in D(Y,\Lambda)$.
\end{prop}

\begin{proof}
We proceed as in the proof of Proposition \ref{p.bcupper!}. For (a) we apply
Proposition \ref{l.trace}. For (b) we apply Corollary \ref{c.pf} (a) to
$Rj_*$. That $Rj_*$ has finite cohomological dimension is a theorem of
Gabber \cite[XVIII$_{\text{A}}$ Corollary 1.4]{ILO}.
\end{proof}

Let $X$ and $Y$ be coherent schemes over $s$. Again we omit the index
``$\et$'' when no confusion arises. For separated morphisms of finite type
of schemes $f\colon X'\to X$ and $g\colon Y'\to Y$, and $\Lambda$ of
torsion, the functors $R(f\ttimes_s g)_!$ and $R(f\ttimes_s g)^!$ are
defined. The base change isomorphism and projection formula above induce the
following K\"unneth formula. The analogue for fiber products of schemes over
a field is \cite[III (1.7.3)]{SGA5}.

\begin{cor}\label{c.Kunnup}
Assume $X$ and $Y$ are Noetherian finite-dimensional and $m\Lambda=0$ for
some $m$ invertible on $X$ and $Y$. Then for $L\in D(X,\Lambda)$, $M\in
D(Y,\Lambda)$, we have
\[Rf^!L\boxtimes^L Rg^! M\simto R(f\ttimes_s g)^!(L\boxtimes^L M). \]
\end{cor}

Let $\Lambda$ be a Noetherian commutative ring. Imitating \cite[Section
9.1]{Org}, we say that a sheaf of $\Lambda$-modules $\cF$ on $X\ttimes_s Y$
is \emph{constructible} if every stalk is finitely generated and there exist
finite partitions $X=\bigcup_i X_i$, $Y=\bigcup_j Y_j$ into disjoint
constructible locally closed subsets such that the restriction $\cF$ to each
$X_i\ttimes_s Y_j$ is locally constant. We write $\Shv_c(-,\Lambda)$ for the
full subcategory of $\Shv(-,\Lambda)$ spanned by constructible sheaves. The
subcategory $\Shv_c(X\ttimes_s Y,\Lambda)\subseteq \Shv(X\ttimes_s
Y,\Lambda)$ is stable under kernels, cokernels, and extensions, but in
general not stable under subobjects or quotients. Even for $X$ and $Y$
Noetherian, constructible sheaves are in general very different from
Noetherian sheaves. For example, in the case where $X$ and $Y$ are spectra
of separably closed fields, $\pt\ttimes_s \pt$ can be identified with the
topos associated to the underlying topological space of $G$, on which a
sheaf is Noetherian if and only if the support is finite  and the stalks are
finitely generated. See however Lemma \ref{l.Noeth} below.

\begin{lemma}\label{l.gen}
For $\cF$ constructible, there exist \'etale morphisms $u\colon U\to X$ and
$v\colon V\to Y$ of schemes with $U$ and $V$ affine and an epimorphism
$(u\ttimes_s v)_!\Lambda_{U\ttimes_s V}\to \cF$.
\end{lemma}

\begin{proof}
The proof is similar to that of \cite[IX Proposition 2.7]{SGA4}. Given
$\alpha\colon (u_\alpha \ttimes_s v_\alpha)_!\Lambda_{U_\alpha\ttimes_s
V_\alpha}\to \cF$, the set $E_\alpha$ of points $(x,y)\in X^\cons\times
Y^\cons$ such that $\alpha$ is an epimorphism on $x\ttimes_s y$ is open.
Here $X^\cons$ and $Y^\cons$ denote respectively $X$ and $Y$ equipped with
the constructible topology. By Remark \ref{r.gen}, $(E_\alpha)$ form an open
cover of $X^\cons\times Y^\cons$, and thus admits a finite subcover $(E_i)$
by the quasi-compactness of $X^\cons\times Y^\cons$. It suffices to take
$U=\coprod_i U_i$ and $V=\coprod_i V_i$.
\end{proof}

The following lemma is not needed in the rest of this paper.

\begin{lemma}\label{l.Noeth}
Assume that $X$ and $Y$ are Noetherian schemes. Assume that for each pair of
points $x\in X$, $y\in Y$, the double coset $G_x\backslash G/G_y$ is finite.
Here $G_x$ and $G_y$ denote respectively the images of $\Gal(\bar x/x)$ and
$\Gal(\bar y/y)$ in $G$, well-defined up to conjugation. Then $\cF\in
\Shv(X\ttimes_s Y,\Lambda)$ is constructible if and only if it is
Noetherian.
\end{lemma}

\begin{proof}
The proof is similar to those of \cite[IX Proposition 2.9]{SGA4} and
\cite[Lemme 9.3]{Org}. Since $\Shv(X\ttimes_s Y,\Lambda)$ admits a family of
constructible generators, it suffices to show that every constructible sheaf
$\cF$ is Noetherian. We show that every filtered family $(\cF_\alpha)$ of
subsheaves of $\cF$ admits a maximal element. Since each $\cF_\alpha$ is a
filtered colimit of constructible sheaves, and every morphism of
constructible sheaves has constructible image, each $\cF_\alpha$ is a
filtered union of constructible subsheaves. Thus we may assume that each
$\cF_\alpha$ is constructible. The maximal points of $X\ttimes_s Y$ are of
the form $p=(\bar x,\bar y,\phi)$, where $\bar x\to X$, $\bar y\to Y$ are
geometric generic points. For every point $q$ of $X\ttimes_s Y$, there
exists a specialization from a maximal point $p$ to $q$. The assumptions
imply that there are only finitely many isomorphism classes of maximal
points. Let $\alpha$ be such that $(\cF_\alpha)_p=(\cF_\beta)_p$ for every
maximal point $p$ and every $\beta\ge \alpha$. Let $U\subseteq X$ and
$V\subseteq Y$ be nonempty open subsets such that $\cF_\alpha|_{U\ttimes_s
V}$ and $\cF|_{U\ttimes_s V}$ are locally constant. Then
$\cF_\alpha|_{U\ttimes_s V}=\cF_\beta|_{U\ttimes_s V}$ for all $\beta\ge
\alpha$ and we conclude by Noetherian induction.
\end{proof}

\begin{lemma}\label{l.local}
Let $f\colon X_0\to X$ be a morphism of coherent schemes. Assume either (a)
$f$ weakly \'etale \cite[Definition 1.2]{BS} or (b) $(f,\Z/m\Z)$ universally
locally acyclic for some integer $m>0$ such that $m\Lambda=0$. For $L\in
D^-_c(X\ttimes_s Y,\Lambda)$ and $M\in D^+(X\ttimes_s Y,\Lambda)$, the
canonical map
\[(f\ttimes_s\id_Y)^*R\cHom(L,M)\to R\cHom((f\ttimes_s\id_Y)^*L,(f\ttimes_s\id_Y)^*M)\]
is an isomorphism.
\end{lemma}

\begin{proof}
In case (a), we may assume that $X_0$ and $X$ are affine and that $f$ is
ind-\'etale, by \cite[Theorem 1.3 (3)]{BS}. By Lemma \ref{l.gen}, we may
assume that $L=(u\ttimes_s v)_!\Lambda$. In this case, the map can be
identified with the base change map $(f\ttimes_s\id_Y)^*R(u\ttimes_s v)_*
M'\to R(u_{X_0}\ttimes_s v)_*(f_U\ttimes_s\id_Y)^* M'$, which is an
isomorphism by limit arguments for $f$ ind-\'etale and by \cite[XVI
Th\'eor\`eme 1.1]{SGA4} in case (b). Here $M'=(u\ttimes_s v)^* M$.
\end{proof}

The following lemma will be used later in Subsection \ref{s.4.4}.

\begin{lemma}\label{l.localshrbc}
Let $f\colon X_0\to X$ be a weakly \'etale morphism of coherent schemes. Let
$g\colon X'\to X$ be a finite morphism of finite presentation. Form a
Cartesian square
\[\xymatrix{X'_0\ar[r]^{f'}\ar[d]_{g_{0}}& X'\ar[d]^g \\
X_{0}\ar[r]^f & X.}
\]
Then the map $f'^*Rg^!\to Rg_{0}^!f^*$ adjoint to the inverse of the finite
base change isomorphism $f^*g_*\simto g_{0*}f'^*$ via Construction
\ref{c.adj} is an isomorphism on $D^+(X,\Lambda)$.
\end{lemma}

\begin{proof}
For $M\in D(X,\Lambda)$, we have a commutative diagram
\[\xymatrix{f^*R\cHom(g_*\Lambda,M)\ar@{-}[r]^\sim \ar[d] &
f^*g_*Rg^!M\ar[r]^\sim & g_{0*}f'^*Rg^!M\ar[d]\\
R\cHom(f^*g_*\Lambda,f^*M) & R\cHom(g_{0*}\Lambda,f^*M)\ar@{-}[r]^\sim\ar[l]_\sim & g_{0*}Rg_{0}^! f^*M.}
\]
For $M\in D^+$, since $g_*\Lambda$ is constructible, the left vertical arrow
is an isomorphism by Lemma \ref{l.local} (case $Y=s$). It follows that the
same holds for the right vertical arrow. We conclude by the fact that
$g_{0*}$ is conservative.
\end{proof}

\begin{lemma}\label{l.RHomdim}
Let $X$ be a Noetherian scheme with $m\Lambda=0$ for some $m$ invertible on
$X$. Let $L\in D_\cft^{\le a}(X,\Lambda)$ and $M\in D^+(X,\Lambda)$. Assume
that $M$ has tor-amplitude $\ge b$. Then $R\cHom(L,M)$ has tor-amplitude
$\ge b-a$.
\end{lemma}

\begin{proof}
The proof is similar to \cite[Th.\ finitude, Remarque 1.7]{SGA4d}. The case
where $L$ has locally constant cohomology sheaves is trivial. For the
general case, we may assume that $L=u_!\cF$ for an immersion $u\colon U\to
X$ and $\cF\in D_\cft$ with locally constant cohomology sheaves. Then
$R\cHom(L,M)\simeq Ru_*R\cHom(\cF,Ru^!M)$. It remains to show that $Ru_*$
and $Ru^!$ preserve objects of tor-amplitude $\ge n$. For $Ru_*$ this
follows from  Remark \ref{r.tor} (with $Y=s$). For $Ru^!$ we may assume that
$u$ is a closed immersion, and it suffices to apply the assertion for
$Rj_*$, where $j$ is the complementary open immersion.
\end{proof}

\begin{prop}\label{p.RHomKunn}
Assume $X$ and $Y$ are both Noetherian schemes, and $m\Lambda=0$ for some
$m$ invertible on $X$ and $Y$. Let $L\in D^-_c(X,\Lambda)$, $M\in
D_\cft(Y,\Lambda)$, $L'\in D^+(X,\Lambda)$, $M'\in D^+(Y,\Lambda)$. Assume
that either of the following holds:
\begin{enumerate}[(a)]
\item $L\in D_\cft(X,\Lambda)$ and $L'\boxtimes^L M'\in D^+$; or
\item $M'$ has tor-amplitude $\ge n$ for some integer $n$.
\end{enumerate}
Then the canonical map
\begin{multline*}
R\cHom(L,L')\boxtimes^L R\cHom(M,M') \to R\cHom(p_1^*L,p_1^* L')\otimes^L R\cHom(p_2^*M,p_2^*M')\\
\to
R\cHom(L\boxtimes^L M,L'\boxtimes^L M')
\end{multline*}
is an isomorphism in $D(X\ttimes_s Y,\Lambda)$.
\end{prop}

\begin{proof}
By Lemma \ref{l.local}, we may replace $X$ and $Y$ by strict localizations
at geometric points, which are finite-dimensional. In case (b), by Lemma
\ref{l.RHomdim}, $R\cHom(M,M')$ has tor-amplitude bounded from below, so
that we may assume $L=f_!\Lambda$ for $f\colon X'\to X$ \'etale with $X'$
affine, which is a special case of (a). Let us prove case (a). We may assume
$L=u_!\cF$, $M=v_!\cG$ for immersions $u\colon U\to X$, $v\colon V\to Y$ and
$\cF$, $\cG$ in $D_{\cft}$ of locally constant cohomology sheaves. In this
case, $L\boxtimes^L M\simeq (u\ttimes_s v)_! (\cF\boxtimes^L \cG)$ and the
map in question is the composition
\begin{align*}
&Ru_* R\cHom(\cF, Ru^! L')\boxtimes^L Rv_* R\cHom(\cG,Rv^! M')\\
\simeq &R(u\ttimes_s v)_* (R\cHom(\cF, Ru^! L')\boxtimes^L R\cHom(\cG,Rv^! M'))\\
\to &R(u\ttimes_s v)_* R\cHom(\cF\boxtimes^L \cG,Ru^! L' \boxtimes^L
Rv^! M')\\
\simeq &R(u\ttimes_s v)_* R\cHom(\cF\boxtimes^L \cG,R(u\ttimes_s
v)^!(L'\boxtimes^L M')).
\end{align*}
The first and last isomorphisms are Corollaries \ref{c.Kunn*} (a) and
\ref{c.Kunnup}. Here again we used Gabber's theorem that $Ru_*$ and $Rv_*$
have finite cohomological dimension \cite[XVIII$_{\text{A}}$ Corollary
1.4]{ILO}. Thus we are reduced to showing the proposition under the
additional assumption that the cohomology sheaves of $L$ and $M$ are locally
constant. This case is obvious by taking stalks.
\end{proof}

\begin{prop}
Assume $\Lambda$ of torsion. Let $f\colon X'\to X$ and $g\colon Y'\to Y$ be
separated morphisms of finite presentation. Then $R(f\ttimes_s g)_!$
preserves $D^b_c$ and $D_\ft$.
\end{prop}

\begin{proof}
It suffices to show $R(f\ttimes_s g)_!\cF\in D^b_c$ for $\cF$ a
constructible sheaf. By Lemma \ref{l.gen}, we may assume $\cF=\Lambda$. In
this case, we apply K\"unneth formula $R(f\ttimes_s g)_!\Lambda\simeq
Rf_!\Lambda\boxtimes^L Rg_!\Lambda$ (Corollary \ref{c.Kunn!}) and the
classical theorem that $Rf_!\Lambda$ and $Rg_!\Lambda$ are in $D^b_c$
\cite[XVII Th\'eor\`eme 5.3.6]{SGA4}. It follows from projection formula
that $R(f\ttimes_s g)_!$ preserves complexes of finite tor-dimension.
\end{proof}

\begin{lemma}\label{l.dev}
Assume $\ell\Lambda=0$ for a prime number $\ell$. Let
$\cF\in\Shv_c(X\ttimes_s Y,\Lambda)$ be a sheaf that becomes constant on
$U\ttimes_s V$ for finite \'etale covers $U\to X$ and $V\to Y$. Then there
exist finite \'etale covers $p\colon X'\to X$ and $q\colon Y'\to Y$ such
that $\cF$ is a direct summand of a sheaf $\cF'$ equipped with a finite
filtration with graded pieces of the form $(p\ttimes_s q)_* C_n$ with $C_n$
constant.
\end{lemma}

\begin{proof}
The proof is similar to that of \cite[IX Proposition 5.5]{SGA4}. We may
assume that $U\to X$ and $V\to Y$ are Galois of groups $G$ and $H$,
respectively. Choose $\ell$-Sylow subgroups $G'$ of $G$ and $H'$ of $H$. Let
$X'=U/G'$, $Y'=V/H'$. Take $\cF'=(p\ttimes_s q)_*(p\ttimes_s q)^*\cF$. The
composition $\cF\to \cF'\xrightarrow{\Tr} \cF$ of the adjunction and trace
maps equals $[G:G'][H:H']\id $, which is an isomorphism. Thus $\cF$ is a
direct summand of $\cF'$. Moreover, $(p\ttimes_s q)^*\cF$ is a
$\Lambda[L]$-module $M$, where $L=G'\times H'$.  Since $(g-1)^{\# L}=0$ in
$\Lambda[L]$ for each $g\in L$, the augmentation ideal $I$ of $\Lambda[L]$
satisfies $I^{(\# L)^2}=0$. Thus the filtration $(I^n M)_{n\ge 0}$ of $M$ is
finite and provides the desired filtration of $\cF'$.
\end{proof}

Assume $m\Lambda=0$ with $m$ invertible on $X$ and on $Y$.

\begin{prop}\label{p.prodcons}
Assume that $X$ is (a) quasi-excellent or (b) of finite type over a
Noetherian regular scheme of dimension $1$. Assume that $Y$ is (a) or (b).
Let $f\colon X'\to X$ and $g\colon Y'\to Y$ be morphisms of finite type.
Then $R(f\ttimes_s g)_*$ and $R(f\ttimes_s g)^!$ (the latter for $f$ and $g$
separated) preserve $D^b_c$ and $D_\cft$. Moreover, $R\cHom_{X\ttimes_s Y}$
induces $(D_c^-)^\op\times D_c^+\to D^+_c$.
\end{prop}

\begin{proof}
Let us first show that $R(f\ttimes_s g)_*$ preserves $D^b_c$. The
preservation of $D_{\cft}$ follows from this and Remark \ref{r.tor}. We may
assume $\ell\Lambda=0$ for some prime $\ell\mid m$. By Noetherian induction
on $X'$ and $Y'$, it suffices to show that for $\cF\in \Shv_c(X'\ttimes_s
Y',\Lambda)$, there exist open immersions $u\colon U'\to X'$, $v\colon V'\to
Y'$ with $U'$ and $V'$ nonempty such that $R(u\ttimes_s v)_*\cF'\in D^+_c$
and $R(fu\ttimes_s gv)_*\cF'\in D^+_c$, where $\cF'=\cF|_{U'\ttimes_s V'}$.
For this it suffices to show that for some open subgroup $H<G$ and $t=BH$,
we have $R(u_t\ttimes_t v_t)_*\cF'_t\in D^+_c$ and $R(f_tu_t\ttimes_t
g_tv_t)_*\cF'_t\in D^+_c$, where $\cF'_t=\cF'|_{U'_t\ttimes_t V'_t}$. There
exist $U'$, $V'$, and $t$ such that $\cF'_t$ is locally constant and
trivialized by finite \'etale covers on each pair of connected components of
$U'_t$ and $V'_t$. Changing notation (replacing $\cF$ by $\cF'_t$), it
suffices to show that for $\cF$ locally constant and trivialized by finite
\'etale covers, we have $R(f\ttimes_s g)_*\cF\in D^b_c$. By Lemma
\ref{l.dev}, we may assume that $\cF=(p\ttimes_s q)_*C$ for finite \'etale
covers $p\colon X''\to X'$ and $q\colon Y''\to Y'$ and $C$ constant.
Changing notation again, we may assume $\cF$ constant. By projection
formula, we may assume $\Lambda=\Z/\ell\Z$ and $\cF=\Lambda$. This case
follows from the K\"unneth formula (Corollary \ref{c.Kunn*} (b)) and
finiteness theorems for $Rf_*\Lambda$ and $Rg_*\Lambda$ of Deligne
\cite[Th.\ finitude, Th\'eor\`eme 1.1]{SGA4d} and Gabber \cite[Introduction,
Th\'eor\`eme~1]{ILO}.

For $R(f\ttimes_s g)^!$, we reduce to the case of smooth morphisms, which
follows from Proposition \ref{l.trace}, and the case of closed immersions,
which follows from the first assertion applied to the complementary open
immersions.

To show $R\cHom(\cF,L)\in D^+_c$ for $\cF$ constructible and $L\in D^+_c$,
we may assume $\cF=(u\ttimes_s v)_!\Lambda$ by Lemma \ref{l.gen}. Then
$R\cHom(\cF,L)\simeq R(u\ttimes_s v)_*(u\ttimes_s v)^*L$.
\end{proof}

Assume that $X$ is (a) excellent admitting a dimension function or (b) of
finite type over a regular Noetherian scheme of dimension one. Then $X$ is
equipped with a dualizing complex $K_X$ for $D_{\cft}(X,\Lambda)$, unique up
to tensor product with invertible objects. For $f\colon X'\to X$ separated
of finite type, $Rf^!K_X$ is a dualizing complex for $D_{\cft}(X',\Lambda)$.
If $\Lambda$ is Gorenstein, then $K_X$ is also a dualizing complex for
$D^b_c(X,\Lambda)$ (\cite[Th.\ finitude, 4.7]{SGA4d}, \cite[XVII
Th\'eor\`emes 6.1.1, 7.1.2, 7.1.3]{ILO}). Let $\Lambda=\prod_i\Lambda_i$
with each $\Spec(\Lambda_i)$ connected. For each $i$, $K_X$ determines a
dimension function $\delta_{X,i}$: for each $x\in X$,
$R\Gamma_x(K)\otimes_\Lambda \Lambda_i$ is concentrated in degree
$-2\delta_{X,i}(x)$.

Assume $Y$ is (a) or (b). Then $X$ and $Y$ are equipped with dualizing
complexes $K_X$ and $K_Y$. Let $K_{X\ttimes_s Y}=K_X\boxtimes^L K_Y$ and
$D_{X\ttimes_s Y}=R\cHom(-,K_{X\ttimes_s Y})$. Proposition \ref{p.RHomKunn}
(b) takes the following form.

\begin{cor}\label{c.DKunn}
For $L\in D^-_c(X,\Lambda)$ and $M\in D_{\cft}(Y,\Lambda)$, the canonical
map
\[D_X L\boxtimes^L D_Y M\to D_{X\ttimes_s Y}(L\boxtimes^L M)\]
is an isomorphism.
\end{cor}

\begin{prop}\label{p.dual}\leavevmode
\begin{enumerate}
\item $D_{X\ttimes_s Y}$ preserves $D_{\cft}(X\ttimes_s Y,\Lambda)$. The
    evaluation map $\ev_L\colon L \to D_{X\ttimes_s Y}D_{X\ttimes_s Y} L$
    is an isomorphism for $L\in D_{\cft}(X\ttimes_s Y,\Lambda)$.
\item Assume $\Lambda$ Gorenstein. $D_{X\ttimes_s Y}$ preserves
    $D^b_c(X\ttimes_s Y,\Lambda)$. The evaluation map $\ev_L\colon L \to
    D_{X\ttimes_s Y}D_{X\ttimes_s Y} L$ is an isomorphism for $L\in
    D^b_c(X\ttimes_s Y,\Lambda)$. The quasi-injective dimension of
    $K_{X\ttimes_s Y}$ (namely, the cohomological dimension of
    $D_{X\ttimes_s Y}$ when restricted to $D^b_c(X\ttimes_s Y,\Lambda)$)
    is
    \[d=\sup_i\left(\dim
    \Lambda_i-2\inf\delta_{X,i}-2\inf\delta_{Y,i}\right).
    \]
\end{enumerate}
\end{prop}

In particular, in (2) $K_{X\ttimes_s Y}$ has finite quasi-injective
dimension if and only if $\Lambda$, $X$ and $Y$ are finite-dimensional.

\begin{proof}
Let us first show the preservation of $D_\cft$ and $D^b_c$ (the latter
assuming $\Lambda$ Gorenstein) and biduality. This is trivial if $X$ and $Y$
are regular and we restrict to complexes with locally constant cohomology
sheaves. We reduce the general case to the case of $L=(u\ttimes_s v)_! \cF$,
where $u\colon U\to X$ and $v\colon V\to Y$ are immersions with $U$ and $V$
regular and $\cF$ has locally constant cohomology sheaves. Then
$D_{X\ttimes_s Y}L\simeq R(u\ttimes_s v)_* D_{U\ttimes_s V}\cF$. Note that
$R(u\ttimes_s v)_*$ preserves $D^b_c$ and $D_\cft$ by Proposition
\ref{p.prodcons}. The map $\ev_L$ is the composition
\begin{multline*}
(u\ttimes_s v)_! \cF\xrightarrow[\sim]{\ev_\cF} (u\ttimes_s v)_! D_{U\ttimes_s V}D_{U\ttimes_s V}\cF \\
\xrightarrow[\sim]{\eqref{e.DD}} D_{X\ttimes_s Y}R(u\ttimes_s v)_* D_{U\ttimes_s V}\cF\simeq D_{X\ttimes_s Y}D_{X\ttimes_s Y}(u\ttimes_s v)_! \cF,
\end{multline*}
where we used Proposition \ref{p.bidual} below.

It remains to show that for $\Lambda$ Gorenstein, the quasi-injective
dimension $c$ of $K_{X\ttimes_s Y}$ is $d$. This is similar to the proof of
\cite[XVII Proposition 6.2.4.1]{ILO}. For each maximal ideal $\fm$ of
$\Lambda_i$, taking $L$ to be the constant sheaf $\Lambda/\fm$ on
$\overline{\{x\}}\ttimes_s\overline{\{y\}}$, extended to $X\ttimes_s Y$, we
get $c\ge \dim \Lambda_\fm-2\delta_{X,i}(x)-2\delta_{Y,i}(y)$. It follows
that $c\ge d$. To show $c\le d$, we reduce to the case $L=(u\ttimes_s v)_!
\cF$ as above with $\cF$ a sheaf. We may further assume $U$ and $V$ are
connected, with generic points $x$ and $y$ respectively. Then $D_{U\ttimes_s
V}\cF$ has cohomological degrees $\le
\sup_i(\dim\Lambda_i-\delta_{X,i}(x)-\delta_{Y,i}(y))$. Moreover, by
\cite[XVIII$_{\text{A}}$ Theorem 1.1]{ILO} and Remark \ref{r.cd},
$R(u\ttimes_s v)_*$ has cohomological dimension $\le 2\dim\overline U
+2\dim\overline V$. It follows that $c\le d$.
\end{proof}

Let $f\colon X'\to X$ and $g\colon Y'\to Y$ be separated morphisms of finite
type. We have a natural transformation
\begin{multline}\label{e.DD}
R(f\ttimes_s g)_!
D_{X'\ttimes_s Y'} \xrightarrow{\ev} D_{X\ttimes_s Y}D^\op_{X\ttimes_s Y}R(f\ttimes_s g)_!
D_{X'\ttimes_s Y'} \\
\mathrel{\tfrac{\eqref{e.1.13.2}}{\sim}} D_{X\ttimes_s Y}(R(f\ttimes_s g)_*)^\op D^\op_{X'\ttimes_s Y'}
D_{X'\ttimes_s Y'} \xrightarrow{\ev} D_{X\ttimes_s Y}(R(f\ttimes_s g)_*)^\op,
\end{multline}
where $\ev$ denotes the evaluation maps.

\begin{prop}\label{p.bidual}
Let $L\in D^-_c(X'\ttimes_s Y',\Lambda)$. Assume either $L\in D^b_c$ or $X$
and $Y$ are finite-dimensional. Then $R(f\ttimes_s g)_! D_{X'\ttimes_s Y'}L
\to D_{X\ttimes_s Y}R(f\ttimes_s g)_* L$ is an isomorphism.
\end{prop}

\begin{proof}
For $L\in D^b_c$, by Lemma \ref{l.local}, we may assume that $X$ and $Y$ are
strictly local. Thus we may assume $X$ and $Y$ finite-dimensional. Since
$R(f\ttimes_s g)_*$ has finite cohomological dimension, we may assume $L\in
\Shv_c$. By Lemma \ref{l.gen}, we may further assume $L=(u\ttimes_s
v)_!\Lambda\simeq M\boxtimes^L N$ for $u$ and $v$ affine and \'etale, where
$M=u_!\Lambda$ and $N\in v_!\Lambda$ in $D_\cft$. Via K\"unneth formula for
$R(f\ttimes_s g)_!$, $R(f\ttimes_s g)_*$, and $D$ (Corollaries
\ref{c.Kunn!}, \ref{c.Kunn*}, \ref{c.DKunn}), \eqref{e.DD} can be identified
with the $\boxtimes^L$ of the isomorphisms
\[Rf_!D_{X'}M\simto D_{X}Rf_* M,\quad Rg_!D_{Y'}N\simto D_{Y}Rg_* N.\]
\end{proof}

\begin{remark}\label{r.dual}
The natural transformation \eqref{e.DD} can be interpreted formally as
follows. We write $\alpha\colon F\dashv G$ for an adjunction
$\alpha_{X,Y}\colon \Hom(FX,Y)\simeq \Hom(X,GY)$. Then $\alpha^\op\colon
G^\op\dashv F^\op$, where $\alpha^\op_{Y,X}=\alpha_{X,Y}^{-1}$. By a
\emph{duality} on a category $\cC$ we mean a functor $D\colon \cC^\op\to
\cC$ equipped with an adjunction $\alpha\colon D^\op\dashv D$ such that
$\alpha^\op=\alpha$. Here $D^\op\colon (\cC^\op)^\op\to \cC^\op$ and we have
identified $(\cC^\op)^\op$ with $\cC$. Similarly $\alpha^\op\colon
D^\op\dashv (D^\op)^\op$ and we have identified $(D^\op)^\op$ with $D$.
Equivalently, a duality on $\cC$ is a functor $D\colon \cC^\op\to \cC$
equipped with a natural transformation $\ev\colon \id_\cC\to DD^\op$ such
that the composite $D\xrightarrow{\ev D} DD^\op D \xrightarrow{D \ev^\op} D$
is the identity. Note that we do not require $\ev$ to be a natural
isomorphism. If $(\cC,\otimes)$ is a closed symmetric category and $K$ is an
object of $\cC$, then $D_K=\cHom(-,K)$, where $\cHom$ is the internal Hom
functor, is a duality on $\cC$ \cite[Construction A.4.1]{SZ}. The adjunction
$D_K^\op\dashv D_K$ is given by
\[
\Hom(L,D_K M)\simeq \Hom(L\otimes M,K) \simeq \Hom(M\otimes L,K)\simeq \Hom(M,D_K L)
\simeq \Hom_{\cC^\op}(D_K^\op L,M).
\]
Here the second isomorphism is induced by the symmetric constraint $L\otimes
M\simeq M\otimes L$.

The natural transformation \eqref{e.DD} is obtained from Construction
\ref{c.adj} applied to the opposite of the canonical natural isomorphism
\begin{equation}\label{e.DD2}
R(u\ttimes_s v)_* D_{U\ttimes_s V}\simeq
D_{X\ttimes_s Y} R(u\ttimes_s v)_!^\op.
\end{equation}
In the terminology of \cite[Definition A.3.3]{SZ}, \eqref{e.DD} and
\eqref{e.DD2} are form transformations, transposes of each other.
\end{remark}

\begin{example}\label{ex.S}
Let $S$ be a Henselian local scheme and let $i\colon s\to S$ be the
inclusion of the closed point. The functor $i^{-1}\colon \Et(S)\to \Et(s)$
between \'etale sites admits a left adjoint $\pi^{-1}\colon \Et(s)\to
\Et(S)$, which commutes with finite limits and extends the equivalence
between finite \'etale sites. The functor $\pi^{-1}$ is continuous and
induces a morphism of topoi $\pi\colon S_\et\to s_\et$, left adjoint to
$i_\et$ (which means $\pi_*\simeq i^*$ here). For any scheme $Y$ over $s$,
we have $Y\ttimes_s S\simeq Y\atimes_s S \simeq Y\atimes_S S$ by Remark
\ref{r.o} and Lemma \ref{l.univ} below.

In the special case where $S$ is the spectrum of a Henselian discrete
valuation ring, we recover the topos $Y\ttimes_s S$ of \cite[XIII Section
1.2]{SGA7II}. A sheaf $\cF$ on $Y\ttimes_s S$ can be identified with a
triple $(\cF_s,\cF_\eta,\phi)$, where $\cF_s$ is a sheaf on $Y$, $Y_\eta$ is
a sheaf on $Y\ttimes_s \eta\simeq Y\atimes_s \eta \simeq Y\atimes_S \eta$,
and $\phi\colon p^*\cF_s\to \cF_\eta$ is a morphism of sheaves. Here
$p\colon Y\ttimes_s\eta \to Y$ denotes the projection. The functor
$R(f\ttimes_s \id_S)_!$ in this case reduces to \cite[XIII 2.1.6
c)]{SGA7II}.
\end{example}

\begin{lemma}\label{l.univ}
Consider morphisms $\xymatrix{X\ar[r]^a & s\ar@/^/[r]^{i}\ar@{}[r]|\top& S
\ar@/^/[l]^\pi &Y\ar[l]_{b}}$ in a $2$-category with $\pi$ left adjoint to
$i$. Then $X\atimes_{a,s,\pi b} Y\simeq X\atimes_{ia,S,b} Y$.
\end{lemma}

\begin{proof}
The two oriented products satisfy the same universal property. Indeed, for
$X\xleftarrow{x} T\xrightarrow{y} Y$, giving $ax \Leftarrow \pi by$ is
equivalent to giving $iax \Leftarrow \pi b$.
\end{proof}

\section{The functor $L\Co$}\label{s.2}

After preliminaries on the Iwasawa twist (Sections \ref{ss.Iwa} and
\ref{ss.Iwafiber}), we study adjunctions and duality for the functor $L\Co$.
The main results of this section are Theorems \ref{t.adj} and \ref{t.LPhi}.

\subsection{Iwasawa twist}\label{ss.Iwa}
Let $\Lambda$ be a commutative ring with $m\Lambda=0$ for some integer
$m>0$. Let $G$ be a profinite group. Recall that the completed group ring
\[R=\Lambda[[G]]=\lim_{H}\Lambda[G/H],\]
where $H$ runs through open normal subgroups of $G$, can be identified with
the ring of $\Lambda$-valued measures on $G$ with convolution product. For
any discrete $\Lambda$-$G$-module $M$ and any normal subgroup
$P\triangleleft G$, we write $M^P$ and $M_P$ for the $\Lambda$-$G$-modules
of $G$-invariants and $G$-coinvariants, respectively. Note that if the
(supernatural) order of $P$ is prime to $m$, then the canonical map $M^P\to
M_P$ is an isomorphism, with inverse carrying the class of $x\in M$ to $\int
px \,d\mu(p)$, where $\mu$ is the Haar measure on $P$ of mass $1$. In this
case we write ${}^P M=\Coker(M^P\to M)$ and we have a canonical isomorphism
\[M\simeq M^P\oplus {}^P M.
\]

Let $A$ be a set of prime numbers containing all prime divisors of $m$ and
let $\Z_A=\prod_{\ell\in A}\Z_\ell$. Let $G$ be a profinite group fitting
into a short exact sequence
\begin{equation}\label{e.Galois}
1\to Z_A(1)\to G\xrightarrow{\pi} G_s\to 1,
\end{equation}
where $Z_A(1)$ is isomorphic to $\Z_A$. The conjugation action of $G_s$ on
$Z_A(1)$ provides a character $\chi\colon G_s\to \Z_A^\times$.  Let $\sigma$
be a generator of $Z_A(1)$ and let $t=\sigma-1\in
\Lambda[[Z_A(1)]]\subseteq\Lambda[[G]]$. Recall that the classical Iwasawa
algebra $\Lambda[[Z_A(1)]]$ is simply $\Lambda[[t]]$.

\begin{lemma}\label{l.R}
$t$ is a non-zero-divisor in $R$ and $tR=Rt$ does not depend on the choice
of $\sigma$.
\end{lemma}

\begin{proof}
Let $r\in R$ such that $tr=0$. We regard $r$ as a measure on $G$ with values
in $\Lambda$. For any open and closed subset $X\subseteq G$, we have
$r(\sigma X)=r(X)$. Let $H$ be an open normal subgroup of $G$ such that
$H\cap  Z_A(1)=mn Z_A(1)$, and let $H'=n Z_A(1)\cdot H$. Then
$H'=\coprod_{i=1}^m \sigma^{in} H$, so that $r(H'g)=mr(Hg)=0$ for any $g\in
G$. Since any open subgroup of $G$ contains such an $H'$, we have $r=0$.
Similarly, $rt=0$ implies $r=0$.

That $tR=Rt$ follows from the identity
\begin{equation}\label{e.twist}
g t g^{-1}=(1+t)^{\chi(g)}-1=\sum_{n\ge
1}
\binom{\chi(g)}{n}t^n
\end{equation}
in $R$ for $g\in G$. Here we have denoted $\chi(\bar g)$ by $\chi(g)$, where
$\bar g$ is the image of $g$ in $G_s$.  That $tR$ does not depend on the
choice of $\sigma$ follows from the standard fact that $t\Lambda[[Z_A(1)]]$
does not depend on the choice of $\sigma$, which follows in turn from the
expansion of $(1+t)^s-1$ for $s\in \Z_A^\times$.
\end{proof}

\begin{remark}
Although we do not need this, let us show that the canonical homomorphism
$\Lambda[[\pi]]\colon \Lambda[[G]] \to\Lambda[[G_s]]$ induces an isomorphism
$R/tR\simto \Lambda[[G_s]]$. Clearly the image of $t$ in $\Lambda[[G_s]]$ is
zero. Moreover, any continuous section of $\pi$ (not necessarily a
homomorphism) induces a section of $\Lambda[[\pi]]$. Now let $s\in R$ be a
measure on $G$ such that $\Lambda[[\pi]](s)=0$. We need to find a measure
$r\in R$ such that $s=tr$. For $H$ and $H'$ as in the proof of Lemma
\ref{l.R} and $g\in G$, we take $r(H'g)=(\tau(n)s)(Hg)$, where
\[\tau(n)\colonequals \sum_{i=1}^m\sum_{j=0}^{in-1}\sigma^j=-\sum_{j=0}^{mn-1}\left\lfloor\frac{j}{n}\right\rfloor\sigma^j\in R.\]
Note that the definition depends only on $H'g$. Indeed, we have
$\sigma^n\tau(n)-\tau(n)=\sum_{j=0}^{mn-1}\sigma^j$, so that
$(\sigma^n\tau(n)s-\tau(n)s)(Hg)=s(Z_A(1)Hg)=0$ by the assumption on $s$.
Let $H_0<H$ be an open normal subgroup of $G$ with $Z_A(1)\cap H_0=amn
Z_A(1)$ and let $H'_0=an Z_A(1)\cdot H_0$. We have $H=\coprod_{h\in
S}\coprod_{k=0}^{a-1}\sigma^{-kmn}H_0 h$ for a subset $S\subseteq H$ and
$H'=\coprod_{h\in S}\coprod_{k=0}^{a-1}H'_0 \sigma^{-kn}h$. Applying the
identity $b=\sum_{k=0}^{a-1}\lfloor \frac{b+k}{a} \rfloor$ to $b=\lfloor
j/n\rfloor$, we get
\[\sum_{k=0}^{a-1}\sigma^{kn}\tau(an)=-\sum_{j=0}^{amn-1}\left\lfloor\frac{j}{n}\right\rfloor\sigma^j=\sum_{k=0}^{a-1}\sigma^{kmn}\tau(n),\]
so that $r(H'g)=\sum_{h\in S}\sum_{k=0}^{a-1}r(H'_0\sigma^{-kn}hg)$. It
follows that $r$ is a well-defined measure on $G$. Finally,
$t\tau(n)=\sum_{i=1}^m \sigma^{in}$ so that $(tr)(H'g)=s(H'g)$.
\end{remark}

The following definition is due to Beilinson, at least for $G=\Z_\ell(1)$.

\begin{definition}[Iwasawa twist]\label{d.twist}
Consider the $R$-bimodule $R(1)^\tau=tR$. For an $R$-module $M$, we consider
the $R$-modules
\[M(1)^\tau=R(1)^\tau\otimes_R M,\quad M(-1)^\tau=\Hom_R(R(1)^\tau,M).\]
\end{definition}

By Lemma \ref{l.R}, we have isomorphisms of $R$-modules $M\simto
M(1)^\tau(-1)^\tau$ and $M(-1)^\tau(1)^\tau\simto M$. The $R$-bimodule
$R(1)^\tau$ is invertible in the sense that we have an isomorphism of
$R$-bimodules $R(1)^\tau\otimes_R R(-1)^\tau\simeq R$.

The $R$-module $M(1)^\tau$ can be described more explicitly as follows. For
each $\sigma$, we define a ring isomorphism $\rho_\sigma\colon R\to R$ by
\[\rho_\sigma(g)=\sum_{n\ge 1}\binom{\chi(g)}{n}t^{n-1}g\]
for $g\in G$.  That $\rho_\sigma$ defines a ring homomorphism follows from
the identity $gt=t \rho_\sigma(g)$ \eqref{e.twist} and the fact that $t$ is
a non-zero-divisor in $R$. We have
\[\rho_\sigma^{-1}(g)=g\sum_{n\ge 1}\binom{\chi(g^{-1})}{n}t^{n-1}.\]
Consider the isomorphism of $\Lambda$-modules $u\colon M\to M(1)^\tau$
carrying $x$ to $t\otimes x$. For $g\in G$, we have $g(t\otimes x)=t\otimes
\rho_\sigma(g)x$, namely $u(\rho_\sigma(g)x)=gu(x)$. Thus if
$M=(M,\alpha\colon R\to\End_\Lambda(M))$, then $u\colon (M,\alpha\circ
\rho_\sigma)\simto M(1)^\tau$ is an isomorphism of $R$-modules.

\begin{remark}\label{r.crossed}
The constructions preserve discrete $\Lambda$-$G$-modules. For $M$ a
discrete $\Lambda$-$G$-module, $M(-1)^\tau$ can be described as follows.
Composition with the crossed homomorphism $Z_A(1)\to R(1)^\tau$ carrying
$\xi$ to $\xi-1$ gives an isomorphism of $\Lambda$-$G$-modules $M(-1)^\tau
\simto Z^1_\cont(Z_A(1),M)$, where $Z^1_\cont(Z_A(1),M)$ is the
$\Lambda$-module of continuous crossed homomorphisms $f\colon Z_A(1)\to M$,
with $G$-action given by $(gf)(\xi)=g(f(g^{-1}\xi g))$ for $g\in G$.
\end{remark}

\begin{remark}\label{r.t2}
If $t^n M=0$ for some integer $n$ such that all primes $\ell \le n$ are
invertible in $\Lambda$, then we have an isomorphism of $R$-modules $
M(1)\colonequals Z_A(1)\otimes_{\Z_A} M\simto M(1)^\tau$ carrying
$\sigma\otimes x$ to $t\otimes \sum_{i=1}^{n} (-1)^{i-1}t^{i-1}x/i$.
\end{remark}

\begin{remark}\label{r.A}
Let $A'$ be a subset of $A$ containing all prime factors of $m$ and let
$Z_{A'}(1)$ be the maximal pro-$A'$ quotient of $Z_A(1)$. Let
$K=\Ker(Z_A(1)\to Z_{A'}(1))$. For a discrete $G$-module $M$, we have
$M(1)^\tau\simeq M^{K}(1)^\tau\oplus {}^{K}M$. Indeed, for $M$ such that
$M^K=0$, since $M^{Z_A(1)}=M_{Z_A(1)}=0$, $t$ acts bijectively on $M$.
\end{remark}

\begin{remark}
In our applications \eqref{e.Galois} splits so that $G\simeq Z_A(1)\rtimes
G_s$. In this case an $R$-module gives rise to a $G_s$-$R_0$-module, where
$R_0=\Lambda[[Z_A(1)]]$ is a $G_s$-ring. One can define Iwasawa twist for
$G_s$-$R_0$-module by $tR_0\otimes_{R_0} -$, which is compatible with
Definition \ref{d.twist}.
\end{remark}

\subsection{Iwasawa twist on fiber products}\label{ss.Iwafiber}

Let $S$ be the spectrum of a Henselian discrete valuation ring of residue
characteristic $p\nmid m$. Let $s$ and $\eta$ be the closed point and the
generic point of $S$, respectively. Let $\bar \eta\to \eta$ be an algebraic
geometric point and let $\bar s$ be the closed point of the normalization of
$S$ in $\bar \eta$. Let $G_\eta=\Gal(\bar\eta/\eta)$. Assume that $p\not\in
A$. Let $I<G$ be the inertia group and $P<I$ be the inertia group prime to
$A$. Then $G=G_\eta/P$ fits into a split short exact sequence
\eqref{e.Galois}, with $Z_A(1)$ being the Tate twist $\Z_A(1)$ of $\Z_A$,
$G_s=\Gal(\bar s/s)$, and $\chi$ the cyclotomic character. For a sheaf $M$
of $\Lambda$-modules on~$\eta$, we define $M(\pm 1)^\tau=M^{P}(\pm
1)^\tau\oplus {}^P M$. By Remark \ref{r.A}, this does not depend on the
choice of $A$. In particular, we can take $A$ to be either maximal (the set
of all primes except $p$) or minimal (the set of prime divisors of $m$).

Now let $Y$ be a scheme over $s$. A sheaf on $Y\ttimes_s BG$ is a pair
$(M,\alpha)$ where $M$ is a sheaf on $Y_{\bar s}$ and $\alpha$ is a
continuous action of $G$ on $M$, compatible with the action of $G$ on
$Y_{\bar s}$ via $G_s$. The construction $(M,\alpha)\mapsto (M,\alpha\circ
\rho_\sigma)$ extends to the topos $Y\ttimes_s BG$ and does not depend on
the choice of $\sigma$ up to isomorphism. Indeed, for $s\in \Z_A^\times$, we
have $\alpha(r)\colon (M,\alpha\circ \rho_{\sigma^s})\simto (M,\alpha\circ
\rho_\sigma)$, where $r=\sum_{n\ge 1}\binom{s}{n}t^{n-1}\in R^\times$. We
define $M(1)^\tau=(M,\alpha\circ \rho_\sigma)$ and
$M(-1)^\tau=(M,\alpha\circ \rho_\sigma^{-1})$. The latter has a more
canonical description by Remark \ref{r.crossed}: For $U$ a quasi-compact
object of the \'etale site of $Y_{\bar s}$, we have
\begin{equation}\label{e.crossed}
M(-1)^\tau(U)=Z^1_\cont(\Z_A(1),M(U)).
\end{equation}
We have a canonical map
\[M\to M(-1)^\tau\]
carrying $x\in M(U)$ to the crossed homomorphism $\sigma\mapsto \sigma x-x$.

\begin{definition}\label{d.Iwasawa}
For a sheaf $M$ of $\Lambda$-modules on $Y\ttimes_s \eta$, we define
\begin{equation}
M(\pm 1)^\tau=M^{P}(\pm 1)^\tau\oplus {}^P M.
\end{equation}
We will sometimes write $\tau M$ for $M(1)^\tau$. We get an auto-equivalence
\[\tau\colon \Shv(Y\ttimes_s \eta,\Lambda)\to \Shv(Y\ttimes_s \eta,\Lambda).\]
We consider the map
\begin{equation}\label{e.var}
\iota\colon M\to M(-1)^\tau
\end{equation}
given by the canonical map on $M^P$ and the identity on ${}^P M$.
\end{definition}

The kernel of \eqref{e.var} is $M^I$ and the cokernel is $M_I(-1)$. Here we
used the fact that if the action of $I$ on $M$ is trivial, then $M(\pm
1)^\tau\simeq M(\pm 1)$ (Remark \ref{r.t2}).

Let $p\colon Y\ttimes_s \eta\to Y$ be the projection. The functor $p^*$ on
categories of sheaves is fully faithful. For a sheaf $M$ on $Y\ttimes_s
\eta$, $p_*M$ can be identified with $M^I$.

\begin{lemma}\label{l.Rp}
For $M\in D(Y\ttimes_s \eta,\Lambda)$, $p^*Rp_*M$ is computed by
$\Cone(M\xrightarrow{\iota} M(-1)^\tau)[-1]$, with the adjunction $p^*Rp_*
M\to M$ given by the totalization of the map of bicomplexes
\begin{equation}\label{e.Rp}
\xymatrix{\cdots \ar[r] & 0\ar[r]\ar[d] & M\ar[r]^\iota\ar@{=}[d]\ar@{}[]+<0ex,3ex>*{0}& M(-1)^\tau\ar[d]\ar[r]\ar@{}[]+<0ex,3ex>*{1} & 0\ar[d]\ar[r] & \cdots\\
\cdots\ar[r] & 0\ar[r] &M\ar[r] & 0\ar[r] & 0 \ar[r] & \cdots.}
\end{equation}
\end{lemma}

\begin{proof}
Since the functor $M\mapsto M(-1)^\tau$ is exact, it suffices to show that
for $M$ an injective sheaf, the sequence $0\to M^I \to M\xrightarrow{\iota}
M(-1)^\tau\to 0$ is exact. We have already seen that $\Ker(\iota)=M^I$. It
remains to show that $M^P\to M^P(-1)^\tau$ is an epimorphism. This can be
identified with $C^0_\cont(K,M^P)\to Z^1_\cont(K,M^P)$, where $K=I/P\simeq
\Z_A(1)$, which is an epimorphism by the following lemma applied to the
injective sheaf $M^P$ on $Y\ttimes_s BG$.
\end{proof}

\begin{lemma}
Let $H\to G_s$ be an epimorphism of profinite groups of kernel $K$ and let
$q\colon Y\ttimes_s BH\to Y$ be the projection. Then for $M\in
\Shv(Y\ttimes_s BH,\Lambda)$, $q^*Rq_*M$ is computed by the complex
$C^*_\cont(K,M)$. Here $C^i_\cont(K,M)$ is the sheaf on $Y\ttimes_s BH$ such
that for any quasi-compact object $U$ of the \'etale site of $Y_{\bar s}$,
$C^i_\cont(K,M)(U)=C^i_\cont(K,M(U))$ is the $\Lambda$-module of continuous
$i$-cochains, with $H$-action given by
$(gf)(\xi_0,\dots,\xi_i)=g(f(g^{-1}\xi_0 g,\dots,g^{-1}\xi_i g))$ for $g\in
H$.
\end{lemma}

\begin{proof}
We have $M=\colim_J M^J$, where $J$ runs through open subgroups of $K$,
normal in $H$, and $C^i_\cont(K,M)=\colim_J C^i_\cont(K/J,M^J)$.  Since the
functors $C^i_\cont(K,-)$ are exact, it suffices to show that
$H^iC^*_\cont(K,M)=0$ for $i>0$ and $M$ injective. In this case, $M^J$ is
injective as sheaf on $Y\ttimes_s B(H/J)$. Thus we may assume $K$ finite. In
this case $C^*_\cont(K,M)=\cHom(L,M)$, where $L$ is the standard resolution
of $\Z$ by $\Z[K]$-modules, regarded as sheaves on $X\ttimes_s BK$.
\end{proof}

\begin{remark}\label{r.Rp}
By Lemma \ref{l.Rp}, $Rp_*$ has cohomological dimension $\le 1$, and
$R^1p_*M\simeq M_I(-1)$ for $M\in\Shv(Y\ttimes_s \eta,\Lambda)$.
\end{remark}

\subsection{Some quasi-periodic adjunctions}
Let $j\colon Y\ttimes_s \eta\to Y\ttimes_s S$ be the inclusion of the open
subtopos. Between categories of abelian sheaves, we have a sequence of
adjoint functors
\begin{equation}\label{e.adj}
 \Co \dashv j_! \dashv j^* \dashv j_*,
\end{equation}
where $\Co \cF=\Coker(\phi\colon p^*\cF_s\to \cF_\eta)$ for
$\cF=(\cF_s,\cF_\eta,\phi)$ as in Example \ref{ex.S}. This sequence cannot
be extended, as neither $\Co$ nor $j_*$ is exact (unless $Y$ is empty). We
will see however that the derived adjunction sequence can be extended into a
loop up to twist.

Note that $\Co$ admits a left derived functor
\[L\Co\colon D(Y\ttimes_s S,\Lambda)\to D(Y\ttimes_s \eta,\Lambda).\]
For each $M\in D(Y\ttimes_s S,\Lambda)$, $L\Co M$ is computed by $\Co M'$,
where $M'\to M$ is a  quasi-isomorphism and $M'$ is a complex of sheaves of
the form $(\cF_s,\cF_\eta,\phi)$ for which $\phi$ is a monomorphism. In
fact, $L\Co$ is the functor $\Phi$ of \cite[XIII 1.4.2]{SGA7II}. The
adjunctions \eqref{e.adj} induce adjunctions of derived functors $L\Co
\dashv j_! \dashv j^* \dashv Rj_*$ (see for example \cite[Theorem
14.4.5]{KS}).

\begin{theorem}\label{t.adj}
Between the derived categories $D(Y\ttimes_s \eta,\Lambda)$ and
$D(Y\ttimes_s S,\Lambda)$, we have a canonical adjunction $Rj_*\dashv \tau
L\Co$.
\end{theorem}

Therefore, we have a quasi-periodic sequence of adjoint
functors
\begin{equation}\label{e.adjDj} L\Co \dashv j_! \dashv j^* \dashv
Rj_*\dashv \tau L\Co.
\end{equation}

We start by constructing an adjunction for the projection $p\colon
Y\ttimes_s \eta \to Y$.

\begin{prop}\label{p.2.13}
Between derived categories $D(Y\ttimes_s \eta,\Lambda)$ and $D(Y,\Lambda)$,
we have a canonical adjunction $Rp_*\dashv p^*(1)[1]$.
\end{prop}

Thus we have a quasi-periodic sequence of adjoint
functors
\begin{equation}\label{e.ploop}
p^*\dashv Rp_*\dashv p^*(1)[1].
\end{equation}

\begin{proof}
The isomorphism $R^1p_*\Lambda\simeq \Lambda(-1)$ in Remark \ref{r.Rp}
induces $\Tr_\Lambda\colon Rp_*p^*\Lambda(1)[1]\to \Lambda$. We define the
co-unit of the adjunction to be the trace map
\[\Tr_L\colon Rp_* p^*L(1)[1]\xleftarrow{\sim} Rp_*\Lambda(1)[1]\otimes^L L\xrightarrow{L\otimes^L \Tr_\Lambda} L\]
for $L\in D(Y,\Lambda)$, where the first arrow is the projection formula
map, which is an isomorphism by Corollary \ref{c.pf} (a) or by Lemma
\ref{l.Rp}. By construction, $p^*\Tr_L$ is given by the totalization of the
morphism of double complexes
\[\xymatrix{\cdots\ar[r] & 0 \ar[r]\ar[d] &p^*L(1)\ar[r]^0\ar[d]\ar@{}[]+<0ex,3ex>*{-1} & p^*L\ar@{=}[d]\ar[r]\ar@{}[]+<0ex,3ex>*{0} & 0\ar[d]\ar[r] & \cdots\\
\cdots\ar[r] &0 \ar[r] & 0\ar[r]& p^*L\ar[r] & 0\ar[r] & \cdots.}
\]
Indeed, this is clear in the case $L=\Lambda$ and the general case follows.
We define the unit $\epsilon_M\colon M\to p^*Rp_* M(1)[1]$ for $M\in
D(Y\ttimes_s \eta,\Lambda)$ to be the totalization of the morphism of double
complexes
\[\xymatrix{\cdots\ar[r] & 0 \ar[r]\ar[d] &0\ar[r]\ar[d]\ar@{}[]+<0ex,3ex>*{0} & M\ar@{=}[d]\ar[r]\ar@{}[]+<0ex,3ex>*{1}& 0\ar[r] & \cdots\\
\cdots\ar[r] &0 \ar[r] &M(1)^\tau \ar[r] & M\ar[r]& 0\ar[r] & \cdots.}
\]
It follows that $p^*\Tr_L\circ \epsilon_{p^*L}=\id$ and
$\Tr_{Rp_*M(1)[1]}\circ Rp_*\epsilon_{M}(1)[1] =\id$. For the latter we
reduce to the easy case where $M$ is an injective sheaf.
\end{proof}

\begin{remark}
The proposition is a form of the Poincar\'e-Verdier duality for the inertia
group $I$, and can be compared with other Poincar\'e-Verdier dualities. For
$f$ a proper topological submersion of locally compact spaces (resp.\ proper
smooth morphism of schemes) of relative dimension $d$, we have adjunctions
\[f^*\dashv Rf_*\dashv Rf^!,\]
where $Rf^!\simeq (f^*-\otimes o_f)[d]$ (resp.\ $Rf^!\simeq f^*(d)[2d]$).
Here $o_f$ is the orientation sheaf.
\end{remark}

To construct the adjunction $Rj_*\dashv \tau L\Co$, we need (part (1) of)
the following.

\begin{lemma}\label{l.adj}
Let $j\colon U\to X$ be an open subtopos and let $i\colon V\to X$ be the
complementary closed subtopos. Let $\cC$ be a category. Let $F\colon \cC\to
D(X,\Lambda)$ and $G\colon D(X,\Lambda)\to \cC$ be functors.
\begin{enumerate}
\item A natural transformation $\id \to GF$ is an adjunction if and only
    if the compositions $\id\to GF\to (Gi_*)(i^* F)$ and $\id\to GF\to
    (GRj_*)(j^* F)$ are adjunctions.

\item A natural transformation $GF\to \id$ is an adjunction if and only if
    the compositions $(Gi_*)(Ri^! F)\to GF\to \id$ and $(G j_!)(j^*F)\to
    GF\to \id$ are adjunctions.
\end{enumerate}
\end{lemma}

That $\epsilon\colon \id \to GF$ is the unit of an adjunction (or, in short,
is an adjunction) means that the composite
\begin{equation}\label{e.comp1}
\Hom(FA,B)\to \Hom(GFA,GB)\xrightarrow{\epsilon_A} \Hom(A,GB)
\end{equation}
is a bijection for all $A\in \cC$ and $B\in D(X,\Lambda)$. That $\eta\colon
GF\to \id$ is the co-unit of an adjunction (or, in short, is an adjunction)
means that the composite
\begin{equation}\label{e.comp2}
\Hom(B,FA)\to \Hom(GB,GFA)\xrightarrow{\eta_A} \Hom(GB,A)
\end{equation}
is a bijection for all $A\in \cC$ and $B\in D(X,\Lambda)$.

\begin{proof}
The ``only if'' parts are trivial. For the ``if'' part of (1), the
assumption implies that \eqref{e.comp1} is an isomorphism for $B=Rj_*M$  and
for $B=i_*L$, which implies that the same holds for all $B$. For the ``if''
part of (2), the assumption implies that \eqref{e.comp2} is an isomorphism
for $B=j_!M$ and for $B=i_*L$, which implies that the same holds for all
$B$.
\end{proof}

Let $i\colon Y\simeq Y\ttimes_s s\to Y\ttimes_s S$ be the inclusion. Then
$i^*j_*\simeq p_*$.

\begin{proof}[Proof of Theorem \ref{t.adj}]
Consider the topos $Y\ttimes_s\eta^{[1]}$ of morphisms of $Y\ttimes_s\eta$
and the morphism of topoi $\lambda\colon Y\ttimes_s\eta^{[1]}\to Y\ttimes_s
S$ with $\lambda^*$ carrying $(\cF_s,\cF_\eta,\phi)$ to
$(p^*\cF_s,\cF_\eta,\phi)$. We have $L\Co\simeq L\Coker\circ \lambda^*$.  By
Lemma \ref{l.Rp}, for $M\in D(Y\ttimes_s \eta,\Lambda)$, $\lambda^*Rj_*M$ is
computed by the totalization of the diagram \eqref{e.Rp} considered as a
double complex in $\Shv(Y\ttimes_s\eta^{[1]},\Lambda)$, with the vertical
arrows representing $\phi$. Thus $L\Co Rj_*M$ is computed by the diagram
\eqref{e.Rp} considered as a triple complex in
$\Shv(Y\ttimes_s\eta,\Lambda)$, with the rows of the diagram numbered $-1$
and $0$. This gives an isomorphism $\id\simto \tau L\Co Rj_*$. It is easy to
check that the composition
\[
\id\simto \tau L\Co Rj_*\to (\tau L\Co \, i_*)(i^* Rj_*)\simeq (p^*(1)[1])Rp_*
\]
is the adjunction constructed in Proposition \ref{p.2.13} and
\[\id\simto \tau L\Co Rj_*\to (\tau L\Co Rj_*)(j^* Rj_*)\simeq
\id\circ \id
\]
is the trivial adjunction. Thus, by the above lemma, we get $Rj_*\dashv \tau
L\Co$.
\end{proof}

\begin{remark}
Between $\Shv(Y,\Lambda)$ and $\Shv(Y\ttimes_s S,\Lambda)$, we have a
sequence of adjoint functors
\[
\pi_?\dashv \pi^*\dashv \pi_*=i^* \dashv i_* \dashv i^!.
\]
Here $\pi\colon Y\ttimes_s S\to Y$ is the projection and
$\pi_?=R^2i^!(1)=R^1p_*j^*(1)$. Between $D(Y,\Lambda)$ and $D(Y\ttimes_s
S,\Lambda)$, we have adjunctions
\begin{equation}\label{e.adji}
Ri^!(1)[2]\dashv \pi^*\dashv \pi_*=i^* \dashv
i_* \dashv Ri^!.
\end{equation}
Since $p=\pi j$, compositions of the corresponding functors of the sequences
\eqref{e.adji} and \eqref{e.adjDj} collapse into \eqref{e.ploop}.

To construct the adjunction $Ri^!(1)[2]\dashv \pi^*$, we apply the last
assertion of Lemma \ref{l.adj}. The co-unit $(Ri^!(1)[2])\pi^*\simto \id$ is
given by the cycle class map. The compositions
\begin{gather*}
(\id(1)[2])(\id(-1)[-2])\simeq ((Ri^!(1)[2])i_*)(Ri^! \pi^*)\simto (Ri^!(1)[2])\pi^*\simto
\id,\\
(Rp_*(1)[1])(p^*)\simeq ((Ri^!(1)[2])j_!)(j^*\pi^*)\to (Ri^!(1)[2])\pi^*\simto
\id
\end{gather*}
are adjunctions.
\end{remark}

\begin{remark}
Let $M\in D(Y\ttimes_s S,\Lambda)$. We have distinguished triangles
\begin{gather}
\label{e.tr1} \pi^*\pi_* M\to M\to j_!L\Co M\to \pi^*\pi_* M[1],\\
\label{e.tr2} j_!j^* M\to M\to i_*i^* M\to j_!j^* M[1],\\
\label{e.tr3} i_*Ri^! M\to M \to Rj_*j^*M\to i^*Ri^! M[1],\\
\label{e.tr4} Rj_*\tau L\Co M \to M \to \pi^* Ri^! M(1)[2]\to Rj_*\tau L\Co M[1].
\end{gather}
Each triangle above is right adjoint of the preceding one (the first
triangle being right adjoint to the last one). Applying $j^*$ to
\eqref{e.tr1} and \eqref{e.tr4}, we obtain distinguished triangles
\begin{gather}
\label{e.treta1} p^*M_s \to M_\eta \xrightarrow{\can} L\Co M \to p^*M_s[1],\\
\label{e.treta2} L\Co M \xrightarrow{\var} M_\eta(-1)^\tau \to p^*Ri^!M[2]\to L\Co M[1].
\end{gather}
Here $\var$ is the \emph{variation} map. One can check that $\var\circ \can$
is the canonical map $\iota\colon M_\eta\to M_\eta(-1)^\tau$.

For $M$ tame (namely, $M^P\simeq M$), the morphism $\var$ corresponds via
\eqref{e.crossed} to morphisms $\var(\xi)\colon (L\Co M)_{\bar \eta} \to
M_{\bar \eta}$ for $\xi\in I$, which are the classical variation maps
\cite[XIII (1.4.3.1)]{SGA7II}.
\end{remark}

\subsection{Duality and $L\Co$}\label{ss.DLCo}

\begin{construction}\label{c.d}
Let $\cC$ and $\cD$ be categories equipped with dualities $D_\cC\colon
\cC^\op\to \cC$ and $D_\cD\colon \cD^\op\to \cD$ (Remark \ref{r.dual}). Let
$F,G\colon \cC\to \cD$ and $F',G'\colon \cD\to \cC$ be functors equipped
with adjunctions $F'\dashv F$ and $G\dashv G'$. A natural transformation
$FD_\cC\to D_{\cD} G^\op$ corresponds by adjunction to $G'^\op D_{\cD}^\op
\to D_\cC^\op F'$, which corresponds by taking opposites to a natural
transformation $D_\cC F'^\op\to G' D_\cD$.
\end{construction}

Let $f\colon Y\to s$ be a separated morphism of finite type. We let
$K_{Y\ttimes_s \eta}=p^*Rf^!\Lambda_s\simeq
R(f\ttimes_s\id_\eta)^!\Lambda_\eta$ and $K_{Y\ttimes_s
S}=\pi^*Rf^!\Lambda_s\simeq R(f\ttimes_s\id_S)^!\Lambda_S$ (by Corollary
\ref{c.Kunnup}).

\begin{construction}
We construct a natural transformation
\begin{equation}\label{e.nt}
D_{Y\ttimes_s \eta}(L\Co)^\op\to
\tau L\Co D_{Y\ttimes_s S}
\end{equation}
of functors $D(Y\ttimes_s S,\Lambda)^\op\to D(Y\ttimes_s \eta,\Lambda)$, by
applying Construction \ref{c.d} and Theorem \ref{t.adj} to the natural
transformation \eqref{e.DD}
\begin{equation}\label{e.gamma}
\gamma\colon j_! D_{Y\ttimes_s \eta} \to
D_{Y\ttimes_s S} (Rj_*)^\op
\end{equation}
adjoint to the canonical isomorphism $D_{Y\ttimes_s \eta}\simeq
j^*D_{Y\ttimes_s S} (Rj_*)^\op$.
\end{construction}

By construction \eqref{e.nt} is given by
\begin{multline*}
\Hom(L,D(L\Co) M)\simeq \Hom((L\Co) M,DL)\simeq \Hom(M,j_!DL)\\
\xrightarrow{\gamma_L}
\Hom(M,DRj_*L)\simeq \Hom(Rj_*L,DM)\simeq \Hom(L,\tau (L\Co) DM).
\end{multline*}

\begin{theorem}[= 0.1 (2)]\label{t.LPhi}
The natural transformations \eqref{e.nt} and \eqref{e.gamma} are
isomorphisms.
\end{theorem}

\begin{proof}
It suffices to show that $\gamma_L$ is an isomorphism for every $L\in
D(Y\ttimes_s \eta,\Lambda)$. Since the source and target of $\gamma$ both
carry coproducts to products, we may assume $L\in D^-$. Since $Rj_*$ has
cohomological amplitude contained in $[0,1]$ and both $D_{Y\ttimes_s S}$ and
$D_{Y\ttimes_s \eta}$ have cohomological amplitude $\ge -2\dim(Y)$, we may
assume $L$ is a sheaf. We may further assume that $L$ is a coproduct of
sheaves of the form $u_! \Lambda$ with $u$ \'etale and affine. The case of
$u_!\Lambda$ is a special case of Proposition \ref{p.bidual}.
\end{proof}

\begin{remark}\label{r.trcomp}
We let $K_Y=Rf^!\Lambda_s(-1)[-2]$. Similarly to the above, we have natural
isomorphisms
\begin{gather}\label{e.trcomp}
i^*D_{Y\ttimes_s S}\simto D_{Y}(Ri^!)^\op,\\
\label{e.Da} D_{Y\ttimes_s S}(\pi^*)^\op \simto \pi^*(1)[2]D_Y
\end{gather}
on the unbounded derived categories, adjoint to each other. Moreover,
\eqref{e.Da} is the inverse of the K\"unneth formula map. Via the
isomorphisms \eqref{e.gamma} and \eqref{e.trcomp}, the distinguished
triangles \eqref{e.tr2} and \eqref{e.tr3} are compatible. Similarly, via the
isomorphisms \eqref{e.nt} through \eqref{e.Da}, the distinguished triangles
\eqref{e.tr1} and \eqref{e.tr4} are compatible.
\end{remark}

\section{Nearby and vanishing cycles over Henselian discrete valuation
rings}\label{s.2+}

In this section we assume as in Section \ref{ss.Iwafiber} that $S$ is the
spectrum of a Henselian discrete valuation ring, of generic point $\eta$ and
closed point $s$. Let $\Lambda$ be a Noetherian ring satisfying $m\Lambda=0$
for some integer $m$ invertible on $S$.

Let $X\to S$ be a morphism of schemes. Consider the morphisms of topoi
\[X\xrightarrow{\Psi} X\atimes_S S \xleftarrow{\overleftarrow{i}} X_s\ttimes_s S.\]
We call $\Psi^s=\overleftarrow{i}^*\Psi_*$ the \emph{sliced} nearby cycle
functor. As explained in \cite[(1.2.9)]{IZ}, this is the functor $\Psi$ of
\cite[XIII 1.3.3]{SGA7II}. Between derived categories equipped with the
symmetric monoidal structures given by derived tensor products, the functor
$\overleftarrow{i}^*$ is a symmetric monoidal functor, and $R\Psi_*$ is a
right-lax symmetric monoidal functor (see for example \cite[Construction
3.7]{IZQ}). The composite $R\Psi^s$ is a right-lax symmetric monoidal
functor. By adjunction (see for example \cite[XVII D\'efinition
12.2.3]{ILO}), we get a morphism natural in $A,K\in D(X,\Lambda)$
\begin{equation}\label{e.PsiHom}
R\Psi^s R\cHom(A,K)\to R\cHom(R\Psi^s A,R\Psi^s K).
\end{equation}

Let $f\colon X\to Y$ be a morphism of schemes over $S$, $f$ separated of
finite type. Via Construction \ref{c.adj}, the natural transformation
\cite[XIII (2.1.7.3)]{SGA7II}
\begin{equation}
R(f_s\ttimes_s \id_S)_!R\Psi^s_{X} \to R \Psi^s_{Y} Rf_!
\end{equation}
corresponds to a natural transformation
\begin{equation}\label{e.Psi!0}
R\Psi^s_{X} Rf^! \to R(f_s\ttimes_s \id_S)^! R \Psi^s_{Y}.
\end{equation}
See \eqref{e.Psis3} and \eqref{e.Psis4} for generalizations.

We fix $K_S\in D^b_c(S,\Lambda)$, not necessarily dualizing. For $a\colon
X\to S$ separated of finite type, we define $K_X\colonequals Ra^!K_S$ and
$K_{X_s\ttimes_s S} \colonequals R(f_s\ttimes_s \id_S)^!K_S$. Applying
\eqref{e.Psi!0} to $K_S$, we get $R\Psi^s K_X\to K_{X_s\ttimes_s S}$.
Composing with \eqref{e.PsiHom}, we get a natural transformation
\begin{equation}\label{e.PsiD}
R\Psi^s D_X \to D_{X_s\ttimes_s S} R
\Psi^s.
\end{equation}

\begin{theorem}\label{t.PsiS}
The canonical map $R\Psi^s D_X L\to  D_{X_s\ttimes_s S} R \Psi^s L$
\eqref{e.PsiD} is an isomorphism for $L\in D^-_c(X,\Lambda)$.
\end{theorem}

This confirms a prediction of Deligne \cite{Deligne}. Since $j^*
R\Psi^s\simeq R\Psi^s_\eta j_0^*$, where $j_0\colon X_\eta\to X$ and
$j\colon X_s\ttimes_s \eta\to X_s\ttimes_s S$ are the inclusions, the
theorem extends Gabber's theorem on duality for $R\Psi_\eta^s$.

\begin{cor}[Gabber]\label{c.Gabber}
The canonical map $R\Psi^s_\eta D_{X_\eta} L\to D_{X_s\ttimes_s \eta}
R\Psi^s_\eta L$ is an isomorphism for $L\in D^-_c(X_\eta,\Lambda)$. Here
$K_{X_s\ttimes_s \eta}\colonequals R(f_s\ttimes_s \id_\eta)^!K_\eta$.
\end{cor}

\begin{remark}
In \cite[Th\'eor\`eme 4.2]{IllusieAutour}, Corollary \ref{c.Gabber} is
proved for $K_\eta=\Lambda_\eta$ and $L\in D_{\cft}(X_\eta,\Lambda)$. The
case $K_\eta=\Lambda_\eta$ and $L\in D^-_c(X_\eta,\Lambda)$ follows. Indeed,
since $R\Psi^s_\eta$ has cohomological amplitude contained in
$[0,\dim(X_s)]$, $D_{X_\eta}$ has cohomological amplitude $\ge
-2\dim(X_\eta)$, and $D_{X_s\ttimes_s \eta}$ has cohomological amplitude
$\ge -2\dim(X_s)$, we may assume $L\in \Shv_c$. We then reduce to the case
$L=j_!\Lambda$ for $j$ \'etale of finite type. In this case $L\in
D_{\cft}(X_\eta,\Lambda)$.
\end{remark}

Combining Theorem \ref{t.PsiS} with Theorem \ref{t.LPhi}, we obtain duality
for $\Phi^s=L\Co\circ R\Psi^s$ (Corollary \ref{c.Beilinson}), which is (at
least for $S$ strictly local) a theorem of Beilinson \cite[2.3]{Beilinson}.
Theorem \ref{t.PsiS} thus encodes duality for both nearby cycles and
vanishing cycles and provides a new proof of Beilinson's theorem.

In the rest of this section, we show that Theorem \ref{t.PsiS} follows from
our results over general bases. The case where $S$ is excellent of Theorem
\ref{t.PsiS} (and Corollary \ref{c.Gabber}) follows from Theorem
\ref{t.Psis} by Example \ref{e.Psigood} (2), which extends to the unbounded
derived category since $R\Psi$ has finite cohomological dimension in our
case. The general case of Corollary \ref{c.Gabber} follows from the
excellent case by Lemma \ref{l.local} (b) and the universal local acyclicity
of $\hat \eta\to \eta$, because $R\Psi^s_\eta$ commutes with the base change
$\hat S\to S$. Here $\hat S$ denotes the completion of $S$ at $s$ and $\hat
\eta$ denotes the generic point of $\hat S$. We will deduce Theorem
\ref{t.PsiS} from Corollary \ref{c.Gabber}, after some preliminaries on the
vanishing topos $X\atimes_S S$.

Let $X$ be a scheme over $S$. The topos $X\atimes_S S$ is glued from the
pieces $X_s$, $X_s\ttimes_s \eta$, and $X_\eta$. Consider the diagram of
topoi
\begin{equation}\label{e.Psixy}
\xymatrix{X_\eta\ar[r]^{\Psi_\eta}\ar[d]_{j_0}\ar@{}[rd]|{\Leftrightarrow} & X\atimes_S \eta\ar[d]_{\overleftarrow{j}}\ar@{}[rd]|{\Leftrightarrow} & X_s\ttimes_s \eta\ar[d]^j\ar[l]_{\overleftarrow{i}_\eta}\\
X\ar[r]^\Psi & X\atimes_S S\ar@{}[d]|{\Leftrightarrow} & X_s\ttimes_s S\ar[l]_{\overleftarrow{i}} \\
& X_s\ar[ul]_{i_0} \ar[ur]_i.}
\end{equation}
Here $\Psi_\eta$, $\overleftarrow{j}$, $j_0$, $j$ are open embeddings, and
$i$, $\overleftarrow{i}$, $\overleftarrow{i}_\eta$, $i_0$ are closed
embeddings. We have $\Psi^s=\overleftarrow{i}^*\Psi_* $ and
$\Psi^s_\eta=\overleftarrow{i}_\eta^*\Psi_{\eta*}$ (denoted respectively by
$\Psi$ and $\Psi_\eta$ in \cite[XIII 1.3]{SGA7II}).

\begin{lemma}\label{l.three}
We have
\begin{gather}
\label{e.3.4.1} R\Psi^s i_{0*}\simto i_*,\\
\label{e.3.4.2} R\Psi^s j_{0!}\xleftarrow{\sim}
j_!R\Psi^s_\eta,\\
\label{e.3.4.3} R\Psi^s Rj_{0*}\simto Rj_*R\Psi^s_\eta.
\end{gather}
Here the functors are between unbounded derived categories.
\end{lemma}

\begin{proof}
This follows from
\begin{gather}\label{e.Psi1}
R \Psi i_{0*}\simeq (\overleftarrow{i} i)_*,\quad R
\Psi j_{0!}\xleftarrow{\sim} \overleftarrow{j}_!R\Psi_\eta,\quad R\Psi Rj_{0*}\simeq R\overleftarrow{j}_*R
\Psi_\eta,\\
\overleftarrow{i}^*\overleftarrow{i}_*\simto \id,\quad \overleftarrow{i}^*\overleftarrow{j}_!\xleftarrow{\sim} j_!\overleftarrow{i}_\eta^*,\quad
\overleftarrow{i}^*R\overleftarrow{j}_*\simto Rj_*\overleftarrow{i}_\eta^*.\label{e.Psi2}
\end{gather}
To show that the middle arrow in \eqref{e.Psi1} is an isomorphism, we apply
the functors $\overleftarrow{j}^*$ and $(\overleftarrow{i}i)^*$ and use the
fact that $\Psi^*R\Psi_*\simto \id$ (Remark \ref{r.Psiu} below). All the
other isomorphisms except the last one of \eqref{e.Psi2} are trivial. The
last isomorphism of \eqref{e.Psi2} follows from Corollary \ref{c.obc}, but
we give a direct proof here. The last arrow of \eqref{e.Psi2} being
trivially an isomorphism on $\overleftarrow{i}_{\eta*}M$, it remains to show
\begin{equation}\label{e.Psinontriv}
\overleftarrow{i}^* R\overleftarrow{j}_*\Psi_{\eta!}=0.
\end{equation}
Let $y\to X_s$ be a geometric point. Note that
\begin{equation}\label{e.Psistalk}
(R\overleftarrow{j}_* L)_{y}\simeq
R\Gamma(X_{(y)}\atimes_{S'} \eta',L_{(y)})\simeq R\Gamma(\eta',L_{y})
\end{equation}
for $L\in D^+(X\atimes_{S} \eta,\Lambda)$, where $S'=S_{(s)}$ is the strict
Henselization of $S$, $\eta'$ is the generic point of~$S'$, and $L_{(y)}$
and $L_y$ are the restrictions of $L$ to $X_{(y)}\atimes_{S'} \eta'$ and
$y\atimes_{S'} \eta'\simeq \eta'$, respectively. Indeed, if $p\colon
X_{(y)}\atimes_{S'} \eta'\to \eta'$ denotes the projection, then $p_*$ is
isomorphic to the restriction by \cite[XI Corollaire 2.3.1]{ILO}. It follows
that $\overleftarrow{j}_*$ has cohomological dimension $\le 1$ and
consequently \eqref{e.Psistalk} holds for $L$ unbounded.
\eqref{e.Psinontriv} follows.
\end{proof}

\begin{proof}[Proof of Theorem \ref{t.PsiS}]
For $L=i_{0*}M$, the map can be identified with the isomorphism
\[R\Psi^s
D_X i_{0*} M\simeq R\Psi^s i_{0*}D_{X_s} M\xrightarrow[\sim]{\eqref{e.3.4.1}} i_*D_{X_s} M\simeq
D_{X_s\ttimes_s S}i_{*} M\xleftarrow[\sim]{\eqref{e.3.4.1}}D_{X_s\ttimes_s S}R\Psi^s i_{0*}M.
\]
For $L=j_{0!}N$, the map can be identified with
\begin{multline*}
R\Psi^s D_X j_{0!}N\simeq R\Psi^s Rj_{0*}D_{X_\eta} N\xrightarrow[\sim]{\eqref{e.3.4.3}}
Rj_*R\Psi^s_\eta D_{X_\eta} N\\\xrightarrow[\sim]{a} Rj_* D_{X_s\ttimes_s \eta} R\Psi^s_\eta
N\xrightarrow[\sim]{\eqref{e.1.13.2}} D_{X_s\ttimes_s S} j_!R\Psi^s_\eta N\xrightarrow[\sim]{\eqref{e.3.4.2}} D_{X_s\ttimes_s S} R\Psi^s j_{0!}N,
\end{multline*}
where $a$ is an isomorphism by Corollary \ref{c.Gabber}.
\end{proof}

\begin{remark}
Let $X$ be a scheme over $S$. The map $Ri_0^!\to Ri^!R\Psi^s$ induced from
\eqref{e.3.4.1} is an isomorphism. Indeed, by the lemma, this holds on
$i_{0*}M$ and $Rj_{0*} N$. Thus, for $L\in D(X,\Lambda)$, the distinguished
triangles \eqref{e.treta1} and \eqref{e.treta2} applied to $R\Psi^s L$ are
\begin{gather}
\label{e.final1} p^*L_s \to R\Psi^s(L_\eta)\xrightarrow{\can} \Phi^s(L)\to p^*L_s[1]\\
\label{e.final2} \Phi^s(L) \xrightarrow{\var} R\Psi^s(L_\eta)(-1)^\tau \to p^*Ri_0^! L[2]\to \Phi^s(L)[1].
\end{gather}
A special case of \eqref{e.final2} was given in \cite[Th\'eor\`eme
3.3]{IllPerv}.

By Theorem \ref{t.PsiS} and Remark \ref{r.trcomp}, when $X$ is separated of
finite type over $S$, \eqref{e.final1} and \eqref{e.final2} are compatible
with each other via duality.
\end{remark}

\begin{remark}\label{r.nodual}
There is no good duality on $X\atimes_S S$ unless $X=X_\eta\coprod X_s$ as
topological spaces or $\Lambda=0$. Indeed, by \eqref{e.Psinontriv}, we have
$(\overleftarrow{j}\Psi_\eta)_!\simto R\overleftarrow{j}_*\Psi_{\eta!}$. Now
if there are dualities that swap $J_!$ and $RJ_*$ for the open immersions
$J=\overleftarrow{j}$ and $J=\Psi_\eta$, then
$R(\overleftarrow{j}\Psi_\eta)_*\simeq \overleftarrow{j}_!R\Psi_{\eta*}$,
which implies $i_0^*Rj_{0*}\simeq i_0^*\Psi^*R\Psi_*Rj_{0*}\simeq (\Psi
i_0)^*R(\overleftarrow{j}\Psi_\eta)_*=0$.
\end{remark}

\section{Sliced nearby cycles over general bases}\label{s.3}
In this section, we study the sliced nearby cycle functor over general bases
and base change. The main result is Theorem \ref{t.good}, the analogue of
Orgogozo's theorem for $!$-pullback. The proof of Theorem \ref{t.Psisc} will
be completed in Section \ref{s.3+}.

\subsection{Definition}
In this subsection we define the sliced nearby cycle functor and discuss
some basic properties.

Let $f\colon X\to S$ be a morphism of schemes. For a point $s\to S$ with
values in a field, we consider the \emph{slice} $X_s\atimes_S S$ of the
vanishing topos, where $X_s=X\times_S s$, and the morphisms of topoi
\[X\xrightarrow{\Psi_f} X\atimes_S S \xleftarrow{i\atimes_S \id_S} X_s\atimes_S S.\]
Let $\Lambda$ be a commutative ring. We define the \emph{sliced} nearby
cycle functor
\[R\Psi^s=R\Psi^s_f\colon D(X,\Lambda)\to D(X_s\atimes_S S,\Lambda)\]
to be $(i\atimes_S \id_S)^*R\Psi_f$. The most essential case is when $s$ is
a point of $S$. However, we will also need the case where $s$ is a geometric
point.

Note that $R\Psi_f^s$ is the right derived functor of
$\Psi_f^s=R^0\Psi_f^s$, with $R\Psi_f^s L$ computed by $\Psi_f^s L'$, where
$L'\to L$ is a quasi-isomorphism with $L'$ homotopically injective. The
functor $(i\atimes_S \id_S)^*$ is a symmetric monoidal functor, and
$R\Psi_f$ is a right-lax symmetric monoidal functor (see for example
\cite[Construction 3.7]{IZQ}). The composite $R\Psi^s_f$ is a right-lax
symmetric monoidal functor. By adjunction (see for example \cite[XVII
D\'efinition 12.2.3]{ILO}), we get a morphism natural in $L,K\in
D(X,\Lambda)$
\begin{equation}\label{e.PsiHomS}
R\Psi^s_f R\cHom(L,K)\to R\cHom(R\Psi^s_f L,R\Psi^s_f K).
\end{equation}

We show that the slice is in fact a fiber product of topoi. Let $S_{(s)}$ be
the Henselization of $S$ at $s$ and let $i\colon s_0\to S_{(s)}$ be the
inclusion of the closed point (so that $s_0$ is the spectrum of the
separable closure in $s$ of the residue field of $S$ at the image of $s$).
The morphism of topoi $i_\et$ admits a left adjoint $\pi\colon
(S_{(s)})_\et\to (s_0)_\et$ (Example \ref{ex.S}).

\begin{lemma}\label{l.id}
We have equivalences of topoi $X_s\ttimes_{s_0} S_{(s)} \simeq
X_s\atimes_{S_{(s)}}S_{(s)}\simto X_s\atimes_S S$.
\end{lemma}

\begin{proof}
The first equivalence follows from Example \ref{ex.S}. The second
equivalence follows from \cite[XI Proposition 1.11]{ILO}.
\end{proof}

In the sequel we will often identify the equivalent topoi in the lemma. For
$L\in D^+(X,\Lambda)$, $R\Psi_f^s L$ can be identified with
$R\Psi_{f_{(s)}}^s (L|_{X_{(s)}})$, where $f_{(s)}\colon X_{(s)}\to S_{(s)}$
is the base change of $f$.

\begin{remark}\label{r.local}
Assume $s$ is a geometric point and $L\in D^+(X,\Lambda)$. One can further
restrict the sliced nearby cycles $R\Psi^s L$ on $X_s\ttimes S_{(s)}$ in two
ways. (a) For any geometric point $t\to S_{(s)}$, the restriction  of
$R\Psi^s L$ to the \emph{shred} $X_s\ttimes t$ (which is called ``slice'' in
\cite[Section 6]{Org} and \cite[Section 1.3]{IZ}) is computed by Orgogozo's
shredded nearby cycle functor
\[R\Psi^s_t=(i^s)^*R(j^s_t)_*\colon D^+(X_{(t)},\Lambda)\to D^+(X_s,\Lambda).\]
Here $i^s\colon X_s\to X_{(s)}$, $j^s_t\colon X_{(t)}\to X_{(s)}$. (b) For
any geometric point $x\to X_s$, the restriction of $R\Psi^s L$ to the
\emph{local section} $x\ttimes S_{(s)}$ is computed by the localized
pushforward $Rf_{(x)*}(L|_{X_{(x)}})$ \cite[(1.12.6)]{IZ}, where
$f_{(x)}\colon X_{(x)}\to S_{(s)}$ is the strict localization of $f$ at~$x$.
\end{remark}

\begin{remark}\label{r.Psiu}
The morphism of topoi $\Psi_f\colon X\to X\atimes_S S$ is an embedding. In
other words the adjunction map $\Psi_f^* \Psi_{f*}\to \id$ is an
isomorphism. This follows from the identification of $\Psi_f^*$ with
$p_{1*}$ \cite[XI Proposition 4.4]{ILO} or the computation by shreds:
$\Psi^s_s=(i^s)^*$.

It follows that the adjunction map $\Psi_f^* R\Psi_f\to \id$ is an
isomorphism. For $L,K\in D(X,\Lambda)$, we have an isomorphism
\[A_f\colon R\Psi_f R\cHom(L,K)\simto R\cHom(R\Psi_f L,R\Psi_f K)\]
given by
\[\Hom(\Psi_f^* M\otimes^L L,K)\simto \Hom(\Psi_f^* (M\otimes^L R\Psi_f L), K)\simeq \Hom(M\otimes^L R\Psi_f L,R\Psi_f K)\]
for $M\in D(X\atimes_S S,\Lambda)$. The map \eqref{e.PsiHomS} is the
composite
\[R\Psi_f^sR\cHom(L,K)\xrightarrow[\sim]{A_f} (i\atimes_S \id_S)^*R\cHom(R\Psi_f L,R\Psi_f K)\to R\cHom(R\Psi_f^s L,R\Psi_f^s K).\]
\end{remark}

\begin{lemma}\label{l.Psi1}
Let $X\xrightarrow{f} T\xrightarrow{g} S$ be morphisms of schemes. For any
geometric point $t\to T$, we have a natural isomorphism
\begin{equation}\label{e.Psi0}
R(\Psi_{gf}^t L)|_{X_t\ttimes S_{(t)}}\simto R(\id \ttimes
    g_{(t)})_*R\Psi_{f}^t L
\end{equation}
in $L\in D^+(X,\Lambda)$, where $X_t=X\times_T t$, $g_{(t)}\colon T_{(t)}\to
S_{(t)}$ is the strict localization of $g$ at $t$, and $t$ in $\Psi_{gf}^t$
denotes the composition $t\to T\xrightarrow{g} S$.
\end{lemma}

Consider the diagram of topoi
\[\xymatrix{X\ar[r]^{\Psi_{f}} \ar[rd]_{\Psi_{gf}}^\Leftrightarrow & X\atimes_T T\ar[d]^{\id_X\atimes g}\ar@{}[rd]|{\Leftrightarrow} & \ar[l]_{i\atimes \id_T}X_t\atimes_T T\ar[d]^{\id_{X_t}\atimes g}\\
& X\atimes_S S & X_t\atimes_S S.\ar[l]_{i\atimes \id_S}}
\]
Via Lemma \ref{l.id}, \eqref{e.Psi0} is
\[(i\atimes \id_S)^* R\Psi_{gf}\simeq (i\atimes \id_S)^*R(\id_X \atimes g)_*R\Psi_f \simto  R(\id_{X_t}\atimes g)_*(i\atimes \id_T)^*R\Psi_f,\]
where the second isomorphism is trivial base change \cite[Lemma A.9]{IZ}.

\subsection{Review of local acyclicity}
\begin{definition}\label{d.slc}
Let $f\colon X\to S$ be a morphism of schemes and let $L\in D(X,\Lambda)$.
\begin{enumerate}
\item Following \cite[Th.\ finitude, D\'efinition 2.12]{SGA4d}, we say
    that $(f,L)$ is \emph{locally acyclic} if the canonical map
    $\alpha_L\colon L_x\to R\Gamma(X_{(x)t},L)$ is an isomorphism for
    every geometric point $x\to X$ and every algebraic geometric point
    $t\to S_{(x)}$. Here $X_{(x)t}\colonequals X_{(x)}\times_{S_{(x)}} t$
    denotes the Milnor fiber.

\item Following \cite[Th.\ finitude, App., 2.9]{SGA4d}, we say that
    $(f,L)$ is \emph{strongly locally acyclic} if $L\otimes^L M$ is
    locally acyclic for every $\Lambda$-module $M$.
\end{enumerate}
For $(*)$ denoting one of the above properties, we say that $(f,L)$ is
universally $(*)$ if $(*)$ holds after every base change $g\colon T\to S$.
\end{definition}

\begin{remark}\label{r.lc}
Assume $L\in D^+$. Then $(f,L)$ is locally acyclic if and only if $\Phi_f
L=0$ (namely, the canonical map $p_1^*L\to R\Psi_f L$ is an isomorphism) and
$R\Psi_f L$ commutes with locally quasi-finite base change
(\cite[Proposition 2.7]{Saito}, \cite[Example 1.7 (b)]{IZ}, see also
\cite[XV Corollaire 1.17]{SGA4}).
\end{remark}

The following extends \cite[Proposition 7.6.2, Theorem 7.6.9
(i)$\Leftrightarrow$(iii)]{Fu}.

\begin{lemma}\label{l.slc}
Assume $(f,L)$ locally acyclic and that either of the following conditions
holds:
\begin{enumerate}
\item The Milnor fibers $X_{(x)t}$ have finite cohomological dimension.
\item $L$ has tor-amplitude $\ge n$ for some integer $n$ and every
    $\Lambda$-module of finite presentation can be embedded into a free
    $\Lambda$-module.
\end{enumerate}
Then $(f,L)$ is strongly locally acyclic. Moreover, under condition (1), for
any $N\in D(X,\Lambda)$ of locally constant cohomology sheaves,
$(f,L\otimes^L N)$ is locally acyclic.
\end{lemma}

Condition (1) is satisfied if $\Lambda$ is torsion and $f$ is locally of
finite type, by limit arguments \cite[Corollary 7.5.7]{Fu}.

\begin{proof}
Note that $\alpha_{L\otimes^L M}$ is the composition
\[L_x\otimes^L M \xrightarrow{\alpha_L\otimes^L M} R\Gamma(X_{(x)t}, L)\otimes^L M \xrightarrow{\beta} R\Gamma(X_{(x)t},L\otimes^L M).\]
In case (1), the projection formula map $\beta$ is an isomorphism
\cite[Lemma A.8]{IZ}. The same formula implies the last statement of the
lemma, as $N|_{X_{(x)}}\simeq \pi^*N_x$, where $\pi\colon (X_{(x)})_\et\to
\pt$. In case (2), we show by induction on $m$ that
$\Cone(\alpha_{L\otimes^L M})\in D^{\ge m}$. This is trivial for $m=n-1$.
For the general case, we may assume $M$ of finite presentation. Choose a
short exact sequence $0\to M\to F\to M'\to 0$ with $F$ free. Since
$\alpha_{L\otimes^L F}$ is an isomorphism, $\Cone(\alpha_{L\otimes^L
M})\simeq \Cone(\alpha_{L\otimes^L M'})[-1]$ and it suffices to apply the
induction hypothesis to $M'$.
\end{proof}

\subsection{Functoriality in $X$}\label{s.4.3}
In the sequel we will often omit the notation $R$ for right derived
functors.

In this subsection, we study the functoriality of the sliced nearby cycle
functor in $X$ by giving analogues of the natural transformations \cite[XIII
(2.1.7.1)--(2.1.7.4)]{SGA7II} over general bases. These will be used many
times in the sequel, including in the definition of the map \eqref{e.A} in
Theorem \ref{t.Psisc}.

Let $X\xrightarrow{f} Y\xrightarrow{b} S$ be morphisms of schemes.  Consider
the diagram of topoi
\begin{equation}\label{e.4.3}
\xymatrix{X\ar[r]^{\Psi_{bf}} \ar[d]_f\ar@{}[rd]|{\Leftrightarrow} & X\atimes_S S\ar[d]^{f\atimes_S \id_S}\ar@{}[rd]|{\Leftrightarrow} & \ar[l]_{i_X}X_s\atimes_S S\ar[d]^{f_s\atimes_S \id_S}\\
Y\ar[r]^{\Psi_b} & Y\atimes_S S & Y_s\atimes_S S,\ar[l]_{i_Y}}
\end{equation}
where the square on the left is Cartesian. The base change maps
\begin{gather}
\label{e.Psis01}(f\atimes\id_S)^*\Psi_b \to \Psi_{bf}f^*,\\
\label{e.Psis02}i_Y^*(f\atimes\id_S)_*\to
(f_s\atimes \id_S)_* i_X^*
\end{gather}
induce
\begin{gather}
\label{e.Psis1} (f_s\ttimes_{s_0}\id_{S_{(s)}})^* \Psi_{b}^s \to \Psi_{bf}^s f^*,\\
\label{e.Psis2} \Psi_b^sf_*\to (f_s\ttimes_{s_0} \id_{S_{(s)}})_*\Psi_{bf}^s.
\end{gather}

\begin{remark}\leavevmode\label{r.Psis0}
\begin{enumerate}
\item If $f$ is \'etale, or if $\Lambda$ is torsion and $f$ is locally
    acyclic locally of finite type, then \eqref{e.Psis01} and
    \eqref{e.Psis1} are isomorphisms, by \cite[Remark 1.20]{IZ} (which
    extends to the unbounded derived category by the finiteness of the
    cohomological dimension of $\Psi$ \cite[Proposition 3.1]{Org}).

\item If $f$ is integral, or if $\Lambda$ is torsion and $f$ is proper,
    then \eqref{e.Psis02} and \eqref{e.Psis2} are isomorphisms by
    Proposition \ref{p.Org} and Remark \ref{r.Org0}.
\end{enumerate}
\end{remark}

Assume $\Lambda$ of torsion. Assume $X$ and $Y$ coherent and $f$ separated
and of finite type. We will define maps
\begin{gather}
\label{e.Psis3} (f_s\ttimes_{s_0} \id_{S_{(s)}})_!\Psi_{bf}^s\to \Psi_b^sf_!,\\
\label{e.Psis4} \Psi_{bf}^s f^!\to (f_s\ttimes_{s_0}\id_{S_{(s)}})^! \Psi_{b}^s,
\end{gather}
adjoint to each other via Construction \ref{c.adj}. The precise definitions
will be given in \eqref{e.Psis3d} and \eqref{e.Psis4d}. For $f$ proper,
\eqref{e.Psis3} is the inverse of \eqref{e.Psis2}. For $f$ an open
immersion, \eqref{e.Psis4} is the inverse of \eqref{e.Psis1}. For general
$f$, we take a compactification of $f$. The task of checking that the
construction does not depend on the choice of the compactification can be
conveniently divided into two with the help of oriented $!$-pushforward and
$!$-pullback functors, that we will now discuss.

\begin{construction}\label{c.o!}
We construct, for morphisms of schemes $X\xrightarrow{f}Y\xrightarrow{b} S$
with $X$ and $Y$ coherent and $f$ separated of finite type, a functor
\[R(f\atimes_S \id_S)_!\colon D(X\atimes_S S,\Lambda)\to D(Y\atimes_S S,\Lambda),\]
isomorphic to $R(f\atimes_S \id_S)_*$ for $f$ proper and left adjoint to
$(f\atimes_S \id_S)^*$ for $f$ an open immersion, and compatible with
composition. For this, we apply Construction \ref{con.!} to $f_\et$ and to
the diagram $Y_\et\xrightarrow{b_\et} S_\et\xleftarrow{p_1}S\atimes_S S$. By
\cite[XI Lemme 2.5]{ILO}, $p_1$ is a locally coherent morphism of locally
coherent topoi. Moreover, the source and target are of the desired form by
\cite[XI Proposition 4.2]{ILO} (cf.\ Remark \ref{r.Org0} (3)).

Consider a Cartesian square
\[
\xymatrix{X'\ar[r]^{h'}\ar[d]_{f'} & X\ar[d]^f\\
Y'
\ar[r]^h & Y}
\]
of coherent schemes over $S$ with $f$ separated of finite type. We apply
Constructions \ref{con.!} (1), (2) (to $g=\Psi\colon S_\et\to S\atimes_S
S$), (3), and obtain isomorphisms
\begin{gather}
\label{e.oshr1}(h\atimes_S \id_S)^*R(f\atimes_S \id_S)_!\simeq R(f'\atimes_S
    \id_S)_!(h'\atimes_S \id_S)^*,\\
\label{e.oshr2}\Psi_b^* R(f\atimes_S \id_S)_!\simeq Rf_! \Psi_{bf}^*,\\
R(f\atimes_S \id_S)_! L\otimes^L M\simeq R (f\atimes_S \id_S)_!(L\otimes^L (f\atimes_S \id_S)^*M)
\end{gather}
for $L\in D(X\atimes_S S,\Lambda)$ and $M\in D(Y\atimes_S S,\Lambda)$.

Next we apply Construction \ref{c.dim}. The functor $R(f\atimes_S \id_S)_!$
has cohomological dimension $\le 2d$, where $d=\max_{y\in Y} \dim(X_y)$, and
admits a right adjoint
\[R(f\atimes_S \id_S)^!\colon D(Y\atimes_S S,\Lambda)\to
D(X\atimes_S S,\Lambda).
\]

For $m\Lambda=0$ with $m$ invertible on $Y$ and $f$ flat with fibers of
dimension $\le d$, we have a trace map \eqref{e.tr}
\[
\tr_{f\atimes_S \id_S}\colon (f\atimes_S \id_S)^*(d)[2d]\to R(f\atimes_S \id_S)^!.
\]
\end{construction}

Consider
\begin{gather}
\label{e.alpha}\alpha'\colon (f\atimes_S \id_S)_!\Psi_{bf}\to \Psi_b f_!,\quad \alpha\colon \Psi_{bf} f^!\simto (f\atimes_S \id_S)^!\Psi_{b},\\
\label{e.beta}\beta'\colon (f_s\atimes_S\id_S)_!i_X^*\simto i_Y^*(f\atimes_S \id_S)_!,\quad \beta\colon i_X^*(f\atimes_S
\id_S)^!\to (f_s\atimes_S \id_S)^!i_Y^*,
\end{gather}
where the two maps in each line are adjoint to each other via Construction
\ref{c.adj}, $\alpha$ is right adjoint to \eqref{e.oshr2} and $\beta'$ is
the inverse of \eqref{e.oshr1}. The map \eqref{e.Psis3} is the composite
\begin{equation}\label{e.Psis3d}
(f_s\atimes_S \id_S)_!i_X^*\Psi_{bf} \xrightarrow[\sim]{\beta' \Psi_{bf}} i_Y^*(f\atimes_S \id_S)_!\Psi_{bf} \xrightarrow{i_Y^*\alpha'} i_Y^*\Psi_{b}f_!,
\end{equation}
and \eqref{e.Psis4} is the composite
\begin{equation}\label{e.Psis4d}
i_X^*\Psi_{bf} f^!\xrightarrow[\sim]{i_X^*\alpha} i_X^*(f\atimes_S \id_S)^!\Psi_{b} \xrightarrow{\beta\Psi_b} (f_s\atimes_S \id_S)^!i_Y^*\Psi_{b}.
\end{equation}
For $f$ proper, $\alpha'$ is the obvious isomorphism.

\begin{lemma}\label{l.trace2}
Assume $m\Lambda=0$ with $m$ invertible on $Y$ and $f$ flat with fibers of
dimension $\le d$. Then the following square commutes
\begin{equation}\label{e.trace}
\xymatrix{(f_s\ttimes_{s_0}\id)^*\Psi_b^s (d)[2d]\ar[d]_{\tr_{f_s\ttimes_{s_0} \id}}\ar[r]^-{\eqref{e.Psis1}} &\Psi_{bf}^s
f^*(d)[2d]\ar[d]^{\tr_f}
\\
(f_s\ttimes_{s_0}\id)^!\Psi_b^s& \Psi_{bf}^s f^!.\ar[l]_{\eqref{e.Psis4}}}
\end{equation}
If, moreover, $f$ is smooth of dimension $d$, then \eqref{e.Psis4} is an
isomorphism.
\end{lemma}

\begin{proof}
The square \eqref{e.trace} decomposes into
\[\xymatrix{(f_s\atimes_S \id)^*i_Y^*\Psi_b(d)[2d]\ar[d]_{\tr_{f_s\atimes_S \id}}\ar@{-}[r]^\sim & i_X^* (f\atimes_S \id)^*\Psi_b(d)[2d] \ar[d]^{\tr_{f\atimes_S \id}}\ar[r]^{\eqref{e.Psis01}}
& i_X^*\Psi_{bf}f^*(d)[2d]\ar[d]^{\tr_f}\\
(f_s\atimes_S \id)^! i_Y^*\Psi_b & i_X^*(f\atimes_S \id)^!\Psi_{b}\ar[l]_{\beta\Psi_b} & i_X^*\Psi_{bf} f^!\ar[l]_{i_X^*\alpha}^\sim.}
\]
The inner cells commute by construction. For $f$ smooth of dimension $d$,
\eqref{e.Psis4} is an isomorphism, because the other three sides of the
square \eqref{e.trace} are isomorphisms by Proposition \ref{l.trace} and
Remark \ref{r.Psis0} (1).
\end{proof}

\begin{remark}
One can show that if $m\Lambda=0$ with $m$ invertible on $Y$, $f$ is smooth
of dimension $d$, $S$ is Noetherian and either $S$ is excellent or $b$ is of
finite type, then $\tr_{f\atimes_S \id_S}\colon (f\atimes_S
\id_S)^*(d)[2d]\to R(f\atimes_S \id_S)^!$ is an isomorphism. In this case,
$\beta$ is an isomorphism.
\end{remark}

\subsection{Base change and sliced $!$-base change}
Theorem \ref{t.Psisc} states that the sliced nearby cycle functor commutes
with duality after pullback by a modification of the base, and remains so
after further pullback. As duality swaps pullback and $!$-pullback, in order
to prove Theorem \ref{t.Psisc} we need to consider the commutation of the
sliced nearby cycle functor with both pullback and $!$-pullback. In this
subsection, we study such commutation.

Consider a Cartesian square
\begin{equation}\label{e.CartT}
\xymatrix{X_T\ar[r]^{g_X}\ar[d]_{f_T} & X\ar[d]^f\\ T\ar[r]^g & S.}
\end{equation}
of schemes. Let $t\to T$ be a geometric point and let $X_t=X\times_S t$.
Consider the diagram of topoi
\begin{equation}\label{e.topoi}
\xymatrix{X_T\ar[r]^{\Psi_{f_T}}\ar[d]_{g_X}\ar@{}[rd]|{\Leftrightarrow} & X_T\atimes_T T\ar[d]^{g_X\atimes_g g}\ar@{}[rd]|{\Leftrightarrow} &X_t\atimes_T T\ar[d]^{\id_{X_t}\atimes_g g}\ar[l]_{i_{T,t}}\\
X\ar[r]^{\Psi_f}& X\atimes_S S& X_t\atimes_S S\ar[l]_{i}.}
\end{equation}
Let $\Lambda$ be a commutative ring. The base change maps
\begin{gather}\label{e.Psibc}
\BC_{f,g}\colon (g_X\atimes_g g)^* \Psi_{f} \to \Psi_{f_T} g_X^*,\\
i^*(g_X\atimes_g g)_*\to (\id_{X_t}\atimes_g g)_*i_{T,t}^*
\end{gather}
induce
\begin{gather}\label{e.Psibcs}
\BC_{f,g}^t\colon (\id\ttimes g_{(t)})^*\Psi_{f}^t \to  \Psi_{f_T}^tg_X^*,\\
\label{e.Psibcs2}\Psi_{f}^t g_{X*} \to (\id\ttimes g_{(t)})_*\Psi_{f_T}^t
\end{gather}
Here $t$ in $\Psi_f^t$ denotes the composition $t\to T\xrightarrow{g}S$.

For $L\in D(X,\Lambda)$, we say that $\Psi_f L$ commutes with base change by
$g$ if $\BC_{f,g}(L)$ is an isomorphism. Note that this holds if and only if
$\BC_{f,g}^t(L)$ is an isomorphism for all geometric points $t$ of $T$. For
$g$ \'etale, $\BC_{f,g}(L)$ is an isomorphism for all $f$ and $L$. For cases
where $\BC_{f,g}(L)$ is an isomorphism for all $g$, see Section \ref{s.4.4}.

\begin{lemma}\label{l.Psibc}
Consider a diagram of schemes with Cartesian squares
\[\xymatrix{X'\ar[r]^{f'}\ar[d]_{r_X} & Y'\ar[r]^{b'}\ar[d]^{r_Y} & S'\ar[d]^{r}\\
X\ar[r]^f & Y\ar[r]^b & S.}
\]
Let $\alpha\colon (r_X\atimes_r r)^* (\id_X\atimes_b b)_*\to
(\id_{X'}\atimes_{b'}b')_* (r_X\atimes_{r_Y} r_Y)^*$ be the base change map.
Then the diagrams of base change maps
\begin{equation}\label{e.bccomp2}
\xymatrix{ (r_X\atimes_r r)^* (\id_X\atimes_b b)_* \Psi_f\ar[r]^{\alpha\Psi_f}\ar@{-}[d]_{\simeq}
& (\id_{X'}\atimes_{b'}b')_* (r_X\atimes_{r_Y} r_Y)^*\Psi_f \ar[r]_{(\id_{X'}\atimes_{b'}b')_*\BC_{f,r_Y}}& (\id_{X'}\atimes_{b'}b')_*\Psi_{f'}r_X^*\ar@{-}[d]^{\simeq}\\
(r_X\atimes_r r)^*\Psi_{bf}\ar[rr]^{\BC_{bf,r}}&& \Psi_{b'f'} r_X^*}
\end{equation}
\begin{equation}\label{e.bccomp3}
\xymatrix{(fr_X\atimes_r r)^* \Psi_b \ar[rr]^-{(f'\atimes_{S'}\id_{S'})^*\BC_{b,r}}\ar[d]_{\eqref{e.Psis01}} &&(f'\atimes_{S'}\id_{S'})^* \Psi_{b'}r_Y^*\ar[d]^{\eqref{e.Psis01}}\\
(r_X\atimes_r r)^*\Psi_{bf}f^*\ar[rr]^{\BC_{bf,r}f^*} &&\Psi_{b'f'}(fr_X)^*}
\end{equation}
commute. For every geometric point $s'\to S'$, the diagram
\begin{equation}\label{e.bccomp1}
\xymatrix{(\id\ttimes r_{(s')})^*\Psi^{s'}_{b}f_* \ar[r]^{\eqref{e.Psis2}}\ar[d]_{\BC^{s'}_{b,r} f_*} & (\id\ttimes r_{(s')})^*(f_{s'}\ttimes\id)_*\Psi^{s'}_{bf}
\ar[r] & (f_{s'}\ttimes\id)_*(\id\ttimes r_{(s')})^*\Psi^{s'}_{bf}\ar[d]_{(f_{s'}\ttimes \id)_*\BC^{s'}_{bf,r}}\\
\Psi^{s'}_{b'} r_X^*f_* \ar[r] & \Psi_{b'}^{s'} f'_* r_Y^* \ar[r]^{\eqref{e.Psis2}} & (f_{s'}\ttimes \id)_* \Psi_{b'f'}^{s'} r_Y^* }
\end{equation}
commutes.  Moreover,
\begin{enumerate}
\item Assume $b$ finite. Then the map $\alpha$ is an isomorphism. In
    particular, $\Psi_{bf}L$ commutes with base change by $r$ if and only
    if $\Psi_f L$ commutes with base change by $r_Y$.
\item Assume $f$ and $f'$ are locally acyclic locally of finite type and
    $\Lambda$ of torsion. Then the vertical arrows of \eqref{e.bccomp3}
    are isomorphisms. In particular, if $\Psi_b M$ commutes with base
    change by $r$, then $\Psi_{bf}(f^*M)$ commutes with base change by
    $r$.
\item Assume either $f$ integral or $\Lambda$ of torsion and $f$ proper.
    Then the horizontal arrows of \eqref{e.bccomp1} are isomorphisms. In
    particular, if $\Psi_{bf}L$ commutes with base change by $r$, then
    $\Psi_{b}(f_*L)$ commutes with base change by $r$.
\item Assume $f$ finite. Then $\Psi_{bf}L$ commutes with base change by
    $r$ if and only if $\Psi_{b}(f_*L)$ commutes with base change by $r$.
\end{enumerate}
\end{lemma}

\begin{proof}
The commutativity of the squares follows from the definition. Then (2)
follows from Remark \ref{r.Psis0} (1). Moreover, (3) follows from Remark
\ref{r.Psis0} (2), integral base change, and proper base change, and (4)
follows from (3) and the fact that $(f_{s'}\ttimes\id)_*$ is conservative.
For (1), note that by \cite[Proposition 1.13]{IZ}, for $L\in D^+(X\atimes_Y
Y,\Lambda)$ and every geometric point $x'\to X'$, the restriction of $\alpha
(L)$ to $x'\atimes_{S'} {S'}$ is the base change map
$r_{(s')}^*b_{(y)*}(L|_{x\atimes_Y Y})\to b'_{(y')*}(r_{Y})_{(y')}^*
(L|_{x\atimes_Y Y})$ associated to the commutative square of strict
localizations
\[\xymatrix{Y'_{(y')}\ar[r]^{b'_{(y')}}\ar[d]_{(r_{Y})_{(y')}} &
S'_{(s')}\ar[d]^{r_{(s')}}\\
Y_{(y)}\ar[r]^{b_{(y)}} & S_{(s)},}
\]
where the geometric points are images of $x'$. For $b$ finite, the square is
Cartesian and the functors are exact, so that $\alpha$ is an isomorphism for
$L\in D(X\atimes_Y Y,\Lambda)$. Moreover, in this case, the functors
$b'_{(y')*}$ are conservative, so that $(\id_{X'}\atimes_{b'}b')_*$ is
conservative.
\end{proof}

We will construct an analogue of $\BC_{f,g}^t$ for $!$-pullback. We start by
the finite case.

\begin{lemma}\label{l.Psi!}
Assume $g$ finite. Let $s\to S$ be a geometric point. Then \eqref{e.Psibcs2}
induces an isomorphism
\begin{equation}\label{e.Psi*}
\Psi_{f}^{s}g_{X*}\simto\prod_t (\id\ttimes
g_{(t)})_*\Psi_{f_T}^{t},
\end{equation}
where $t$ runs through points of $T\times_S s$. Moreover, for each $t$, the
map
\begin{equation}\label{e.Psi!}
\Psi_{f_T}^{t}g_X^!\to (\id\ttimes g_{(t)})^!\Psi_{f}^{s}
\end{equation}
adjoint via Construction \ref{c.adj} to the map $(\id\ttimes
g_{(t)})_*\Psi_{f_T}^{t}\to \Psi_{f}^{s}g_{X*}$ induced by the inverse of
\eqref{e.Psi*} is an isomorphism on $D^+$ if $g$ is of finite presentation.
\end{lemma}

\begin{proof}
By \cite[Lemme 9.1]{Org}, the morphisms $g_X\atimes_g \id_T\colon
X_T\atimes_T T\to X \atimes_S T$ and $\coprod_t X_t\atimes_T T\to
X_s\atimes_S T$ are equivalences. Thus \eqref{e.topoi} is deduced from the
diagram
\[
\xymatrix{X_T\ar[r]^{\Psi_{f_T}}\ar[d]_{g_X}\ar@{}[rd]|{\Leftrightarrow} & X \atimes_S T\ar[d]^{\id_X\atimes_S g}\ar@{}[rd]|{\Leftrightarrow} &X_s\atimes_S T\ar[d]^{\id_{X_s}\atimes_S g}\ar[l]_{i_{T}}\\
X\ar[r]^{\Psi_f}& X\atimes_S S& X_s\atimes_S S\ar[l]_{i}.}
\]
The map \eqref{e.Psi*} is the composite
\[i^*\Psi_f g_{X*}\xrightarrow[\sim]{i^*\alpha} i^*(\id_X\atimes_S g)_* \Psi_{f_T}\xrightarrow[\sim]{\beta \Psi_{f_T}} (\id_{X_s}\atimes_S g)_*i_T^*\Psi_{f_T},\]
where $\alpha\colon \Psi_f g_{X*}\simto (\id_X\atimes_S g)_* \Psi_{f_T}$,
and $\beta\colon i^*(\id_X\atimes_S g)_*\to (\id_{X_s}\atimes_S g)_*i_T^*$
is the base change map. By Proposition \ref{p.obc} and Corollary
\ref{c.obc}, the functors in $\beta$ are exact and $\beta$ is an
isomorphism.

Note that $S_{(s)}\times_S T$ is the disjoint union of $T_{(t)}$, and for
each $t$ and each geometric point $x\to X_s$, the square
\[\xymatrix{(X_T)_{(x,t)}\ar[r]^{(g_X)_{(x,t)}}\ar[d]_{(f_T)_{(x,t)}} & X_{(x)}\ar[d]^{f_{(x)}}\\
T_{(t)}\ar[r]^{g_{(t)}} & S_{(s)}}
\]
is Cartesian. By Remark \ref{r.local}, on $D^+$, the restriction of
\eqref{e.Psi*} to $x\ttimes S_{(s)}$ can be identified with $\prod_t$ of
$f_{(x)*}(g_X)_{(x,t)*}\simeq g_{(t)*}(f_T)_{(x,t)*}$. Moreover, for $g$ of
finite presentation, on $D^+$, the restriction of \eqref{e.Psi!} to
$x\ttimes S_{(s)}$ can be identified with the isomorphism
$(f_T)_{(x,t)*}(g_X)_{(x,t)}^!\simto g_{(t)}^!f_{(x)*}$ of Lemma
\ref{l.localshrbc}, adjoint to finite base change.
\end{proof}

\begin{remark}\label{r.bc!}
Assume $g$ finite. By Lemma \ref{l.adjtriv}, $(\id_X\atimes_S g)_*$ admits a
right adjoint $R(\id_X\atimes_S g)^!$. The map \eqref{e.Psi!} is the
restriction of the composite
\[i_T^*\Psi_{f_T} g_X^!\xrightarrow[\sim]{i_T^*\alpha'} i_T^*(\id_X\atimes_S g)^! \Psi_f \xrightarrow{\beta'\Psi_f}(\id_{X_s}\atimes_S g)^!i^*\Psi_f,\]
where $\alpha'\colon \Psi_{f_T} g_X^!\simto (\id_X\atimes_S g)^! \Psi_f$ and
$\beta'\colon i_T^*(\id_X\atimes_S g)^!\to (\id_{X_s}\atimes_S g)^!i^*$ are
adjoint to $\alpha^{-1}$ and $\beta^{-1}$, respectively, via Construction
\ref{c.adj}. Note that $\alpha'$ is right adjoint to the base change map
$\Psi_f^*(\id_X\atimes_S g)_*\to g_{X*}\Psi_{f_T}^*$, which is an
isomorphism by Proposition \ref{p.obc}. If $g$ is a closed immersion of
finite presentation, then $\beta'$ is an isomorphism on $D^+$ by Corollary
\ref{c.obc} applied to the complement of $g$.
\end{remark}

Next we construct $(\id\times g_{(t)})^!$ in the case where $g$ is not
necessarily finite. Assume $m\Lambda=0$ for $m$ invertible on $S$.

\begin{construction}\label{c.localized}
Let $X$ be a coherent topos. Consider the following subcategory $\cC$ of the
category of schemes. The objects are strictly local schemes (namely, spectra
of strictly Henselian rings). A morphism $f\colon T\to S$ belongs to $\cC$
if it is the strict localization of a morphism of schemes $f_0\colon T_0\to
S_0$ of finite presentation at a geometric point of $T_0$. We define for
such $f$ a functor
\[(\id\ttimes f)^!\colon D^+(X\ttimes S,\Lambda)\to D^+(X\ttimes T,\Lambda),\]
isomorphic to $(\id\ttimes f)^*(d)[2d]$ for $f$ a strict localization of a
smooth morphism of dimension $d$, and right adjoint to $(\id\times f)_*$ for
$f$ finite, and compatible with composition. Note that each $f$ in $\cC$
admits a decomposition $f=f_2 f_1$ in $\cC$, with $f_1$ a closed immersion
and $f_2$ a strict localization of a smooth morphism.

We apply Ayoub's gluing formalism \cite[Th\'eor\`eme 1.3.1]{Ayoub} (see
\cite[Theorem 1.1]{glue} for a common generalization of this and Deligne's
gluing formalism used earlier) to the category $\cC_{F/}$ of objects of
$\cC$ under a fixed object $F$. This category has the advantage of admitting
fiber products, the fiber product of $T$ and $S'$ over $S$ being the strict
localization of $T\times_S S'$ at the image of the closed point of $F$. The
construction of $(\id\ttimes f)^!$ does not depend on the choice of $F\to
T$. Alternatively we may apply \cite[Theorem 5.9]{glue} directly to $\cC$.

Given a commutative square $hg=fi$ in $\cC$ with $h$, $i$ finite, $f$
(resp.\ $g$) a strict localization of a smooth morphism of dimension $d$
(resp.\ $e$), we need to construct an isomorphism
\begin{equation}\label{e.loc!1}
(\id\ttimes g)^!(\id\ttimes h)^!\simeq (\id\ttimes i)^!(\id\ttimes f)^!
\end{equation}
We decompose the square into a commutative diagram in $\cC$ with Cartesian
inner square
\[\xymatrix{R\ar[dr]_{g}\ar@/^1pc/[rr]^{i}\ar@{-->}[r]_{k}
&T'\ar@{-->}[r]_{h'}\ar@{-->}[d]^{f'} & T\ar[d]^f\\
&S'\ar[r]^h & S.}
\]
For the inner square, the map
\begin{equation}\label{e.loc!2}
(\id\ttimes f')^!(\id\ttimes h)^!\to (\id\ttimes h')^!(\id\ttimes f)^!
\end{equation}
adjoint to the inverse of the base change isomorphism $(\id\ttimes
f)^*(\id\ttimes h)_*\simeq (\id\ttimes h')_*(\id\ttimes f')^*$ via
Construction \ref{c.adj} is an isomorphism. Indeed, if $f=pf_0$ with $f_0$
smooth and $p$ a strict localization, the analogous assertion for $f_0$ is
clear and the one for $p$ follows from Proposition \ref{p.bcupper!} and
Lemma \ref{l.localshrbc}. Note that $k$ is a local complete intersection of
virtual relative dimension $e-d$. The cycle class map $\Lambda\to
k^!\Lambda(d-e)[2(d-e)]$ of $k$ \cite[XVI D\'efinition 2.5.11]{ILO} induces
\begin{equation}\label{e.loc!3}
(\id\ttimes g)^!\to (\id\ttimes k)^!(\id\ttimes f')^!,
\end{equation}
which is an isomorphism. Indeed, $k$ is the strict localization of a finite
morphism $R_0\to T'_0$, with $R_0$ and $T'_0$ affine and smooth over $S'$ of
relative dimension $e$ and $d$, respectively. Combining \eqref{e.loc!2} and
\eqref{e.loc!3}, we get \eqref{e.loc!1}.
\end{construction}

\begin{remark}\label{r.localtbc}
For $f\colon T\to S$ as above and $g\colon X\to Y$ a morphism of coherent
topos, we have an isomorphism
\[(g\ttimes \id)^* (\id\ttimes f)^!\simeq (\id\ttimes f)^! (g\ttimes \id)^*\]
by Proposition \ref{p.bcupper!}.
\end{remark}

\begin{remark}\label{r.localdesc}
Let $g\colon T\to S$ be a separated morphism of finite presentation of
coherent schemes. Let $t\to T$ be a geometric point. Consider the
commutative square
\[\xymatrix{T_{(t)}\ar[r]^{g_{(t)}}\ar[d]_{j} & S_{(t)}\ar[d]^{k}\\
T\ar[r]^g&S.}
\]
Note that $g_{(t)}$ is a morphism of $\cC$. We have an isomorphism
\[(\id\ttimes
j)^*(\id\ttimes g)^! \simto (\id\ttimes g_{(t)})^!(\id \ttimes k)^*.
\]
Indeed, we may assume that $g$ admits a factorization
$T\xrightarrow{i}R\xrightarrow{p} S$ with $p$ smooth and $i$ a closed
immersion. The assertion for $p$ is trivial and the assertion for $i$
follows from Proposition \ref{p.bcupper!} and Lemma \ref{l.localshrbc} (or
from limit arguments applied to the complement of $i$).
\end{remark}

\begin{construction}
We construct, for each morphism $f\colon T\to S$ of $\cC$ and $L\in
D^-(X\ttimes S,\Lambda)$, $M\in D^+(X\ttimes S,\Lambda)$, a map
\begin{equation}\label{e.localshr}
(\id\ttimes f)^!R\cHom(L,M)\to R\cHom((\id\ttimes f)^*L,(\id\ttimes f)^!M).
\end{equation}
For $f$ a strict localization of a smooth morphism, we take the restriction
map. For $f$ finite, we take \eqref{e.1.13.1}. In general we take a
decomposition. One checks that the resulting map does not depend on choices
and is compatible with composition.
\end{construction}

\begin{lemma}\label{l.localshr}
Let $X$ be a coherent scheme and let $f\colon T\to S$ be a morphism of
$\cC$. For $L\in D^-_c(X\ttimes S,\Lambda)$, $M\in D^+(X\ttimes S,\Lambda)$,
\eqref{e.localshr} is an isomorphism.
\end{lemma}

\begin{proof}
We decompose $f$ as $gj$, with $g$ separated of finite presentation and $j$
a strict localization. By Remark \ref{r.localdesc}, we have $(\id\ttimes
f)^!\simeq (\id\ttimes j)^*(\id\ttimes g)^!$. The lemma follows from Lemma
\ref{l.local} applied to $j$ and \eqref{e.1.13.1} applied to $g$.
\end{proof}

\begin{construction}\label{c.bc!}
Consider a Cartesian square \eqref{e.CartT} of coherent schemes with $g$
separated of finite presentation. Let $t$ be a geometric point of $T$. We
construct the sliced $!$-base change map
\begin{equation}\label{e.Psibc!}
\BC^{!t}_{f,g}\colon (\id\ttimes g_{(t)})^!\Psi_{f}^t \to \Psi_{f_T}^t g_X^!
\end{equation}
as follows. For $g$ smooth, we take the map induced by $\BC^{t}_{f,g}$
\eqref{e.Psibcs}. For $g$ a finite morphism, we take the inverse of
\eqref{e.Psi!}, which is an isomorphism. In general, the restriction of $g$
to an open subscheme $U\subseteq T$ containing $t$ can be decomposed into
$U\xrightarrow{g'}S'\xrightarrow{g''}S$ with $g'$ finite (one may even take
$g'$ to be a closed immersion) and $g''$ smooth. We take the composition
\[(\id\ttimes g'_{(t)})^!(\id\ttimes g''_{(t)})^! \Psi_f^t \xrightarrow{\BC^{!t}_{f,g''}} (\id\ttimes f'_{(t)})^! \Psi_{f_S}^tg''^!_X
\xrightarrow{\BC^{!t}_{f_{S'},g}} \Psi_{f_U}^tg'^!_{X}g''^!_X.
\]
One checks by restricting to local sections that the resulting base change
map does not depend on choices and is compatible with composition.
\eqref{e.Psibc!} is an isomorphism for $g$ quasi-finite by Zariski's main
theorem.
\end{construction}

\begin{remark}\label{r.ggg}
Assume that $g$ is smooth or a closed immersion. With the notation of
\eqref{e.topoi}, $\BC^{!t}_{f,g}$ is the composite
\[(\id_{X_t} \ttimes g_{(t)})^! i^*\Psi_f \xrightarrow[\sim]{\gamma \Psi_f} i_{T,t}^*(g_X\atimes_g
g)^!\Psi_f\xrightarrow{i_{T,t}^*\BC^!_{f,g}} i_{T,t}^* \Psi_{f_T} g_X^!,
\]
where $\gamma\colon (\id_{X_t} \ttimes g_{(t)})^! i^*\simto
i_{T,t}^*(g_X\atimes_g g)^!$ and $\BC^!_{f,g}\colon (g_X\atimes_g
g)^!\Psi_f\to \Psi_{f_T} g_X^!$. For $g$ a closed immersion, these maps are
given by $\beta'^{-1}$ and $\alpha'^{-1}$, where $\alpha'$ and $\beta'$ are
as in Remark \ref{r.bc!}. For $g$ smooth of dimension~$d$, we define
$(g_X\atimes_g g)^!\colonequals (g_X\atimes_g g)^*(d)[2d]$, $\gamma$ is the
obvious isomorphism, and $\BC^!_{f,g}$ is induced by $\BC_{f,g}$.

It would be nice to extend this interpretation of $\BC^{!t}_{f,g}$ to more
general $g$.
\end{remark}

In the setting of Construction \ref{c.bc!}, for $L\in D^+(X,\Lambda)$, we
say that $\Psi_f L$ commutes with sliced $!$-base change by $g$ if
$\BC_{f,g}^{!t}$ is an isomorphism for all geometric points $t$ of $T$. We
say that $\Psi_f L$ commutes with sliced $!$-base change if it commutes with
sliced $!$-base change by $g$ for all $g$ separated and of finite
presentation. The following lemma is immediate from the construction of
$\BC^{!t}_{f,g}$.

\begin{lemma}\label{l.Psibc!}
$\Psi_f L$ commutes with smooth base change if and only if $\Psi_f
    L$ commutes with sliced $!$-base change.
\end{lemma}

\begin{lemma}\label{l.Psibc!2}\leavevmode
\begin{enumerate}
\item If $\Psi_f L$ commutes with smooth base change, then for every
    $g\colon T\to S$ separated of finite presentation, $\Psi_{f_T} g_X^!L$
    commutes with smooth base change.
\item If for some $g$ finite surjective of finite presentation,
    $\Psi_{f_T} g_X^!L$ commutes with smooth base change, then $\Psi_f L$
    commutes with smooth base change.
\end{enumerate}
\end{lemma}

Part (2) is a dual version of \cite[Lemme 3.3]{Org}.

\begin{proof}
By Lemma \ref{l.Psibc!}, we may replace smooth base change by sliced
$!$-base change. Then (1) follows from the compatibility of the sliced
$!$-base change map with composition. For (2), let $h\colon S'\to S$ be a
separated morphism of finite presentation and form the Cartesian square
\[\xymatrix{T'\ar[r]^{g'}\ar[d]_{h_T} & S'\ar[d]^h\\
T\ar[r]^g & S.}
\]
By the compatibility of the sliced $!$-base change map with composition, for
every geometric point $t'\to T'$, we have
\[(\BC^{!t'}_{f_{S'},g'}h_X^!)((\id\ttimes g'_{(t')})^!\BC^{!s'}_{f,h})=(\BC^{!t'}_{f_T,h_T}g_X^!) ((\id\ttimes (h_T)_{(t')})^!\BC^{!t}_{f,g}),\]
where $s'\to S'$ and $t\to T$ denote the images of $t'\to T'$. Since
$\BC^{!t'}_{f_{S'},g'}$, $\BC^{!t}_{f,g}$, $\BC^{!t'}_{f_T,h_T}(g_X^!L)$ are
isomorphisms, and the family of functors $(\id\ttimes g'_{(t')})^!$, $t'$
above a fixed $s'$, is conservative by Lemma \ref{l.cd} (1) below,
$\BC^{!s'}_{f,h}L$ is an isomorphism.
\end{proof}

\begin{lemma}\label{l.cd}
Let $X$ be a locally coherent topos and $g\colon T\to S$ a separated
surjective morphism of finite presentation of schemes. Let $\Lambda$ be a
commutative ring. Assume either $g$ quasi-finite or $\Lambda$ of torsion.
\begin{enumerate}
\item For any $L\in D(X\ttimes S,\Lambda)$, $(\id_X\ttimes g)^!L\in D^{\ge
    a}$ implies $L\in D^{\ge a}$. In particular, $(\id_X\ttimes g)^!\colon
    D(X\ttimes S,\Lambda)\to D(X\ttimes T,\Lambda)$ is conservative.
\item Assume moreover that $g$ is proper and $m\Lambda=0$ for $m$
    invertible on $S$. Then for every geometric point $s\to S$, the family
    of functors $(\id_X\ttimes g_{(t)})^!\colon D(X\ttimes
    S_{(s)},\Lambda)\to D(X\ttimes T_{(t)},\Lambda)$ is conservative. Here
    $t$ runs through geometric points of $T$ above $s\to S$.
\end{enumerate}
\end{lemma}

\begin{proof}
For (1), let $S'\subseteq S$ be the smallest closed subset such that
$\tau^{\le a-1}L$ is supported on $X\ttimes S'$. Assume $S'$ nonempty. By
\cite[Proposition 17.16.4]{EGAIV}, the base change of $g$ to $S'$ admits
generically a quasi-section: there exist a nonempty open subscheme
$U\subseteq S'$ and a commutative diagram
\begin{equation}\label{e.cddiag}
\xymatrix{T'\ar[r]^{h}\ar[d]_{g'} &T\times_S S'\ar[r]^{i_T}\ar[d]& T\ar[d]^g\\
U\ar@{^{(}->}[r]^j & S'\ar@{^{(}->}[r]^i & S,}
\end{equation}
where $g'$ is finite surjective of finite presentation. Let $L'=L|_{X\ttimes
U}$. Since $i_T h$ is quasi-finite, we have $(\id_X\ttimes g')^! L'\simeq
(\id_X\ttimes i_T h)^!(\id_X\ttimes g)^!L\in D^{\ge a}$, so that
\[R\cHom((\id_X\ttimes g')_*\Lambda,L')\simeq (\id_X\ttimes g')_*(\id_X\ttimes g')^! L'\in D^{\ge
a}.
\]
Moreover, $(\id_X\ttimes g')_*\Lambda\simeq p_2^*g'_*\Lambda$, where
$p_2\colon X\ttimes S'\to S'$ is the projection. The stalks of $g'_*\Lambda$
are nonzero finite free $\Lambda$-modules. Up to shrinking $U$, we may
assume that $g'_*\Lambda$ is locally constant. It follows that $L|_{X\ttimes
U}\in D^{\ge a}$, contradicting the minimality of $S'$. Therefore, $S'$ is
empty and $L\in D^{\ge a}$.

For (2), we may assume that $S$ is strictly local of closed point $s$. Then,
by Remark \ref{r.localdesc}, $(\id\ttimes g_{(t)})^!=(\id\ttimes
j_t)^*(\id\ttimes g)^!$, where $j_t\colon T_{(t)}\to T$ is the strict
localization. It suffices to show that the family $(j_t)$ is surjective. Let
$y\in T$ be a point. Since $g$ is a closed map, $y$ specializes to a point
$z\in g^{-1}(s)$. Let $t\to T$ be a geometric point above $z$. Since $j_t$
is flat, hence generizing by \cite[Proposition 2.3.4]{EGAIV}, $y$ belongs to
the image of $j_t$.
\end{proof}

\begin{remark}
If $S$ is Noetherian finite-dimensional, Constructions \ref{c.localized} and
\ref{c.bc!} extend to the unbounded derived categories, as the functors have
finite cohomological dimensions \cite[XVIII$_{\text{A}}$ Corollary
1.4]{ILO}.
\end{remark}

\subsection{$\Psi$-goodness and weak $\Psi$-goodness}\label{s.4.4}
Let $f\colon X\to S$ be a morphism of schemes and let $L\in D(X,\Lambda)$.
Following \cite[Appendix]{IZ} we say that $(f,L)$ is \emph{$\Psi$-good} if
$\Psi_f L$ commutes with base change. Examples (3) through (5) below are not
needed in this paper.

\begin{example}\label{e.Psigood}\leavevmode
\begin{enumerate}
\item For $L\in D^+(X,\Lambda)$, $(f,L)$ is universally locally acyclic if
    and only if $(f,L)$ is $\Psi$-good and $R\Phi_f L=0$ (Remark
    \ref{r.lc}).
\item Assume $m\Lambda=0$ for some $m$ invertible on $S$. Assume that $S$
    has only finitely many irreducible components and for every
    modification $r\colon S'\to S$, there exists a finite surjective
    morphism $T\to S$ that factors through $r$.  Then all pairs $(f,L)$
    with $L\in D^+(X,\Lambda)$ are $\Psi$-good. In particular, if $(f,L)$
    is locally acyclic, then $(f,L)$ is universally locally acyclic
    (compare with Corollary \ref{c.la}).

    The condition on $S$ holds if (a) $S$ is the spectrum of a valuation
    ring, or (b) $S$ is Noetherian of dimension $\le 1$. In case (a), $r$
    admits a section by valuative criterion. In case (b), $r$ is
    quasi-finite by dimension formula \cite[(5.6.5.1)]{EGAIV}, hence
    finite.

    To show the $\Psi$-goodness, we reduce by standard limit arguments to
    the case $f$ of finite presentation, $\Lambda=\Z/m\Z$, and $L\in
    \Shv_c$. We conclude by Orgogozo's theorem (Theorem \ref{t.Org} (1)
    below), the assumption on $S$, and \cite[Lemme 3.3]{Org}.

\item Assume $f$ of finite type and $\Lambda$ of torsion. Assume that
    there exists an open immersion $j\colon U\to X$ with complement
    $Y=X-U$ quasi-finite over $S$ such that $(fj,L|_U)$ is $\Psi$-good.
    Then $(f,L)$ is $\Psi$-good. This extends \cite[Proposition 6.1]{Org}
    and the proof is similar. Let us show more generally that
    $\BC^t_{f,g}(L)|_{y\ttimes T_{(t)}}$ is an isomorphism for every
    isolated point $y$ of $Y_t$ (without assuming $Y$ quasi-finite over
    $S$). For this, we may assume $f$ proper. By Lemma \ref{l.Psibc} (3),
    $R(f_t\ttimes \id)_*\BC^t_{f,g}(L)$ is an isomorphism. It then
    suffices to note that the cone of $\BC^t_{f,g}(L)$ is supported on
    $Y_{t}\ttimes T_{(t)}$ by assumption.
\item Assume $\Lambda$ of torsion. Let $f_i\colon X_i\to S$, $i=1,2$ be
    morphisms locally of finite type. Then for $(f_i,L_i)$ $\Psi$-good
    with $L_i\in D^-(X_i,\Lambda)$, $(f_1\times_S f_2,L_1\boxtimes^L L_2)$
    is $\Psi$-good by the K\"unneth formula for $R\Psi$ \cite[Theorem
    A.3]{IZ}.
\item Assume $\Lambda$ of torsion and $f$ locally of finite type. Then for
    $(f,L)$ $\Psi$-good with $L\in D^-(X,\Lambda)$, $(f,L\otimes^L f^*M)$
    is $\Psi$-good for all $M\in D(S,\Lambda)$. This follows from the
    projection formula for $R\Psi$ \cite[Proposition A.6]{IZ}.
\end{enumerate}
\end{example}

Let us recall the following form of Orgogozo's theorem \cite{Org}.

\begin{theorem}[Orgogozo]\label{t.Org} Assume $S$ has only finitely many irreducible components.
Assume $f$ of finite presentation, and $\Lambda$ Noetherian such that
$m\Lambda=0$ for some $m$ invertible on $S$.
\begin{enumerate}
\item For $L\in D^b_c(X,\Lambda)$, there exists a modification $g\colon
    T\to S$ such that $(f_{T},g^*_X L)$ is $\Psi$-good.
\item For $L\in D_c(X,\Lambda)$, if $(f,L)$ is $\Psi$-good, then $R\Psi_f
    L\in D_c(X\atimes_S S,\Lambda)$.
\end{enumerate}
\end{theorem}

A sheaf of $\Lambda$-modules $\cF$ on $X\atimes_S S$ is said to be
\emph{constructible} if every stalk is finitely generated and there exist
finite partitions $X=\bigcup_i X_i$, $S=\bigcup_j S_j$ into disjoint
constructible locally closed subsets such that the restriction of $\cF$ to
each $X_i \atimes_S S_j$ is locally constant.

\begin{proof}
The results of \cite{Org} are stated only for the case $\Lambda=\Z/m\Z$, but
as remarked in \cite[Theorem 1.6.1]{IZ}, the proofs can be easily adapted to
the case of general $\Lambda$ as above. Indeed, in the last paragraph of
\cite[Section 4.4]{Org}, we reduce to the case where $\cF$ is constant
instead of $\cF=\Lambda$, and at the end of the proofs \cite[Sections 5.1,
10.2, 10.3]{Org} we work with constant constructible sheaves instead of
$\Lambda$.

(1) is \cite[Th\'eor\`eme 2.1]{Org} (for general $\Lambda$). For (2), by
\cite[Lemme 10.5]{Org}, it suffices to show that for each $i$, there exists
a proper surjective morphism $T\to S$ such that $(R^i\Psi_f
L)|_{X_T\atimes_T T}$ is constructible. For $(f,L)$ $\Psi$-good, we have
$(R^i\Psi_f L)|_{X_T\atimes_T T}\simeq R^i\Psi_{f_T} (L|_{X_T})$. By
\cite[Proposition 3.1]{Org}, there exists a truncation $L'\in
D^b_c(X,\Lambda)$ of $L$ such that $R^i\Psi_{f_T}(L'|_{X_T})\simeq
R^i\Psi_{f_T}(L|_{X_T})$ for all $T\to S$. Then it suffices to apply
\cite[Th\'eor\`eme 8.1]{Org} to $L'$.
\end{proof}

\begin{remark}\label{r.Org}
It follows from Theorem \ref{t.Org} (1) that for $L\in D^b_c(X,\Lambda)$, if
$R\Psi_f L$ commutes with base change by modifications, then it commutes
with any base change.
\end{remark}

Assume $S$ Noetherian, $f$ of finite type, and $\Lambda$ Noetherian such
that $m\Lambda=0$ for some $m$ invertible on $S$.

\begin{cor}\label{c.Orgco}
For $L\in D^b_c(X,\Lambda)$, there exists an open subscheme $U\subseteq S$
of complement of codimension $\ge 2$ such that $(f_U,L|_{X_U})$ is
$\Psi$-good.
\end{cor}

\begin{proof}
There exists an open subscheme $U\subseteq S$ of complement of codimension
$\ge 2$ such that the restriction of the modification $g$ in Theorem
\ref{t.Org} (1) to $U$ is quasi-finite, by Chevalley's semicontinuity
theorem and dimension formula \cite[(5.6.5.1)]{EGAIV}, hence finite. We then
conclude by \cite[Lemme 3.3]{Org}.
\end{proof}

\begin{definition}
For $L\in D_c(X,\Lambda)$, we say that $(f,L)$ is \emph{weakly $\Psi$-good}
if $\Psi_{f}^{s}(L)$ belongs to $D_c(X_s\ttimes S_{(s)},\Lambda)$ for every
geometric point $s\to S$ and $\Psi_f (L)$ commutes with smooth base change.
\end{definition}

By Theorem \ref{t.Org} (2), for $L\in D_c$, if $(f,L)$ is $\Psi$-good, then
$(f,L)$ is weakly $\Psi$-good. We do not know whether the converse holds.

\begin{remark}\leavevmode\label{r.weakgood}
Consider a Cartesian square \eqref{e.CartT} with $g$ finite. Let $M\in
    D_c(X_T,\Lambda)$. If $(f_T,M)$ is weakly $\Psi$-good, then $(f,g_{X*}M)$
    is weakly $\Psi$-good. Indeed, $\Psi_{gf_T}(M)$ commutes with smooth
    base change by Lemma \ref{l.Psibc} (1) and $\Psi_{f}(g_{X*}M)$ commutes
    with smooth base change by Lemma \ref{l.Psibc} (4), and
    $\Psi_f^s(g_{X*}M)$ is in $D_c$ by Lemma \ref{l.Psi!}.
\end{remark}

The following is an analogue of Orgogozo's theorem for $!$-pullback. The
quasi-excellence of $S$ (not needed in Orgogozo's theorem) ensures that
$p^!_X L\in D^b_c$.

\begin{theorem}\label{t.good}
Assume $S$ quasi-excellent. Let $f\colon X\to S$ be a morphism of finite
type and let $L\in D^b_c(X,\Lambda)$. There exists a modification $p\colon
S'\to S$ such that $(f_{S'},p^!_XL)$ is weakly $\Psi$-good.
\end{theorem}

Our proof of the theorem relies on the following preliminary results on
$!$-pullback.

\begin{lemma}\label{l.wg2}
Assume $S$ quasi-excellent. Let $f\colon X\to S$ be a morphism of finite
type and let $L\in D^+_c(X,\Lambda)$.
\begin{enumerate}
\item If $(f,L)$ is weakly $\Psi$-good, then for every $g\colon T\to S$
    separated of finite type, $(f_T,g_X^! L)$ is weakly $\Psi$-good.
\item If for some $g\colon T\to S$ finite surjective, $(f_T,g_X^! L)$ is
    weakly $\Psi$-good, then $(f,L)$ is weakly $\Psi$-good.
\end{enumerate}
\end{lemma}

\begin{proof}
The commutation with smooth base change follows from Lemma \ref{l.Psibc!2}.
In (1) and (2), for every geometric point $t\to T$ above $s\to S$, we have
$\Psi^t_{f_T}(g_X^! L)\simeq (\id\ttimes g_{(t)})^!\Psi_f^s L$. Then (1)
follows from the fact that $(\id\ttimes g_{(t)})^!$ preserves $D^+_c$
(Proposition \ref{p.prodcons}), and (2) follows from the following lemma.
\end{proof}

\begin{lemma}\label{l.pullback}
Let $X$ be a quasi-excellent scheme. Let $g\colon T\to S$ be a separated
surjective morphism of finite type between quasi-excellent schemes. Let
$\Lambda$ be a Noetherian commutative ring such that $m\Lambda=0$ for some
$m$ invertible on $X$ and on $S$. Let $L\in D^+(X\ttimes S,\Lambda)$ such
that $R(\id_X\ttimes g)^! L\in D^+_c(X\ttimes T,\Lambda)$. Then $L\in
D^+_c$.
\end{lemma}

\begin{proof}
The proof is similar to that of Lemma \ref{l.cd} (1). Let $S'\subseteq S$ be
the smallest closed subset such that $L|_{X\ttimes (S-S')}$ belongs to
$D^+_c$. Assume $S'$ nonempty. Consider the diagram \eqref{e.cddiag} and let
$L'=R(\id_X\ttimes ij)^!L$. We have $R\cHom((\id_X\ttimes
g')_*\Lambda,L')\simeq (\id_X\ttimes g')_*R(\id_X\ttimes g')^! L'\simeq
(\id_X\ttimes g')_*R(\id_X\ttimes i_T h)^!R(\id_X\ttimes g)^!L$, which
belongs to $D^+_c$ by the assumptions. Recall that $g'$ is finite and
surjective. The stalks of $g'_*\Lambda$ are nonzero finite free
$\Lambda$-modules. Up to shrinking $U$, we may assume $g'_*\Lambda$ locally
constant. Then $L'\in D^+_c$. Let $S''=S'-U$ and let $u\colon U\to S-S''$
and $v\colon S-S'\to S-S''$ be the inclusions. The distinguished triangle
\[(\id_X\ttimes u)_* L'\to L|_{X\ttimes (S-S'')}\to R(\id_X\ttimes
v)_* L|_{X\ttimes (S-S')}\to (\id_X\ttimes u)_* L'[1]
\]
implies $L|_{X\ttimes (S-S'')}\in D^+_c$, contradicting the minimality of
$S'$. Thus $S'$ is empty and $L\in D^+_c$.
\end{proof}

\begin{proof}[Proof of Theorem \ref{t.good}]
We prove first the case $L=i_{X*}M$, where $i\colon F\subseteq S$ is a
closed immersion with $F$ reduced and there exists a modification $m\colon
F'\to F$ such that $(f_{F'},m^!_X M)$ is weakly $\Psi$-good. The proof is
similar to \cite[Section 4.2]{Org}. We construct a commutative diagram of
schemes
\[\xymatrix{G\ar[r]^q\ar[d] & F_{T}\ar[r]^{i_{T}}\ar[d]^{\pi_F} & T\ar[d]^\pi & T'\ar[l]\ar[d]^r\\
F'\ar[r]^m & F\ar[r]^i & S & S'\ar[l]_p,}
\]
where $p$ and $\pi_\red\colon T\to S_\red$ are modifications, $q$ and $r$
are finite surjective, and the square in the middle is Cartesian, as
follows. Applying \cite[Lemme 4.3]{Org} to $i_\red m$, we get the left and
middle squares. Applying \cite[Lemme 3.2]{Org} to $\pi$, we get the right
square. By Lemma \ref{l.wg2} (1), $(f_{G},q^!_X \pi_{X_F}^!M)$ is weakly
$\Psi$-good. By Lemma \ref{l.wg2} (2), $(f_{F_{T}},\pi_{X_F}^!M)$ is weakly
$\Psi$-good. By Remark \ref{r.weakgood}, $(f_{T},\pi_X^!L)$ is weakly
$\Psi$-good, as $\pi_X^!L\simeq i_{X_{T}*}p_{X_F}^!M$ by base change. By
Lemma \ref{l.wg2} (1), $(f_{T'},r_X^!p_X^!L)$ is weakly $\Psi$-good. By
Lemma \ref{l.wg2} (2), $(f_{S'},p_X^!L)$ is weakly $\Psi$-good.

For the general case, by Orgogozo's theorem (Theorem \ref{t.Org}), there
exists a modification $g\colon T\to S$ such that $(f_{T},g^*_X L)$ is
$\Psi$-good, hence weakly $\Psi$-good. Let $j\colon U\to T$ be the
complement of the exceptional locus $S_1\subseteq T$. Then $\Cone(j_!
(L|_U)\to g^*_X L)$ and $\Cone(j_! (L|_U)\to g^!_X L)$ are supported on
$X\times_S S_1$. By the special case above, we are reduced to proving the
theorem for $S_1$.

Repeating this process, we obtain a sequence $S=S_0\leftarrow S_1\leftarrow
S_2\leftarrow \dotsm$, each $S_{i+1}$ being the exceptional locus of a
modification of $S_i$. It remains to show that $S_n$ is empty for $n\gg 0$.
Assume the contrary. Up to replacing $S$ by an \'etale cover, we may assume
that $S$ admits a dimension function $\delta_0$ \cite[XIV Th\'eor\`eme
2.3.1]{ILO}. We equip $S_n$ with the induced dimension function $\delta_n$
\cite[XIV Corollaire 2.5.2]{ILO}. There exists a sequence of generic points
$\eta_i$ of $S_i$ such that $\eta_i$ specializes to the image of
$\eta_{i+1}$ in $S_i$. Then $\delta_{i+1}(\eta_{i+1})<\delta_i(\eta_i)$, so
that $\delta_i(\eta_i)\to -\infty$ as $i\to \infty$. Let $\eta'_i$ denote
the image of $\eta_i$ in $S$. Then $\delta_0(\eta'_i)\le \delta_i(\eta_i)$.
Thus $\delta(\eta'_i)\to -\infty$. Each $\eta'_{i+1}$ is a specialization of
$\eta'_i$. This contradicts the assumption that $S$ is Noetherian.
\end{proof}

\begin{remark}\label{r.Gabber}
We originally proved Theorem \ref{t.good} under the additional assumption
that $S$ is finite-dimensional. The argument of dimension function above
which allows to remove this assumption is due to Gabber.
\end{remark}

\section{Nearby cycles and duality}\label{s.3+}
In this section, we prove Theorem \ref{t.Psis} on the commutation of the
sliced nearby cycle functor with duality, and deduce Theorem \ref{t.Psisc}.
We then give applications to local acyclicity (Corollary \ref{c.DLA}) and
singular support (Corollary \ref{c.ss}).

\subsection{Duality}\label{s.5.1}
Let $S$ be a coherent scheme and let $\Lambda$ be a torsion commutative
ring. We fix $K_S\in D(S,\Lambda)$, not necessarily dualizing. For $a\colon
X\to S$ separated of finite type, we take $K_X=Ra^! K_S$. For a point $s$ of
$S$ with values in a field, we take $K_{X_s\ttimes_{s_0}
S_{(s)}}=R(a_s\ttimes_{s_0} \id)^! K_{S_{(s)}}$, where $K_{S_{(s)}}$ is the
restriction of $K_S$. Note that $\Psi_{\id_S}^s K_S\simeq K_{S_{(s)}}$.
Applying \eqref{e.Psis4} to $b=\id_S$ and $K_S$, we obtain $\Psi_a^s K_X\to
K_{X_s\ttimes_{s_0} S_{(s)}}$. Composing with \eqref{e.PsiHomS}, we get a
natural transformation
\begin{equation}\label{e.A}
A_{a}^s\colon \Psi_a^s D_{X} \to D_{X_s\ttimes_{s_0}
S_{(s)}}\Psi_a^s,
\end{equation}
where $D_X=R\cHom(-,K_X)$, $D_{X_s\ttimes_{s_0}
S_{(s)}}=R\cHom(-,K_{X_s\ttimes_{s_0} S_{(s)}})$.

By Remark \ref{r.Psiu}, $A_a^s$ is the composition of $i^*A_a$ with
\begin{equation}\label{e.iD}
i^*D_{X\atimes_S S}\to D_{X_s\ttimes_{s_0}
S_{(s)}} i^*,
\end{equation}
where $i\colon X_s\ttimes_{s_0} S_{(s)}\to X\atimes_S S$ and
\[A_a\colon \Psi_aD_{X}\simto D_{X\atimes_S S} \Psi_a\]
is the trivial duality. Here $D_{X\atimes_S S}\colonequals
R\cHom(-,K_{X\atimes_S S})$, $K_{X\atimes_S S}\colonequals \Psi_a K_X$.

In the rest of Subsection \ref{s.5.1}, let $S$ be a Noetherian scheme and
let $\Lambda$ be a Noetherian commutative ring with $m\Lambda=0$ for some
$m$ invertible on $S$. In this case, $K_{X_s\ttimes_{s_0} S_{(s)}}\simeq
(Ra_s^!\Lambda)\boxtimes^L K_{S_{(s)}}$ by Corollary \ref{c.Kunnup}.

\begin{theorem}\label{t.Psis}
Assume $S$ excellent. Let $a\colon X\to S$ be a separated morphism of finite
type and let $K\in D^b_c(S,\Lambda)$. Let $L\in D^-_c(X,\Lambda)$ such that
$(a,L)$ is $\Psi$-good and $\Psi_a(D_XL)$ commutes with smooth base change.
Then, for every point $s$ of $S$ with values in a field, the map $A_a^s(L)$
is an isomorphism.
\end{theorem}

We refer to Example \ref{e.Psigood} for examples of $\Psi$-good pairs. As
$D_X=R\cHom(-,a^! K_S)$, the theorem can be seen as a dual of the projection
formula for $R\Psi$ \cite[Proposition A.6]{IZ}: if $\Psi_a L$ commutes with
finite base change, then
\[
    R\Psi_a L\otimes^L p_2^* M \simto R\Psi_a(L\otimes^L a^* M).
\]

\begin{remark}\leavevmode\label{r.Psis}
\begin{enumerate}
\item If $(a,L)$ satisfies the assumptions of the theorem, then for any
    morphism $g\colon T\to S$ separated and of finite type, the same holds
    for $(a_T,g_X^* L)$. Indeed, $R\Psi_{a_T}(D_{X_T}g_X^* L)$ commutes
    with smooth base change by Lemma \ref{l.Psibc!2} (1), since
    $D_{X_T}g_X^* L\simeq g_X^! D_X L$.

\item If $X\xrightarrow{f} Y \xrightarrow{b} S$ are separated morphisms of
    finite type with $f$ proper and if $(bf,L)$ satisfies the assumptions
    of the theorem, the same holds for $(b,f_*L)$ by Lemma \ref{l.Psibc}
    (3).
\end{enumerate}
\end{remark}

For the proof of Theorem \ref{t.Psis}, we need the following compatibilities
between $A_a^s$ and the constructions of Subsections \ref{s.4.3} and
\ref{s.4.4}.

\begin{lemma}\label{l.Ag}
For any separated morphism $g\colon T\to S$ of finite type and any geometric
point $t\to T$, the diagram
\begin{equation}\label{e.Ag}
\xymatrix{(\id\ttimes g_{(t)})^!\Psi_a^t D_X L\ar[r]^{\BC^{!t}_{a,g}(D_X L)}\ar[d]_{(\id\ttimes g_{(t)})^!A_a^t(L)} &
\Psi^t_{a_T}g_X^! D_X L\ar@{-}[r]^\sim & \Psi^t_{a_T} D_{X_T} g_X^* L\ar[d]^{A_{a_T}^t(g_X^* L)}\\
(\id\ttimes g_{(t)})^!D_{X_t\ttimes S_{(t)}}\Psi_a^tL\ar[r]^{\delta\Psi_a^tL} & D_{X_t\ttimes T_{(t)}}(\id\ttimes g_{(t)})^* \Psi_a^t L
& D_{X_t\ttimes T_{(t)}}\Psi^t_{a_T}g_X^* L\ar[l]_{D_{X_t\ttimes T_{(t)}}\BC^t_{a,g}(L)}}
\end{equation}
commutes. Here $\delta\colon (\id\ttimes g_{(t)})^!D_{X_t\ttimes S_{(t)}}\to
D_{X_t\ttimes T_{(t)}}(\id\ttimes g_{(t)})^*$ is \eqref{e.localshr}.
\end{lemma}

\begin{proof}
We may assume that $g$ is smooth or a closed immersion. By Remark
\ref{r.ggg}, with the notation of \eqref{e.topoi}, the diagram decomposes
into
\[\xymatrix@C=1em{(\id\ttimes g_{(t)})^! i^* \Psi_a D_X  \ar[r]^\gamma_\sim\ar[d]^{A_a}_{\simeq}
& i_{T,t}^*(g_X\atimes_g g)^! \Psi_a D_X \ar[r]^{\BC^!_{a,g}} \ar[d]^{A_a}_\simeq
& i_{T,t}^* \Psi_{a_T} g_X^! D_X  \ar@{-}[r]^\sim
& i_{T,t}^* \Psi_{a_T} D_{X_T} g_X^* \ar[d]^{A_{a_T}}_\simeq\\
(\id\ttimes g_{(t)})^! i^* D_{X\atimes_S S} \Psi_a  \ar[r]^\gamma_\sim\ar[d]^{\eqref{e.iD}}
& i_{T,t}^*(g_X\atimes_g g)^! D_{X\atimes_S S} \Psi_a  \ar[r]^\epsilon
& i_{T,t}^* D_{X_T\atimes_T T}(g_X\atimes_g g)^*\Psi_a \ar[d]^{\eqref{e.iD}}
& i_{T,t}^* D_{X_T\atimes_T T} \Psi_{a_T}g_X^*\ar[d]^{\eqref{e.iD}}\ar[l]_-{\BC_{a,g}}\\
(\id\ttimes g_{(t)})^! D_{X_t\ttimes S_{(t)}} i^*\Psi_a  \ar[r]^\delta
& D_{X_t\ttimes T_{(t)}} (\id\ttimes g_{(t)})^* i^* \Psi_a \ar@{-}[r]^\sim
& D_{X_t\ttimes T_{(t)}} i_{T,t}^*(g_X\atimes_g g)^* \Psi_a
& D_{X_t\ttimes T_{(t)}} i_{T,t}^*\Psi_{a_T} g_X^*,\ar[l]_-{\BC_{a,g}}}
\]
where $\epsilon\colon (g_X\atimes_g g)^! D_{X\atimes_S S}\to D_{X_T\atimes_T
T}(g_X\atimes_g g)^*$ is given by
\begin{multline*}
(g_X\atimes_g g)^! R\cHom(M,\Psi_a K_X) \xrightarrow{\zeta} R\cHom((g_X\atimes_g g)^*M, (g_X\atimes_g g)^!\Psi_a K_X) \\
\xrightarrow{\BC^!_{a,g}(K_X)} R\cHom((g_X\atimes_g g)^*M,\Psi_{a_T} g_X^! K_X).
\end{multline*}
Here $\zeta$ is the restriction map for $g$ smooth and the canonical
isomorphism for $g$ a closed immersion (in this case, as in the proof of
Lemma \ref{l.Psi!}, $g_X\atimes_{g}g$ can be identified with the closed
immersion $\id_X\atimes_S g$). The inner squares commute by construction.
\end{proof}

\begin{remark}\label{r.diagram}
We have $R(\id\ttimes g_{(t)})^! K_{X_t\ttimes S_{(t)}}\simeq K_{X_t\ttimes
T_{(t)}}$ by Remark \ref{r.localdesc}. Thus, by Lemma \ref{l.localshr},
$\delta$ is an isomorphism on $D^-_c$. It follows that for $L$ as in Theorem
\ref{t.Psis}, the horizontal arrows of \eqref{e.Ag} are isomorphisms.
\end{remark}

\begin{lemma}\label{l.Af}
Let $X\xrightarrow{f} Y\xrightarrow{b} S$ be morphisms of schemes with $Y$
coherent and $f$ separated of finite type. For any geometric point $s\to S$,
the following diagrams commute:
\begin{gather}\label{e.Afd}
\xymatrix{\Psi_b^s D_Y f_!L\ar[d]_{A_b^s(f_!L)}\ar@{-}[r]^{\sim}
& \Psi^s_b f_* D_XL\ar[r]^{\eqref{e.Psis2}} & (f_s\ttimes \id)_*\Psi_{bf}^s D_XL\ar[d]^{(f_s\ttimes \id)_*A_{bf}^s(L)}\\
D_{Y_s\ttimes S_{(s)}}\Psi^s_b f_!L\ar[r]^{\eqref{e.Psis3}} & D_{Y_s\ttimes S_{(s)}}(f_s\ttimes \id)_! \Psi^s_{bf}L\ar@{-}[r]^{\eqref{e.1.13.2}}_\sim & (f_s\ttimes \id)_* D_{X_s\ttimes S_{(s)}}\Psi_{bf}^sL}\\
\label{e.Afu}\xymatrix{\Psi_{bf}^s D_Y f^* M\ar[d]_{A_{bf}^s(f^*M)}\ar@{-}[r]^\sim &\Psi_{bf}^s f^!D_Y M\ar[r]^{\eqref{e.Psis4}}
& (f_s\ttimes \id)^! \Psi^s_bD_X M\ar[d]^{(f_s\ttimes \id)^!A_b^s(M)}\\
D_{X_s\ttimes S_{(s)}} \Psi_{bf}^s f^* M\ar[r]^{\eqref{e.Psis1}} & D_{X_s\ttimes S_{(s)}}(f_s\ttimes \id)^* \Psi^s_b M\ar@{-}[r]^{\eqref{e.1.13.1}}_\sim &(f_s\ttimes \id)^! D_{Y_s\ttimes S_{(s)}} \Psi^s_b M.}
\end{gather}
\end{lemma}

For $f$ proper (resp.\ smooth), the horizontal arrows of \eqref{e.Afd}
(resp.\ \eqref{e.Afu}) are isomorphisms by Remark \ref{r.Psis0} (2) (resp.\
Remark \ref{r.Psis0} (1) and Lemma \ref{l.trace2}).

\begin{proof}
By \eqref{e.Psis3d}, with the notation of \eqref{e.4.3}, the diagram
\eqref{e.Afd} decomposes into
\[\xymatrix@C=1em{i_Y^*\Psi_bD_Y f_!  \ar@{-}[r]^\sim\ar[d]_{\simeq}^{A_b} & i_Y^*\Psi_b f_* D_X \ar@{-}[r]^\sim & i_Y^*(f\atimes_S \id)_* \Psi_{bf} D_X\ar[r]^{\eqref{e.Psis02}}\ar[d]^{A_{bf}}_\simeq & (f_s\ttimes \id)_* i_X^*\Psi_{bf} D_X\ar[d]^{A_{bf}}_\simeq\\
i_Y^* D_{Y\atimes_S S} \Psi_b f_! \ar[r]^-{\alpha'}\ar[d]^{\eqref{e.iD}} & i_Y^*D_{Y\atimes_S S} (f\atimes_S \id)_!\Psi_{bf} \ar@{-}[r]^\gamma_\sim\ar[d]^{\eqref{e.iD}} & i_Y^* (f\atimes_S \id)_*D_{X\atimes_S S}\Psi_{bf} \ar[r]^{\eqref{e.Psis02}} & (f_s\ttimes \id)_* i_X^*D_{X\atimes_S S}\Psi_{bf}\ar[d]^{\eqref{e.iD}}\\
D_{Y_s\ttimes S_{(s)}} i_Y^*\Psi_b f_! \ar[r]^-{\alpha'} & D_{Y_s\ttimes S_{(s)}}i_Y^*(f\atimes_S \id)_!\Psi_{bf}\ar@{-}[r]^{\beta'}_\sim & D_{Y_s\ttimes S_{(s)}}(f_s\ttimes \id)_! i_X^*\Psi_{bf}\ar[r]_\sim^{\eqref{e.1.13.2}} & (f_s\ttimes \id)_* D_{X_s\ttimes S_{(s)}}i_X^*\Psi_{bf},}
\]
where $\gamma\colon D_{Y\atimes_S S} (f\atimes_S \id)_!\simeq (f\atimes_S
\id)_*D_{X\atimes_S S}$ is given by
\begin{multline*}
R\cHom((f\atimes_S \id)_!N,\Psi_b K_Y)\xrightarrow[\sim]{\eqref{e.1.13.2}}(f\atimes_S \id)_*R\cHom(N,(f\atimes_S \id)^!\Psi_b K_Y)\\
\xrightarrow[\sim]{\alpha^{-1}(K_Y)} (f\atimes_S \id)_*R\cHom(N,\Psi_{bf}f^!K_Y).
\end{multline*}
Here $\alpha$ is defined in \eqref{e.alpha}. The inner squares commute by
construction.

The commutativity of \eqref{e.Afu} can be proved similarly, using
\eqref{e.Psis4d}.
\end{proof}

\begin{proof}[Proof of Theorem \ref{t.Psis}]
Parts of the proof are similar to \cite[Sections 4.4, 5.1]{Org}, but the
third induction step below is in the opposite direction. By Lemma
\ref{l.local}, up to replacing $S$ by the strict localization at a geometric
point above $s$, we may assume that $S$ is finite-dimensional and $s$ is a
geometric point. By Lemmas \ref{l.cd} (2), \ref{l.Ag} and Remark
\ref{r.diagram}, if the theorem holds for $(a_T,g_X^*L)$ for a proper
surjective morphism $g\colon T\to S$, then it holds for $(a,L)$. In
particular, we may alter $S$.

We proceed by a triple induction. First we proceed by induction on the
dimension $d_S$ of $S$. Up to replacing $S$ by an irreducible component, we
may assume $S$ integral, of generic point $\eta$. For $S$ empty the
assertion is trivial. For each $d_S\ge 0$, we proceed by induction on the
dimension $d_{X_\eta}$ of the generic fiber $X_\eta$ of~$a$. Note that the
assertion holds if $X_\eta$ is empty. Indeed, if $T\subseteq S$ denotes the
schematic image of $a$, then for $s\to S-T$ the source and target of
$A_a^s(L)$ are both zero. For $t\to T$, the source and target of $A_a^t(L)$
are supported on $X_t\ttimes T_{(t)}$. By Lemma \ref{l.Ag}, the restriction
of $A_a^t(L)$ to $X_t\ttimes T_{(t)}$ can be identified with $A_{a_T}^t(L)$,
which is an isomorphism by the induction hypothesis on $d_S$.

Assume $d_{X_\eta}\ge 0$. Note that if $L$ is supported on $Y\subseteq X$
such that $Y_\eta\subseteq  X_\eta$ is nowhere dense, then the assertion
holds. Indeed, if $L=i_*M$, where $i\colon Y\to X$ is the inclusion, then
$A^s_a(L)$ can be identified with $(i_s\ttimes \id)_*A^s_{ai}(M)$ by Lemma
\ref{l.Af}, and the induction hypothesis on $d_{X_\eta}$ applies to $ai$.
This applies in particular if $Y\subseteq X$ is nowhere dense.

For any alteration $g\colon T\to S$, $g_!$ has cohomological amplitude $\le
2d_g$, where $d_g$ denotes the maximum dimension of the fibers of $g$, so
that $g^!$ has cohomological amplitude $\ge -2d_g\ge -2d_S$ by dimension
formula \cite[(5.6.5.1)]{EGAIV}. Thus, up to replacing $K$ by a shift, we
may assume that for every alteration $g\colon T\to S$, $g^! K_S \in D^{\ge
0}$. We proceed by induction on $n$ to show that for every $L\in D^{\le
0}_c(X,\Lambda)$ satisfying the assumptions of the theorem, the cone of
$A_{a}^s(L)$ belongs to $D^{\ge n}$, and the same holds with $S$ replaced by
an alteration. As above, $a^!$ has cohomological amplitude $\ge -2d_a$,
where $d_a$ denotes the maximum dimension of the fibers of $a$. Thus the
source and target of $A_{a}^s(L)$ both belong to $D^{\ge -2d_a}$, and the
assertion is trivial for $n=-2d_a-1$.

Recall that every $C\in \Shv_c(X,\Lambda)$ is Noetherian. Thus, for any
epimorphism $M\to C$, there exists $u_!\Lambda\to M$ with $u\colon U\to X$
\'etale, separated and of finite type, such that the composition
$u_!\Lambda\to C$ is an epimorphism. By \cite[Lemma 13.2.1 (b)]{KS}, we may
assume that $L^q=0$ for $q>0$ and $L^0$ has the form $u_!\Lambda$. Note that
$D_X(u_!\Lambda)\simeq u_*K_U\in D^b_c$. By Orgogozo's theorem (Theorem
\ref{t.Org}) and Theorem \ref{t.good}, up to modifying $S$, we may assume
that $(a,u_!\Lambda)$ satisfies the assumptions of the theorem. It follows
that $(a,\sigma^{\le -1} L)$ satisfies the assumptions of the theorem, where
$\sigma^{\le -1} L$ denotes the naive truncation of $L$. By induction
hypothesis on $n$, we have $\Cone(A_a^s(\sigma^{\le -1} L))\in D^{\ge n}$.
Thus it suffices to show that $A_a^s(u_! \Lambda)$ is an isomorphism.

Choose a compactification $X\xrightarrow{f} Y\xrightarrow{b} S$ of $a$ with
$\dim Y_\eta=\dim X_\eta$.  By Zariski's main theorem, $fu$ admits a
factorization $U\xrightarrow{j} V\xrightarrow{v} Y$ with $v$ finite and $j$
a dominant open immersion. Let $V'$ be the disjoint union of the irreducible
components of $V$ and let $v'\colon V'\to X$. Then the composition
$(fu)_!\Lambda\to v_*\Lambda\to v'_*\Lambda$ is a monomorphism and the
cokernel is supported on a nowhere dense closed subset. Up to modifying $S$,
we may assume that $(b, (fu)_!\Lambda)$, $ (bv',\Lambda)$ and (consequently)
$(b,v'_*\Lambda)$ satisfy the assumptions of the theorem. By Lemma
\ref{l.Af}, $A^s_{a}(u_!\Lambda)$ can be identified with $(f_s\ttimes \id)^*
A_b^s((fu)_! L)$. Thus it suffices to show that $A_b^s(v'_*\Lambda)$ is an
isomorphism. By Lemma \ref{l.Af}, $A_b^s(v'_*\Lambda)$ can be identified
with $(v'_s\times\id)_*A_{bv'}^s(\Lambda)$. Changing notation, we are
reduced to showing that $A_a^s(\Lambda)$ is an isomorphism for $a$ proper,
$X$ integral, and $(a,\Lambda)$ satisfying the assumptions of the theorem.

By \cite[Lemme 4.7]{Org}, up to altering $S$, there exists an alteration
$X'\coprod X''\to X$ such that $X'\to S$ is a plurinodal morphism and
$X''\to S$ is non-dominant. Let $p\colon X'\to X$. There exists an open
immersion $j\colon U\to X$ with $j_\eta$ dominant such that $p_U$ is finite
flat and surjective. Consider the maps
\[\xymatrix{j_!p_{U*}\Lambda \ar[r]\ar[d] & j_!\Lambda\ar[d]\\
Rp_*\Lambda &\Lambda.}
\]
The horizontal map is induced by the trace map and is surjective. The cones
of the vertical maps are supported on closed subsets of $X$ having nowhere
dense intersections with $X_\eta$. Up to modifying $S$, we may assume that
the pairs $(a,j_!\Lambda)$, $(a,j_!p_{U*}\Lambda)$, $(ap,\Lambda)$ and
(consequently, by Remark \ref{r.Psis} (2)) $(a,Rp_*\Lambda)$ satisfy the
assumptions of the theorem. Thus it suffices to show that
$A_a^s(Rp_*\Lambda)$ is an isomorphism. By Lemma \ref{l.Af},
$A_a^s(Rp_*\Lambda)$ can be identified with
$(p_s\times\id)_*A_{ap}^s(\Lambda)$. Changing notation, we are reduced to
showing that $A_a^s(\Lambda)$ is an isomorphism for $a$ plurinodal and
$(a,\Lambda)$ satisfying the assumptions of the theorem.

If $d_{X_\eta}=0$, then $a$ is an isomorphism and $\Psi_a^s$ is the
restriction from $S$ to $S_{(s)}$, so that $A_a^s(\Lambda)$ can be
identified with the identity on $K_{S_{(s)}}$. Assume $d_{X_\eta}\ge 1$.
Then $a$ decomposes into $X\xrightarrow{f} Y\xrightarrow{b} S$, where
$\dim(Y_\eta)=d_{X_\eta}-1$ and $f$ is a projective flat curve with
geometric fibers having at most ordinary quadratic singularities. Up to
modifying $S$, we may assume that $(b,\Lambda)$ satisfies the assumptions of
the theorem. By Lemma \ref{l.Af}, $(f_s\ttimes \id)_* A^s_{a}(\Lambda)$ can
be identified with $A^s_{b}(f_*\Lambda)$. Since $(b,f_*\Lambda)$ satisfies
the assumption of the theorem, $A^s_{b}(f_*\Lambda)$ is an isomorphism by
the induction hypothesis on $d_{X_\eta}$, so that it suffices to show that
$\Cone(A^s_{a}(\Lambda))$ is supported on $(X-U)_s\ttimes S_{(s)}$, where
$U$ is the smooth locus of $f$ (with $X-U$ finite over $Y$). By Lemma
\ref{l.Af}, on $U_s\ttimes S_{(s)}$, $A^s_{a}(\Lambda)$ coincides with
$(f_s\ttimes\id)^* A^s_{b}(\Lambda)(1)[2]$, which is an isomorphism by the
induction hypothesis on $d_{X_\eta}$.
\end{proof}

\begin{proof}[Proof of Theorem \ref{t.Psisc}]
By Orgogozo's theorem (Theorem \ref{t.Org}) and Theorem \ref{t.good}, there
exists a modification $g\colon S'\to S$ such that $(a_{S'},L|_{X_{S'}})$ is
$\Psi$-good and $(a_{S'},g_X^! D_X L)$ is weakly $\Psi$-good. Since
$D_{X_{S'}}(L|_{X_{S'}})\simeq g_X^! D_X L$, Theorem \ref{t.Psis} applies to
$(a_T,L|_{X_T})$ for any $T\to S'$ separated of finite type by Remark
\ref{r.Psis} (1).
\end{proof}

\begin{cor}
Under the assumptions of Theorem \ref{t.Psisc}, there exists an open
subscheme $U\subseteq S$ of complement of codimension $\ge 2$, such that for
every morphism $T\to U$ separated of finite type, and for every point $t\in
T$, the map $A_{a_T}^t(L|_{X_T})$ is an isomorphism.
\end{cor}

\begin{proof}
By Corollary \ref{c.Orgco}, there exists an open subscheme $U\subseteq S$ of
complement of codimension $\ge 2$, such that $(a_U,L|_{X_{U}})$ and
$(a_U,(D_X L)|_{X_U})$ are $\Psi$-good and weakly $\Psi$-good. Since
$D_{X_{U}}(L|_{X_{U}})\simeq (D_X L)|_{X_U}$, we then conclude by Theorem
\ref{t.Psis} and Remark \ref{r.Psis} (1).
\end{proof}

In the case where $K_S$ is a dualizing complex, Theorem \ref{t.Psis} has the
following dual.

\begin{cor}\label{c.DPsis}
Assume $S$ excellent equipped with a dimension function and let $K_S$ be a
dualizing complex for $D_\cft(S,\Lambda)$. Let $a\colon X\to S$ be a
separated morphism of finite type and let $L\in D^b_c(X,\Lambda)$ such that
$(a,L)$ is weakly $\Psi$-good and $(a,D_X L)$ is $\Psi$-good. Assume either
$L\in D_{\cft}$ or $\Lambda$ Gorenstein. Then for every point $s$ of $S$
with values in a field, $A_a^s(L)$ is an isomorphism.
\end{cor}

\begin{proof}
The assumption that $L\in D_{\cft}$ or $\Lambda$ Gorenstein implies that
$L\to D_X D_X L$ is an isomorphism and $M\to D_{X_s\ttimes_{s_0} S_{(s)}}
D_{X_s\ttimes_{s_0} S_{(s)}} M$ is an isomorphism for $M=\Psi_a^s D_X L$ by
Proposition \ref{p.dual}. (For $L\in D_\cft$, $D_XL\in D_\cft$, hence $M\in
D_\cft$ by \cite[Remarque 8.3]{Org}.) Thus $A_a^s(D_X L)$ is an isomorphism
Theorem \ref{t.Psis}. The corollary follows from the following formal result
(see for example \cite[Constructions A.4.5, A.4.6]{SZ}).
\end{proof}

\begin{lemma}
The square
\[\xymatrix{\Psi_a^s D\ar[r]^{A_a^s}\ar[d] & D\Psi_a^s\\
DD\Psi_a^s D\ar[r]^{DA_a^s D} & D\Psi_a^s DD\ar[u],}
\]
where the vertical arrow are induced by the evaluation maps $\id\to DD$, is
commutative.
\end{lemma}

In the case where $K_S$ is a dualizing complex,  Theorem \ref{t.Psisc} has
the following dual.

\begin{cor}\label{c.DPsisc}
Assume $S$ excellent equipped with a dimension function and let $K_S$ be a
dualizing complex for $D_\cft(S,\Lambda)$. Let $a\colon X\to S$ be a
separated morphism of finite type and let $L\in D^b_{c}(X,\Lambda)$. Assume
either $L\in D_{\cft}$ or $\Lambda$ Gorenstein. Then there exists a
modification $g\colon S'\to S$ such that for every morphism $T\to S'$
separated of finite type, and for every point $t$ of $T$, $A_{a_T}^t(h_X^!
L)$ is an isomorphism. Here $h$ denotes the composition $T\to
S'\xrightarrow{g} S$.
\end{cor}

\begin{proof}
By Orgogozo's theorem (Theorem \ref{t.Org}) and Theorem \ref{t.good}, there
exists $g$ such that $(a_{S'},g_X^!L)$ is weakly $\Psi$-good and
$(a_{S'},g_X^*D_XL)$ is $\Psi$-good. Then $(a_{T},h_X^!L)$ is weakly
$\Psi$-good and $(a_T,D_{X_T}h_X^! L)$ is $\Psi$-good, as $D_{X_T}h_X^!
L\simeq h_X^*D_X L$. We conclude by Corollary \ref{c.DPsis}.
\end{proof}

\begin{remark}\label{r.fail}
The analogue of Theorem \ref{t.Psisc} does not hold for the restriction to
shreds or local sections (Remark \ref{r.local}). In fact, the shredded
nearby cycle functor $\Psi^s_s=(i^s)^*$ typically does not commute with
duality, even if $X=S$. Moreover, for $x\to X$ a geometric point above a
geometric generic point $s\to S$, the local section $x\ttimes S_{(s)}$ is a
point and the restriction $(\Psi^s L)|_{x\ttimes S_{(s)}}$ can be identified
with $L_x$, and $L\mapsto L_x$ typically does not commute with duality, even
if $S$ is a point.
\end{remark}

\subsection{Vanishing cycles and local acyclicity}\label{s.van}
The functor $L\Co$ studied in Section \ref{s.2} admits the following
generalization.

\begin{remark}\label{r.glue}
Let $X$ be a topos glued from an open subtopos $U$ and a closed subtopos
$Y$. Let $j\colon U\to X$ and $i\colon Y\to X$ be the embeddings. Then $i$
admits a left adjoint $\pi\colon X\to Y$ if and only if the gluing functor
$p_*=i^*j_*\colon U\to Y$ admits an exact left adjoint $p^*$. For the ``only
if'' part we take $p=\pi j$. For the ``if'' part, we take $\pi^* \cG$ to be
$(\cG,p^*\cG,\varphi)$, where $\varphi\colon \cG\to p_*p^*\cG$ is the
adjunction.

Assume that the above conditions hold. Sheaves $\cF$ on $X$ are triples
$(\cF_Y,\cF_U,\phi)$, where $\phi\colon p^*\cF_Y\to \cF_U$ is a morphism.
Between $\Shv(-,\Lambda)$, we have adjoint functors
\[ \Co \dashv j_! \dashv j^* \dashv j_*,\quad \pi^*\dashv\pi_*=i^*\dashv i_* \dashv i^!,\]
where $\Co \cF=\Coker(\phi\colon p^*\cF_Y\to \cF_U)$. Between derived
categories $D(-,\Lambda)$, we have adjoint functors
\[L\Co \dashv j_!\dashv j^* \dashv Rj_*, \quad .\]
For $M\in D(X,\Lambda)$, $L\Co M$ is computed by $\Co M'$, where $M'\to M$
is a quasi-isomorphism and $M'^q=(M'^q_Y,M'^q_U, \phi^q)$ with $\phi^q$ a
monomorphism for every $q$. In fact, we can take $M'$ to be $\Ker(M\oplus
\Cone(\id_{\pi^*\pi_* M})\to \Cone(\id_{i_*i^*M}))$, where the map is
induced by the adjunctions $\pi^*\pi_* M\to M\to i_*i^*M$, and take $M'\to
M$ to be the map induced by projection.

Generalizing \eqref{e.tr1}, we have a distinguished triangle
\[\pi^*\pi_*M\to M\to j_!L\Co M\to \pi^*\pi_*M[1].\]
\end{remark}

Let $(S,s)$ be a Henselian local pair. Let $Y$ be a locally coherent topos
over $s$. Let $S^\circ\colonequals S-\{s\}$. Remark \ref{r.glue} applies to
the inclusions $j\colon Y\ttimes_{s} S^\circ \to Y\times_{s} S$ and $i\colon
Y\simeq s\ttimes_s Y\to Y\ttimes_{s} S$, with $\pi=p_1\colon Y\ttimes_s S\to
Y$ given by the first projection. In particular, we have the functor
\[L\Co\colon D(Y\ttimes_s S)\to D(Y\ttimes_s S^\circ).\]
The following remark is not needed in this paper.

\begin{remark}
If $S$ is Noetherian and $m\Lambda=0$ with $m$ invertible on $S$, then
$Rj_*$ has finite-cohomological dimension and admits a right adjoint
$\Co\spcheck$. If, moreover, $S$ is excellent equipped with a dimension
function, $K_S$ is a dualizing complex for $D_\cft(S,\Lambda)$, and $Y$ is a
scheme separated of finite type over $s$, then we have $D_{Y\ttimes
S^\circ}(L\Co)\simeq \Co\spcheck D_{Y\ttimes S}$, similarly to Theorem
\ref{t.LPhi}. However, unlike the case of Section \ref{s.2}, $\Co\spcheck$
is very different from $L\Co$ in general. For example, if $S$ is regular,
$Y=s$, $v\colon T\to S$ is the inclusion of a nonempty regular closed
subscheme of codimension $c$, then $\Co\spcheck(v_*\Lambda)\simeq
Ru_*\Lambda(c)[2c-1]$ while $L\Co(v_*\Lambda)\simeq u_!\Lambda[1]$, where
$u\colon S-T\to S^\circ$.
\end{remark}

Let $S$ be an arbitrary scheme. For a morphism of schemes $a\colon X\to S$
and a point $s$ of $S$ with values in a field, we define the sliced
vanishing cycle functor $\Phi^s_a$ to be the composition
\[D(X,\Lambda) \xrightarrow{R\Psi^s_a} D(X_s\ttimes_{s_0} S_{(s)},\Lambda)\xrightarrow{L\Co} D(X_s\ttimes_{s_0}S_{(s)}^\circ,\Lambda),\]
where $s_0$ is the closed point of $S_{(s)}$.

\begin{cor}\label{c.DLA}
Let $a\colon X\to S$ be a morphism of finite type of excellent schemes, with
$S$ regular. Let $\Lambda$ be a Noetherian commutative ring such that
$m\Lambda=0$ for some $m$ invertible on~$S$. Let $K_X$ be a dualizing
complex for $D_{\cft}(X,\Lambda)$ and let $L\in D^b_c(X,\Lambda)$. Assume
either $L\in D_{\cft}$ or $\Lambda$ Gorenstein. Assume $(a,L)$ is
universally locally acyclic. Then $(a,D_XL)$ is universally locally acyclic.
Here $D_X=R\cHom(-,K_X)$.
\end{cor}

This gives an affirmative answer to a question of Illusie. Note that since
$S$ is regular, $\Lambda_S$ is a dualizing complex for
$D_{\cft}(S,\Lambda)$. See also \cite[B.6 2)]{BG} for $S$ smooth over a
field.

\begin{proof}
We may assume $a$ separated. We have $K_X\simeq a^! \Lambda_S\otimes^L M$
for some invertible object $M$ of $D_{\cft}(X,\Lambda)$. Since $D_X L\simeq
R\cHom(L\otimes^L M,a^!\Lambda_S)$ with $L\otimes^L M$ universally locally
acyclic, we may assume $K_X=a^!\Lambda_S$. By Corollary \ref{c.DPsisc},
there exists a modification $S_0\to S$ such that for $T\to S_0$ separated of
finite type, and for every geometric point $t\to T$, we have
\[
\Psi^t ((D_XL)|_{X_T})\simeq \Psi^t D_{X_T}g_X^! L \xrightarrow[\sim]{A^t_{a_T}(g_X^! L)} D_{X_t\ttimes T_{(t)}}\Psi^t g_X^! L,
\]
where $g\colon T\to S$. Let $L_t=L|_{X_t}$. By Example \ref{e.Psigood} (1),
$(a,L)$ is $\Psi$-good and $\Psi^t L=p_1^* L_t\in D^b_c$. Thus, by Lemmas
\ref{l.Psibc!} and \ref{l.localshr} and biduality (Proposition
\ref{p.dual}),
\begin{multline*}
D_{X_t\ttimes T_{(t)}}\Psi^t g_X^! L\simeq D_{X_t\ttimes
T_{(t)}}(\id \ttimes g_{(t)})^!\Psi^t L\simeq D_{X_t\ttimes
T_{(t)}}(\id \ttimes g_{(t)})^! D_{X_t\ttimes S_{(t)}}D_{X_t\ttimes S_{(t)}}p_1^* L_t\\
\xleftarrow[\sim]{\eqref{e.localshr}} D_{X_t\ttimes
T_{(t)}}D_{X_t\ttimes
T_{(t)}}(\id \ttimes
g_{(t)})^*D_{X_t\ttimes S_{(t)}}p_1^* L_t\simeq (\id \ttimes
g_{(t)})^*D_{X_t\ttimes S_{(t)}}p_1^* L_t.
\end{multline*}
By K\"unneth formula (Corollary \ref{c.DKunn}), $D_{X_t\ttimes
S_{(t)}}p_1^*L_t\simeq p_1^*D_{X_t} L_t$. Thus $\Psi^t ((D_XL)|_{X_T})$ has
the form $p_1^* L'$. It follows that $\Phi_{a_T}((D_X L)|_{X_T})=0$. Thus,
for any morphism $S'\to S$ separated and of finite type, $\Phi_{a_{S'}}((D_X
L)|_{X_{S'}})=0$ by Lemma \ref{l.descent} applied to the \v Cech nerve of
$S_0\times_S S'\to S'$. We conclude by Remark \ref{r.Org}.
\end{proof}

Gabber showed that universal local acyclicity in Corollary \ref{c.DLA} is
equivalent to local acyclicity. He also gave a different proof of the case
$L\in D_{\cft}$ of Corollary \ref{c.DLA}, independent of Corollary
\ref{c.DPsisc}. We present Gabber's results in Section~\ref{s.Gabber}.

Generalizing constructions of Beilinson \cite{BeiSS}, Hu and Yang \cite{HY}
recently defined relative versions of singular support and weak singular
support over a Noetherian base scheme $S$, which exist and are equal on a
dense open subscheme of $S$. For $X\to S$ smooth of finite type and $L\in
D^b_c(X,\Lambda)$, the weak singular support $\rSS^w(L,X/S)$ is defined to
be the smallest element of $\cC^w(L,X/S)$ (if it exists), where
$\cC^w(L,X/S)$ denotes the set of closed conical subsets $C$ of the
cotangent bundle $T^*(X/S)$ such that $L$ is weakly micro-supported on $C$
relatively to $S$. $\rSS^w(L,X/S)$ exists if $(f,L)$ is universally locally
acyclic \cite[4.3, Proposition 4.5]{HY}.

The preservation of local acyclicity by duality implies the following
compatibility of weak singular support with duality.

\begin{cor}\label{c.ss}
Let $X\to S$ be a smooth morphism of regular excellent schemes. Let
$\Lambda$ and $K_X$ be as in Corollary \ref{c.DLA} and let $L\in
D^b_c(X,\Lambda)$. Assume either $L\in D_{\cft}$ or $\Lambda$ Gorenstein.
Then
\[\cC^w(L,X/S)=\cC^w(D_XL,X/S), \quad \rSS^w(L,X/S)=\rSS^w(D_XL,X/S).\]
(The second equality means that $\rSS^w(L,X/S)$ exists if and only if
$\rSS^w(D_XL,X/S)$ exists and the two are equal when they exist.)
\end{cor}

In the case where $S$ is the spectrum of a field, one recovers
\cite[Corollary 4.9]{Saito}.

\begin{proof}
The second equality follows from the first equality. For the first equality,
we show more generally that for every test pair $X\xleftarrow{h}
U\xrightarrow{g} Y$ of smooth schemes of finite type over $S$ with $h$
smooth of relative dimension $d$, $(g,h^*L)$ is locally acyclic if and only
if $(g,h^*D_XL)$ is locally acyclic. Since $(h^*D_XL)(d)[2d]\simeq D_Uh^*L$,
this follows from Corollaries \ref{c.DLA} and \ref{c.la}.
\end{proof}

\section{Local acyclicity (after Gabber)}\label{s.Gabber}

The results of this section are due to Ofer Gabber.

\begin{lemma}\label{l.descent}
Let $a\colon X\to S$ be a morphism of schemes and let $L\in D^+(X,\Lambda)$.
Assume that there exists a hypercovering $g_\bullet\colon S_\bullet\to S$
for the $h$-topology such that $\Phi_{a_{n}}(L|_{X_{n}})=0$ for all $n\ge
0$, where $a_n\colon X_n\to S_n$ denotes the base change of $a$ by $g_n$.
Then $\Phi_a(L)=0$.
\end{lemma}

\begin{proof}
Let $g_{X\bullet}$ be the base change of $g_\bullet$ to $X$ and let
$\vec{g}_\bullet=g_{X\bullet}\atimes_{g_\bullet} g_\bullet$. We have a
commutative diagram
\[\xymatrix{p_1^*\ar[d]\ar[r] & \vec{g}_{\bullet*}\vec{g}_\bullet^*p_1^*\ar@{-}[r]^{\sim} & \vec{g}_{\bullet*}p_{1\bullet}^*g_{X\bullet}^*\ar[d]\\
\Psi_a\ar[r] & \Psi_a g_{X\bullet *}g_{X\bullet}^* \ar@{-}[r]^\sim & \vec{g}_{\bullet*}\Psi_{a_{\bullet}}g_{X\bullet}^*.}
\]
By cohomological descent and oriented cohomological descent
\cite[XII$_{\textrm{A}}$ Th\'eor\`eme 2.2.3]{ILO}, the horizontal arrows are
isomorphisms on $D^+$. The right vertical arrow is an isomorphism on $L$ by
assumption. Thus the left vertical arrow is an isomorphism on $L$.
\end{proof}

\begin{lemma}\label{l.alg}
Let $S$ be a Noetherian scheme and let $U\subseteq S$ be a dense open
subset. For any $s\in S-U$, there exists an immediate Zariski generization
$t$ of $s$ in $S$ with $t\in U$.
\end{lemma}

This follows from \cite[Section 31, Lemma~1]{Matsumura}. We include a proof
for completeness.

\begin{proof}
We may assume $S$ local of center $s$. We proceed by induction on the
dimension $d$ of $S$. For $d=1$, any $t\in U$ works. For $d>1$, there are
infinitely many codimension $1$ points of $S$. Indeed, for $S=\Spec(R)$, by
Krull's principal ideal theorem, each proper principal ideal of $R$ is
contained in a height $1$ prime ideal, so that the maximal ideal $\fm$ of
$R$ is the union of all height $1$ prime ideals. On the other hand, by the
prime avoidance lemma, $\fm$ cannot be the union of finitely many such
primes, as $\fm$ has height $>1$. Let $x$ be a codimension $1$ point of $S$
that is not a maximal point of $S-U$. Then $x\in U$. We conclude by
induction hypothesis applied to the closure of $x$ in $S$.
\end{proof}

\begin{lemma}\label{l.points}
Let $f\colon X\to S$ be a morphism of finite type of Noetherian schemes, $T$
a subscheme of $S$, $(\bar x,\bar t)$ a point of $P=X\atimes_S T$ ($\bar x$
being a point of $X_\et$, $\bar t$ being a point of $T_\et$, and $\bar
t\rightsquigarrow f(\bar x)$ being a morphism of points of $S_{\et}$). Then
$(\bar x,\bar t)$ specializes to a point $(\bar x',\bar t')$ of $P$ such
that the image $x'\in X$ of $\bar x'$ is closed and $\bar t'\rightsquigarrow
f(\bar x')$ is an (\'etale) specialization of codimension $\le 1$.
\end{lemma}

\begin{proof}
The image $x\in X$ of $\bar x$ specializes to a closed point $x'\in X$ and
$\bar x$ specializes to a point $\bar x'$ of $X_\et$ above $x'$. Let $t$ be
the image of $\bar t \to S_{(f(\bar x))}\to S_{(f(\bar x'))}$ and let
$U=T\times_S \overline{\{t\}}$, which is locally closed in
$\overline{\{t\}}$ and contains $t$, hence open dense in $\overline{\{t\}}$.
If $U=\overline{\{t\}}$, it suffices to take $\bar t'=f(\bar x')$. Otherwise
$f(\bar x')\not\in U$ admits an immediate Zariski generization $t'\in U$ by
Lemma \ref{l.alg} and it suffices to take $\bar t'$ to be a geometric point
above $t'$.
\end{proof}

\begin{lemma}\label{l.constr}
Let $f\colon X\to S$ be a morphism of finite type of Noetherian schemes. The
points $(\bar x,\bar t)$ of $X\atimes_S S$ such that the image $x\in X$ of
$\bar x$ is locally closed and $\bar t\rightsquigarrow f(\bar x)$ is an
(\'etale) specialization of codimension $\le 1$, form a conservative family
$\cP$ for the category of constructible sheaves.
\end{lemma}

Note that by Chevalley's constructibility theorem, $x\in X$ locally closed
implies $f(x)\in S$ locally closed. Thus $x\in X$ is locally closed if and
only if $s=f(x)\in S$ is locally closed and $x$ is closed in the fiber
$X_s$.

\begin{proof}
Let $\alpha\colon \cF\to \cG$ be a morphism of constructible sheaves such
that $\alpha_{(\bar x,\bar t)}$ is an isomorphism for all $(\bar x,\bar t)$
in $\cP$. There exist partitions $X=\bigcup_i X_i$ and $S=\bigcup_j S_j$
into locally closed subsets such that the restrictions of $\cF$ and $\cG$ to
$P_{i,j}=X_i\atimes_S S_j$ are locally constant. By Lemma \ref{l.points},
each point of $P_{i,j}$ specializes to a point of $P_{i,j}$ in $\cP$. Thus
$\alpha|_{P_{i,j}}$ is an isomorphism.
\end{proof}

\begin{theorem}\label{t.ULA}
Let $f\colon X\to S$ be a morphism of finite type of Noetherian schemes. Let
$\Lambda$ be a Noetherian commutative ring such that $m\Lambda=0$ for some
$m$ invertible on $S$ and let $L\in D^b_c(X,\Lambda)$. Assume that for every
geometric point $\bar s$ of $S$ above $s\in S$ locally closed and every
strictly local test curve $C\to S_{(\bar s)}$, $(f_C,L|_{X_C})$ is locally
acyclic. Then $(f,L)$ is universally locally acyclic.
\end{theorem}

By a \emph{strictly local test curve} we mean a finite morphism $C\to
S_{(\bar s)}$ such that $C$ is integral and one-dimensional. Note that $C$
is strictly local.

\begin{proof}
By Orgogozo's theorem (Theorem \ref{t.Org}), there exists a modification
$g\colon T\to S$ such that $R\Psi_{f_{T}} (L|_{X_{T}})$ is constructible and
commutes with base change. We claim $\Phi_{f_T} (L|_{X_{T}})=0$. Assuming
the claim, we have, for any morphism $S'\to S$ separated and of finite type,
$\Phi_{f_{S'}}(L|_{X_{S'}})=0$ by Lemma \ref{l.descent} applied to the \v
Cech nerve of the base change of $g$ by $S'\to S$. We conclude by Remark
\ref{r.Org}.

To prove the claim, let $(\bar x,\bar t)$ be a point of $X_T\atimes_T T$ as
in Lemma \ref{l.constr}.  Let $C$ be the closure of the image of $\bar t$ in
$T_{(f_T(\bar x))}$ (equipped with the reduced scheme structure).  It
suffices to show $\Phi_{f_{C}}(L|_{X_{C}})=0$. Let $\bar s=g(f_T(\bar x))\to
S$. If the specialization $g(\bar t)\rightsquigarrow \bar s$ is an
isomorphism, this follows from Deligne's theorem \cite[Th.\ finitude,
Corollaire 2.16]{SGA4d} that $(f_{\bar s},L|_{X_{\bar s}})$ is universally
locally acyclic. Otherwise $C\to S_{(\bar s)}$ is a strictly local test
curve and $\Phi_{f_{C}}(L|_{X_{C}})=0$ by assumption. Indeed, $C$ is clearly
integral and one-dimensional. Moreover, $T_{(f_T(\bar x))}$ is the strict
localization of $T_{(\bar s)}\colonequals T\times_S S_{(\bar s)}$ at a
closed point of the fiber $T_{\bar s}$. Thus $C$ is the limit of a system of
affine schemes quasi-finite over $T_{(\bar s)}$ with \'etale transition
maps, hence a strict localization of $U$ at a closed point $u\in h^{-1}(\bar
s)$, with $h\colon U\to S_{(\bar s)}$ of finite type. The localization
$\Spec(\cO_{U,u})$ is irreducible of dimension $1$ and the image in
$S_{(\bar s)}$ is not a point. It follows that $u$ is an isolated point
$h^{-1}(\bar s)$. By Chevalley's semicontinuity theorem, $h$ is quasi-finite
at $u$. Thus $C$ is finite over $S_{(\bar s)}$.
\end{proof}

Since local acyclicity is stable under quasi-finite base change, we
immediately deduce the following.

\begin{cor}\label{c.la}
Let $f$ and $\Lambda$ be as in Theorem \ref{t.ULA}. Let $L\in
D^b_c(X,\Lambda)$ such that $(f,L)$ is locally acyclic. Then $(f,L)$ is
universally locally acyclic.
\end{cor}

In the case $L=\Lambda$, this answers a question of M.~Artin \cite[XV
Remarque 1.8 b)]{SGA4} (local variant) under the above assumptions. Compare
with Example \ref{e.Psigood} (2).

\begin{remark}\leavevmode\label{r.ULA}
\begin{enumerate}
\item We may replace the assumption in Theorem \ref{t.ULA} by the
    existence, for every strictly local test curve $C\to S_{(\bar s)}$
    with $s\in S$ locally closed, of a surjective morphism of integral
    schemes $g\colon C'\to C$ such that $(f_{C'},L|_{X_{C'}})$ is locally
    acyclic. Indeed, $(f_{C},L|_{X_{C}})$ is $\Psi$-good (by Example
    \ref{e.Psigood} (2) or by the proof of Theorem \ref{t.ULA}) and
    $(g_X\atimes_g g)^*$ is conservative. For $S$ universally Japanese,
    taking $C'$ to be the normalization of $C$, it thus suffices in
    Theorem \ref{t.ULA} to assume the local acyclicity of
    $(f_{C},L|_{X_{C}})$ for strictly local test curves $C\to S_{(\bar
    s)}$ with the additional hypothesis that $C$ is regular. Note that
    such a $C$ is the spectrum of a strictly Henselian discrete valuation
    ring.
\item Assume that $S$ is of finite type over a field or over $\Spec(\Z)$.
    Then it suffices in Theorem \ref{t.ULA} to assume the local acyclicity
    of $(f_{T},L|_{X_{T}})$ for $T\to S$ quasi-finite with $T$ regular of
    dimension $1$ (cf.\ \cite[B.6 5)]{BG}). Indeed, any $C\to S$ as above
    with $C$ regular factorizes through some $T\to S$.
\end{enumerate}
\end{remark}

Gabber's proof of the case $L\in D_\cft$ of Corollary \ref{c.DLA} relies on
Remark \eqref{r.ULA} (1) and the following consequence of the absolute
purity theorem, which is a variant of \cite[Corollary 8.10]{Saito}.

\begin{theorem}\label{t.Gysin}
Let $g\colon T\to S$ be an immersion of Noetherian regular schemes, of
codimension~$c$. Let $\Lambda$ be a Noetherian commutative ring with
$m\Lambda=0$ for some $m$ invertible on $S$. Let $f\colon X\to S$ be a
morphism of schemes and let $L\in D^b(X,\Lambda)$ of finite tor-amplitude
such that $(f,L)$ is strongly locally acyclic (Definition \ref{d.slc}). Then
the Gysin map $a\colon g_X^*L(-c)[-2c]\to Rg_X^!L$ is an isomorphism.
\end{theorem}

For $f=\id_S$ and $L=\Lambda$, we recover the absolute purity theorem.

\begin{proof}
We may assume that $g$ is a closed immersion. Let $j\colon U\to S$ denote
the complement of~$g$. We have a morphism of distinguished triangles
\[\xymatrix{L\otimes^L f^*g_* Rg^! \Lambda\ar[r]\ar[d]_b & L\ar@{=}[d]\ar[r] & L\otimes^L
f^*Rj_*\Lambda\ar[d]^c\ar[r] & \\
g_{X*}Rg_X^! L\ar[r] & L\ar[r] & Rj_{X*}j_X^* L\ar[r] &,}
\]
where $c$ is an isomorphism by \cite[Th.\ finitude, App., Proposition
2.10]{SGA4d} (see \cite[Lemma 7.6.7 (b)]{Fu} for a more detailed proof). It
follows that $b$ is an isomorphism. We have a commutative square
\[\xymatrix{L\otimes^L f^*g_* \Lambda (-c)[-2c]\ar[d]_\simeq\ar[r]^\sim & L\otimes^L f^*g_* Rg^!\Lambda\ar[d]^b_\simeq\\
g_{X*}g_X^*L(-c)[-2c]\ar[r]^{g_{X*}a} & g_{X*}Rg^! L,}
\]
where the upper horizontal isomorphism is the absolute purity theorem. Thus
$g_{X*}a$ is an isomorphism, and so is $a$.
\end{proof}

\begin{cor}\label{c.Gysin}
Let $g\colon T\to S$ be a morphism of finite type of Noetherian regular
schemes admitting ample invertible sheaves. Let $(f,L)$ be as in Theorem
\ref{t.Gysin} with $(f,L)$ universally strongly locally acyclic. Then the
Gysin map $g_X^*L(d)[2d]\to Rg_X^!L$ is an isomorphism, where $d$ is the
virtual relative dimension of $g$.
\end{cor}

\begin{proof}
We factorize $g$ into an immersion followed by a smooth morphism. The case
of a smooth morphism is obvious. The case of an immersion follows from
Theorem \ref{t.Gysin}.
\end{proof}

Here is Gabber's proof of the case $L\in D_\cft$ of Corollary \ref{c.DLA}.
We may assume $K_X=a^!\Lambda_S$. By Remark \ref{r.ULA} (1), it suffices to
show that $(a_C,g_X^*(D_XL)_{(\bar s)})$ is locally acyclic for strictly
local test curves $g\colon C\to S_{(\bar s)}$ with $C$ the spectrum of a
strictly Henselian discrete valuation ring. By Corollary \ref{c.la} and
Lemma \ref{l.slc} (1), $(f,L)$ is universally strongly locally acyclic. Thus
$g_X^*(D_XL)_{(\bar s)}\simeq D_{X_C}Rg_X^! L_{(\bar s)}\simeq D_{X_C}(g_X^*
L_{(\bar s)}(d)[2d])$ by Corollary \ref{c.Gysin}. Thus it suffices to show
that $D_{X_C}$ preserves local acyclicity over $C$. Changing notation, it
suffices to show $\Phi_f D_X L=0$ for $S$ the spectrum of a strictly
Henselian discrete valuation ring. This follows from Beilinson's theorem
(Corollary \ref{c.Beilinson}), or from Theorem \ref{t.PsiS}: $\Psi_f^{s}D_X
L\simeq D_{X_{s}\ttimes S}\Psi_f^{s} L\simeq D_{X_{s}\ttimes
S}p_1^*(L|_{X_{s}})\simeq p_1^*D_{X_{s}}(L|_{X_{s}})$, where the last
isomorphism is K\"unneth formula (Corollary \ref{c.DKunn}).

\end{document}